\documentclass[a4paper,11pt]{amsart}

\usepackage[english]{babel}

\usepackage{amsmath,amsfonts,amssymb,amsthm,mathrsfs}
\usepackage{mathtools}
\usepackage{graphicx}
\usepackage{stackrel}
\usepackage{xcolor}
\usepackage[shortlabels]{enumitem}
\usepackage{esint}
\usepackage[capitalise, nameinlink, noabbrev]{cleveref} %Leave it as the last package

\numberwithin{equation}{section}
\newtheorem{theorem}{Theorem}[section]
\newtheorem{lemma}[theorem]{Lemma}
\newtheorem{proposition}[theorem]{Proposition}

\theoremstyle{definition}
\newtheorem{definition}[theorem]{Definition}
\newtheorem{remark}[theorem]{Remark}

\newcommand{\R}{\mathbb{R}}

\newcommand{\cc}{\mathbf{C}}
\newcommand{\res}{\mathop{\hbox{\vrule height 6pt width .5pt depth 0pt \vrule height .5pt width 6pt depth 0pt}}\nolimits}

\newcommand{\Fc}{\mathcal{F}}

\renewcommand{\H}{\mathcal{H}}
\newcommand{\e}{\varepsilon}
\newcommand{\de}{\delta}
\newcommand{\pa}{\partial}
\newcommand{\beq}{\begin{equation}}
\newcommand{\eeq}{\end{equation}}
\newcommand{\Om}{\Omega}

\newcommand{\wtos}{\stackrel{*}{\rightharpoonup}}

\newcommand{\medint}{-\kern  -,425cm\int}

\newcommand{\eps}{\varepsilon}

 \DeclarePairedDelimiter\abs{\lvert}{\rvert}      %This needs mathtools  the syntax is \abs* to make it adaptable or  \abs[\big],                 
\DeclareMathOperator{\Id}{Id}

\DeclareMathOperator{\Div}{div}

\DeclareMathOperator{\dist}{dist}

\title{Regularity of capillarity droplets with obstacle}

\author[G. De Philippis]{Guido De Philippis}
\address{(G. De Philippis) Courant Institute of Mathematical Sciences, New York University,
 New York, NY, USA}
\email{guido@cims.nyu.edu}

\author[N. Fusco]{Nicola Fusco}
\address{(N. Fusco) Dipartimento di Matematica e Applicazioni, 
Universit\`{a} degli Studi di Napoli ``Federico II'', Napoli, Italy}
\email{n.fusco@unina.it}

\author[M. Morini]{Massimiliano Morini}
\address{(M. Morini) Dipartimento di Scienze Matematiche Fisiche e Informatiche, Universit\`{a} degli Studi di Parma, Parma, Italy}
\email{massimiliano.morini@unipr.it}

\thanks{{\bf Acknowledgments.}
	The work of G.D.P is  partially supported by the NSF grant DMS 2055686 and by the Simons Foundation. The work of N.F. is supported by  the PRIN project 2017TEXA3H. M.M. is  supported by the University of Parma FIL grant ``Regularity, non-linear potential theory and related topics''. Part of this work has been carried out during various visits to the authors institutions, whose support is greatly acknowledged.}

\begin{document}
\begin{abstract} 
In this paper we study the regularity properties of $\Lambda$-minimizers of the capillarity energy in a half space with the wet part constrained to be confined inside a given planar region. Applications to a model for nanowire growth are also provided.
\end{abstract}

\maketitle

\tableofcontents

\section{Introduction}

 Capillarity phenomena occur whenever two or more fluids are situated adjacent each other and do not mix. The separating interface is usually refereed to as a capillary surface. Since the pioneering works by  Young and  Laplace (see  Finn's book \cite{Finnbook} for an historical introduction) these phenomena have been the subject of countless studies in the mathematical and interdisciplinary literature. A modern treatment of the problem is based on Gauss' idea of describing  equilibrium configurations as critical points or (local) minimizers of a free energy accounting for the area of the surface separating the fluids and the surrounding media, for the area of the {\it wet region} due to the  adhesion between the fluids and the walls of the container, and for the possible presence of external fields acting on the system (such as gravity). The existence of minimizing configurations is easily obtained in the framework of sets of finite perimeter. While the regularity inside the container of such configurations reduces to the more classical study of minimal surfaces, a more specific question is related to the regularity of the contact line between the container and the fluid, see \cite{Taylor77} for the physically relevant three dimensional case and \cite{De-PhilippisMaggi15} for a wide extension to more general anisotropic energies and to higher dimensions.
 
 In this paper we study the regularity of the contact line for a capillarity problem where the `container' is a half space $H$ and the wet region  is constrained to be confined inside a given planar domain  $O\subset\pa H$. The motivation for this problem comes from the study of a mathematical model of vapor-liquid-solid
(VLS) nanowire growth considered in the physical literature. 
We recall that during VLS growth a
nanoscale liquid drop of catalyst deposited on the flat tip of the solid cylindrical 
nanowire feeds its vertical growth. In the experiments it is observed that the sharp edge of the nanowire produces a pinning effect and forces the wet part to remain confined inside the top face of the cylinder and the  liquid drop to be contained in the upper half space, see  Figure~\ref{figure1}. 
\begin{figure}\label{figure1}
\includegraphics[scale=0.71]{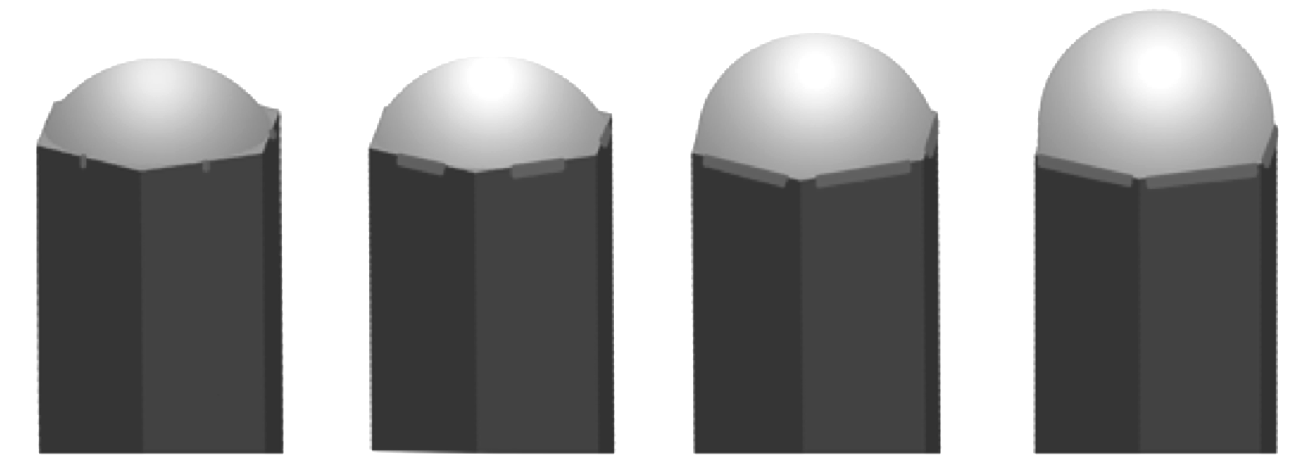}
 %\put(-52,80){$\theta$}\put(-58,56){$\alpha$}\put(-105,73){${}^0$}\put(-115,98){${}^{C}$}\put(-90,40){{\Large$\cc$}}\put(-65,145){{\Large$S_\theta$}}\put(-130,67){$\gamma$}
 \caption{As the volume of the liquid drop increases the drop wets larger regions of the nanowire edge, but remains pinned at the top (reproduced from P. Krogstrup et al. \cite{Curiotto})}
\end{figure}

\subsection{Setting of the problem and main results}

Let us start by fixing some notations: We will work in \(\R^{N}\) and we set
$$
H=\{x \in \R^{N}: x_1> 0\}.
$$
Given \(\sigma\in (-1,1)\) we consider for a set of finite perimeter \(E\subset H\) and an open set \(A\) (not necessarily contained in \(H\)) the {\it capillarity energy}
\[
\begin{split}
 \Fc_\sigma(E;A)&=P(E; H\cap A)+\sigma P(E;\partial H\cap A)
 \\
 &=\H^{N-1}(\partial^* E\cap H \cap A)+\sigma \H^{N-1}(\partial^* E\cap \partial H\cap A)
 \end{split}
\]
where \(\partial^* E\) is the reduced boundary of \(E\), $P(E;G)$ is the perimeter of $E$ in $G$ (see the definitions at the beginning of Section~\ref{sec:densityestimates}) and $\H^{N-1}$ stands for the $(N-1)$-dimensional Hausdorff measure. In case \(A=\R^N\) we will simply write \( \Fc_\sigma(E)\).

We aim to impose a constraint on the {\it wet region}  \(\partial^* E\cap \partial H\) of \(E\). To this end we consider a relatively open set \(O\subset \partial H\) and  we denote by
\[
\mathcal C_{O}=\bigl\{E\subset H \text{ sets of locally finite perimeter  such that \(\partial^* E \cap \partial H\subset \overline O\)}\bigr\}
\]
the class of admissible competitors. We aim at studying the regularity properties of (local) minimizers of the variational problem
\[
\min\{ \Fc_\sigma(E): E\in \mathcal C_O, |E|=m\}.
\]
Note that classical variational arguments imply that if one assumes that \(M:=\overline{ \partial E\cap H}\) is a smooth manifold with boundary, then any  minimizer satisfies the following  Euler-Lagrange conditions:
\begin{enumerate}[label=(\roman*)]
\item (\emph{Constant mean curvature}) There exists \(\lambda>0\) such that  \(H_{M}=\lambda \) in \(M\cap H\), where $H_M$ is the sum of the principle curvatures of $M$ and more precisely coincides with the tangential divergence of the outer unit normal field $\nu_{E}$ to the boundary of $E$;
\item (\emph{Young's inequality}) \(\nu_{E}\cdot \nu_{H}\ge \sigma\) on   \(M\cap \partial H\); 
\item (\emph{Young's law inside $O$}) \(\nu_{E}\cdot \nu_{H}= \sigma\) on  \((M\cap \partial H)\setminus \partial_{\partial H} O\)\footnote{Here and in the sequel  for a set \(U\subset \partial H\) we will denote by \(\partial_{\partial H} U\) its relative boundary in \(\partial H\)}. 
\end{enumerate}

Note that   (iii) above is the classical Young's law which holds true outside the {\it thin   obstacle} $\partial_{\partial H} O$, while  (ii) is a global inequality which should hold true on the whole free boundary   \(M\cap \partial H\).

As it is customary in Geometric Measure Theory, we will remove volume type constraints and we will deal with some perturbed minimality conditions. Note that  this allows to treat several problems at once  (volume constraints, potential terms, etc.).
Concerning the regularity of the obstacle we introduce the following class of open subsets of $\partial H$ satisfying a uniform inner and outer ball condition at every point of the (relative) boundary. More precisely, we give the following definition.
\begin{definition}\label{def: BR} Let $R\in (0,+\infty)$.  We denote by 
$\mathcal{B}_{R}$ the family of all relatively open subsets $O$ of   $\partial H$    such that   for every $x\in \partial_{\partial H} O$  there exist two $(N-1)$-dimensional relatively open balls $B'$, $B''\subset \partial H$ of radius $R$ such that $B'\subset O$,  $B''\subset \partial H\setminus O$, and $\partial_{\partial H} B'\cap  \partial_{\partial H} B''=\{x\}$.

Morever, we set $\mathcal{B}_{\infty}=\cap_{R>0}\mathcal{B}_{R}$.  Note that $\mathcal{B}_{\infty}$ is made up of   all relatively open half-spaces of $\partial H$ and of $\pa H$ itself.
\end{definition}

\begin{remark}\label{rem:BR}
Note that if $O\in \mathcal{B}_{R}$, then $O$ is  of class $C^{1,1}$ with  principal curvatures bounded by $\frac{1}{R}$, see \cite{MoMo00, Dalphin}. Therefore, if $O_h$ is a sequence in $\mathcal{B}_{R_h}$ with $R_h\to R\in(0,+\infty]$, then there exists  a (not relabelled) subsequence such that $\overline{O_h}\to \overline O$ in the Kuratowski sense, where $O\in\mathcal{B}_{R}$. Moreover, for all $\alpha\in(0,1)$ $O_h\to O$ in $C^{1,\alpha}_{loc}$ in the following sense:  given $x\in \partial_{\partial H} O$, there exist a $(N-1)$-dimensional ball $B$ centered at $x$, $\psi_h,\psi\in C^{1,1}(\R^{N-2})$ such that, up to a rotation in $\pa H$,  $\partial_{\partial H} O_h\cap B$ coincides with the graph of $\psi_h $ in $B$, $\partial_{\partial H} O\cap B$ coincides with the graph of $\psi $ in $B$, and $\psi_h\to\psi$ locally in $C^{1,\alpha}(\R^{N-2})$.
\end{remark}

\begin{definition}\label{def:lambdamin}
Given \(\Lambda\geq 0\), \(r_0>0\),  and $O\subset\pa H$ a relatively open set, we say that \(E\in \mathcal C_O\) is a \((\Lambda, r_0)\)-minimizer of \(\Fc_{\sigma}\)  with obstacle \(O\) if
\begin{multline*}
 \Fc_\sigma(E;B_{r_0}(x_0))\le  \Fc_\sigma(F;B_{r_0}(x_0))+\Lambda|E\Delta F|
 \\
 \text{for all \(F\in \mathcal C_O\) such that  \(E\Delta F\Subset B_{r_0}(x_0)\)}\,.
\end{multline*}
When $E$ is a $(\Lambda,r_0)$-minimizer (with obstacle $O$) for every $r_0>0$ we will simply say that $E\subset H$ is a  \((\Lambda,+\infty)\)-minimizer (with obstacle $O$) or simply a \(\Lambda\)-minimizer (with obstacle $O$) of $\Fc_{\sigma}$.
\end{definition}

We are in a position to state the main result of the paper, which establishes full regularity in three dimensions and partial regularity in any dimension, with an estimate on the Hausdorff dimension $\text{dim}_{\mathcal H}$ of the singular set.
\begin{theorem}\label{th:reg}
Let $O\subset\pa H$ be a relatively open set of class $C^{1,1}$ and let $E\in\mathcal C_O$ a $(\Lambda,r_0)$-minimizer with obstacle $O$ of $\mathcal F_\sigma$. Then, the following conclusions hold true:
\begin{itemize}
\item[(i)] if $N=3$, then $\overline{\pa E\cap H}$ is a surface with boundary of class $C^{1,\tau}$ for all $\tau\in(0,\frac12)$.
\item[(ii)] if $N\geq4$, there exists a closed set $\Sigma\subset \overline{\pa E\cap H}\cap\pa H$, with   ${\rm dim}_{\mathcal H}(\Sigma)\leq N-4$, such that $\overline{\pa E\cap H}\setminus\Sigma$ is (locally) a $C^{1,\tau}$ hypersurface with boundary for all $\tau\in(0,\frac12)$.
\end{itemize}
Moreover, 
\begin{align*}
\nu_{E}\cdot \nu_{H}\ge \sigma \quad &\text{on} \quad   M\cap \partial H; \\
\nu_{E}\cdot \nu_{H}= \sigma \quad &\text{on} \quad   (M\cap \partial H)\setminus \partial_{\partial H} O,
\end{align*}
where we set  $M:= \overline{\pa E\cap H}\setminus\Sigma$.
\end{theorem}
The main new point of the above result is the regularity at the  points of the free-boundary  $\overline{\pa E\cap H}\cap\pa H$ lying on the  thin obstacle  $\pa_{\pa H}O$. Indeed the full regularity at the boundary points $\overline{\pa E\cap H}\cap O$ for $N=3$ follows  from the classical result of \cite{Taylor77}, while the partial regularity for $N\geq4$ from the more recent paper \cite{De-PhilippisMaggi15}. The idea the proof of Theorem~\ref{th:reg} is based on the following dichotomy, inspired by the work of Fern\'andez-Real and Serra,~\cite{Fernandez-RealSerra20}, where a thin obstacle problem for the area functional is studied,  see also the work of Focardi and Spadaro  \cite{focardi-spadaro20} for the non parametric case.  If in a $r$-neighborhood of $x\in\overline{\pa E\cap H}\cap\pa_{\pa H}O$ the boundary $\pa E\cap H$ is contained in a sufficiently thin strip with a slope  sufficiently larger than the slope $\alpha(\sigma)$ given by the Young's law, then we show by a barrier argument that at scale $r/2$ the free boundary fully coincides with the thin obstacle $\pa_{\pa H}O$.  If instead the slope of the strip is sufficiently close to $\alpha(\sigma)$, then a  linearization procedure allows us to reduce to a Signorini type problem and to deduce from the  decay estimates available for the latter   a flatness improvement result at a smaller scale. Iterating this dichotomy argument leads to a  {\it boundary $\e$-regularity} result, see Theorem~\ref{th:epsreg}, which, combined with a suitable monotonicity formula,  finally yields  the proof of Theorem~\ref{th:reg}. Since the barrier argument requires closeness in the ``uniform norm'', in order to make the dichotomy effective, we need to establish a decay of the excess in the same norm. For this step we follow the ideas introduced by Savin in \cite{Savin}  to deal with interior regularity of minimal surfaces and we rely on a partial Harnack inequality which is obtained by a nontrivial extension to the boundary of these ideas.
% Note also that the connection  was already known in the context of the  thin obstacle problem for nonparametric minimal surfaces, see for instance \cite{focardi-spadaro20}. Such a 
 The connection  with the Signorini problem also explains why we cannot expect better regularity than $C^{1,1/2}$, see for instance \cite{Richardson78}.

\subsection{Applications to a model for nanowire growth}\label{nanosec}

\begin{figure}\label{figure2}
\vskip -1cm
\includegraphics[scale=0.25]{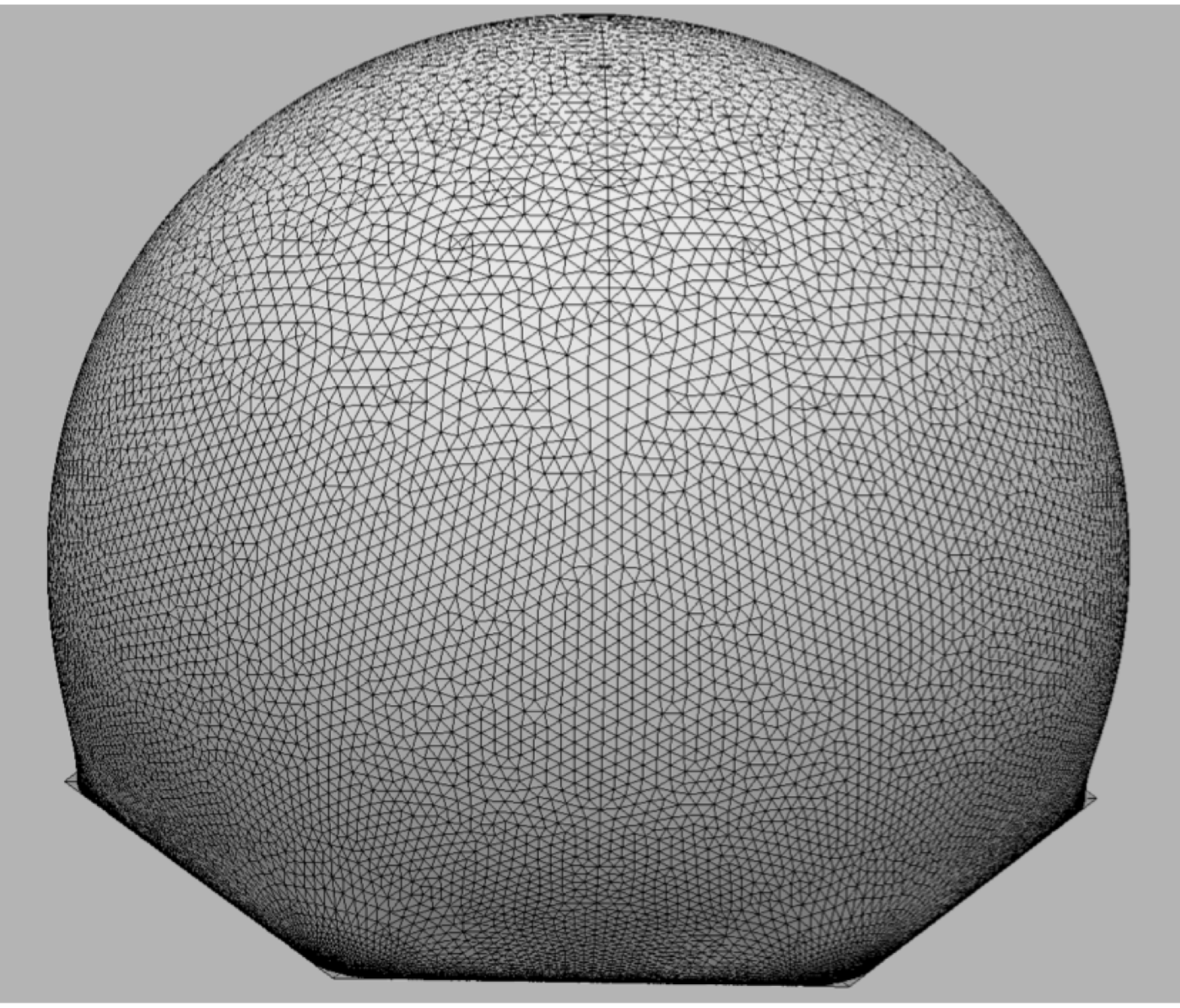}
\includegraphics[scale=0.39]{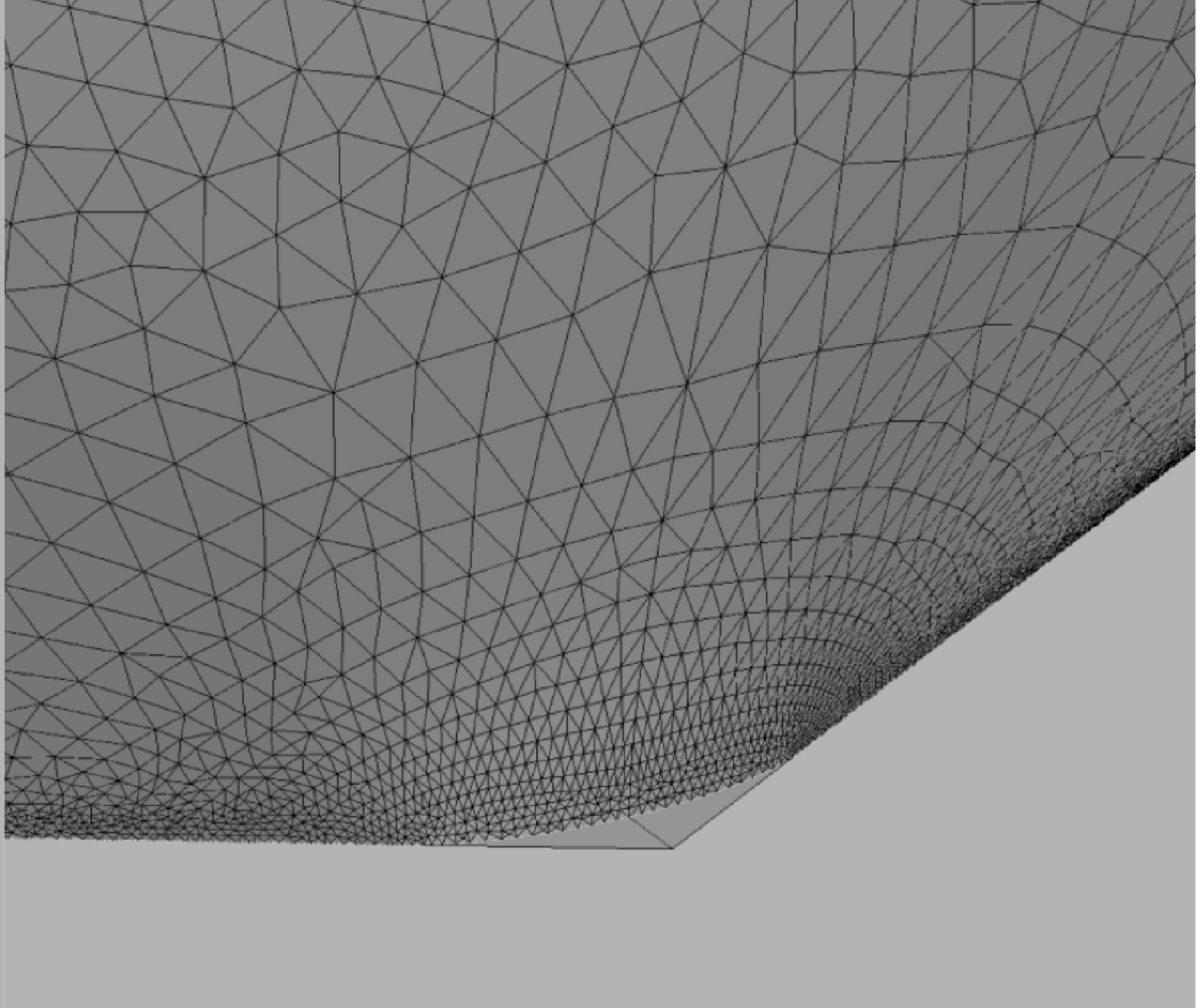}
 %\put(-52,80){$\theta$}\put(-58,56){$\alpha$}\put(-105,73){${}^0$}\put(-115,98){${}^{C}$}\put(-90,40){{\Large$\cc$}}\put(-65,145){{\Large$S_\theta$}}\put(-130,67){$\gamma$}
 \caption{On the left a drop sitting on the tip of a nanowire with hexagonal section. The picture on the right shows the behaviour of the droplet near a corner (courtesy of B. Spencer)}
\end{figure}

As anticipated at the beginning of this introduction, we  conclude by applying the above regularity theory to a model of vapor-liquid-solid
(VLS) nanowire growth considered in the physical literature and studied in \cite{FFLM22}. 
 
Following the work of several authors, see the references in \cite{FFLM22}, we consider a continuum framework for
nanowire VLS growth. We model the nanowire as a semi-infinite cylinder  $\cc=\omega\times(-\infty,0]$, where $\omega\subset\R^2$ is a bounded sufficiently regular domain,   and the liquid drop as a set
$E\subset\mathbb{R}^{3}\setminus\cc$ of finite perimeter. Typically the observed nanowires have either a regular or a polygonal section. 
Given $\sigma\in(-1,1)$  we consider the free energy%
\[
J_{\sigma,\omega}(E):=\mathcal{H}^{2}(\partial^{\ast}E\setminus\cc)+\sigma\mathcal{H}^{2}(\partial^{\ast}E\cap\cc)\,.
\]
The shape of the liquid drop is then described by (local) minimizers of $J_{\sigma,\omega}$ under a volume constraint. To this aim we say that $E\subset\R^3\setminus\cc$ is a {\it volume constrained local minimizer of $J_{\sigma,\omega}$} if there exists $\e>0$ such that $J_{\sigma,\omega}(E)\leq J_{\sigma,\omega}(F)$ for all $F\subset\R^3\setminus\cc$ with $|F|=|E|$ and $|E\Delta F|<\e$.
 Here we study  local minimizing configurations corresponding to liquid drops sitting on the top $\cc_{top}:=\omega\times\{0\}$ of the cylinder and contained in the upper half space $\mathscr H:=\{x_3>0\}$.  

In the case where $\omega$ is a $C^{1,1}$  domain we prove the following regularity result. 
\begin{theorem}\label{cinqueuno}
 Let $\omega\subset\R^2$ be a bounded domain of class $C^{1,1}$ and let $E\subset\mathscr H$ be a volume constrained local minimizer of $J_{\sigma,\omega}$. Then $\overline{\pa E\cap\mathscr H}$ is a surface with boundary of class $C^{1,\tau}$ for all $\tau\in(0,\frac12)$.
\end{theorem}

 In the experiments it is observed that for nanowires with a polygonal section, see for instance \cite{Curiotto}, the liquid drop never wets the corners of the polygon as illstrated in Figure~\ref{figure2}. Here we give a rigorous proof of this fact when $\sigma<0$ and, for a general $\sigma$, under the additional assumption that  the contact between the liquid drop and   $\gamma:=\pa\omega\times\{0\}$ is nontangential (see Definition~\ref{nontang}). 

\begin{theorem}\label{cinquedue}
Let $\omega\subset\R^2$ be a convex polygon and let $E\subset\mathscr H$ be a volume constrained local minimizer of $J_{\sigma,\omega}$. If $\sigma<0$, then $\overline{\pa E\cap\mathscr H}$ is a surface with boundary of class $C^{1,\tau}$ for all $\tau\in(0,\frac12)$ and the contact line $\overline{\pa E\cap\mathscr H}\cap\cc_{top}$ does not contain any vertex of the polygon. If $\sigma\geq0$ the same conclusion holds provided that $E$ has a nontangential contact at all points of $\overline{\pa E\cap\mathscr H}\cap\gamma$.
\end{theorem}

\section{Density estimates and compactness}\label{sec:densityestimates}

Given a Lebesgue measurable set $E\subset\R^N$ , we say that $E$ is of locally finite perimeter if there exists a $\R^N$-valued Radon measure $\mu_E$ such that
$$
\int_E\nabla\varphi\,dx=\int_{\R^N}\varphi\,d\mu_E\qquad \text{for all $\varphi\in C^1_c(\R^N;\R^N)\,.$}
$$
If $G\subset \R^N$ is a Borel set we denote by $P(E; G)= |\mu_E|(G)$ the perimeter of $E$ in $G$. For all relevant definitions and properties of sets of finite perimeter we shall refer to the book \cite{Maggi12}. In the following we denote by $\partial^*E$ the reduced boundary of a set of finite perimeter and by 
$\pa^eE$ the {\em essential boundary}, which  is defined as
$$
\pa^eE:=\R^N\setminus(E^{(0)}\cup E^{(1)})\,,
$$
where $E^{(0)}$ and $E^{(1)}$ are the sets of points where the density of $E$ is $0$ and $1$, respectively. Since the perimeter measure coincides with the $\H^{N-1}$ measure restricted to the reduced boundary $\pa^*E$, we will often write
$\H^{N-1}(\pa^*E\cap \Om)$ instead of $P(E;\Om)$. Note that  if $E\subset H$ the characteristic function of $\partial^* E\cap \partial H$ is the trace of \(1_E\) intended as a \(BV\) function (for the definition of trace see \cite[Th. 3.87]{AmbrosioFuscoPallara00}).

In the following, when dealing with a set of locally finite perimeter $E$, we will always assume that $E$ coincides with a precise representative that satisfies the property $\pa E=\overline{\pa^*E}$, see \cite[Remark 16.11]{Maggi12}. A possible choice is given by $E^{(1)}$, for which one may  check that
\beq\label{Euno}
\pa E^{(1)}=\overline{\pa^*E}\,.
\eeq

\medskip
Given a set \(E\) we denote by \(E_{x,r}\) the set  \(E_{x,r}:=(E-x)/r\). A ball  centered in \(x\) and of radius \(r\) is denoted by \(B_r(x)\), if the center is \(0\) we simply set \(B_r\). 
 Note that if \(E\) is \((\Lambda, r_0)\)-minimizer of \(\Fc_{\sigma}\)   with obstacle \(O\in \mathcal B_R\) and \(x\in \partial H\), \(r>0\) then 
\(E_{x,r}\) is a \((\Lambda r, r_0/r)\)  minimizer of \(\Fc_{\sigma}\)  with obstacle \(O_{x,r}\in \mathcal B_{\frac Rr}\). Hence we can assume without loss of generality that \(E\) is \((\Lambda, 1)\)-minimizer with obstacle $O\in \mathcal B_R$, where
\begin{equation}\label{e:basicassumption}
 \Lambda +\frac 1R\leq \boldsymbol{c_0}\leq 1\,,
\end{equation}
with \(\boldsymbol{c_0}\) is a small constant, depending on $\sigma$ and $N$ (to be chosen later). 
%We will denote by \(C\) a generic  large constant (in particular \(\ge 1\)) and by \(c\) a generic small constant (in particular \(\le 1\)).  
We will say, with a slight abuse of language, that a constant is \emph{universal} if it only depends on $\Lambda$,   \(\sigma\)  and on the dimension. 

%For two positive quantities \(X\) and \(Y\),  we will sometimes write \(X\lesssim Y\) if there exists a universal constant \(C\) such that \(X\le C Y\) and we write \(X\gtrsim Y\) if \(Y\lesssim X\).\mnote{G: A me questa notazione piace, ma se preferite la togliamo...}.

We start by observing that the constraint on the wet part can be replaced by a suitable penalization. To this end we introduce the following functional  
$$
\mathcal{F}^{O}_\sigma(F;A):=P(F;  A\setminus \overline O)+\sigma P(F;O\cap A)\,,
$$
where $A\subset\R^N$ is an open set and $F\subset H$ is a set of finite perimeter. If $A=\R^N$ we simply write $\mathcal{F}^{O}_\sigma(F)$.
We also denote by $\mathbf{C}_O$ the semi-infinite cylinder constructed over $O$, that is,
$$
\mathbf{C}_O:=\{x=(x_1,\dots,x_N)\in\R^N: x_1>0 \,\,\text{and}\,\, (x_2,\dots,x_N)\in O\}\,.
$$
\begin{lemma}\label{lm:osta}
Assume that $E\subset H$ is a $(\Lambda, r_0)$-minimizer for $\mathcal F_\sigma$ in the sense of Definition~\ref{def:lambdamin}.
Then, $E$ is a $(\Lambda, r_0)$-minimizer for $\mathcal F^O_\sigma$ without obstacle, that is,
$$
 \Fc^O_\sigma(E;B_{r_0}(x_0))\le  \Fc^O_\sigma(F;B_{r_0}(x_0))+\Lambda|E\Delta F| \text{ for all \(F\subset H\) such that   \(E\Delta F\Subset B_{r_0}(x_0)\)}\,.
$$
\end{lemma}
\begin{proof}
Since the argument is local we may assume without loss of generality that $E$ is a set of finite perimeter.

Let $F$ be as in the statement and let $B$ be an open  ball of radius  $r_0$ such that $E\Delta F\Subset B$ and 
$\H^{N-1}(\partial^*F\cap \partial B)=0$, and set $B^+=B\cap H$. For every $\eps>0$ (sufficiently small) we set 
$$
F_\eps:=[(F\cap\{x_1>\eps \}\cap B^+)\setminus \mathbf{C}_O)]\cup (F\cap \mathbf{C}_O)\cup (F\setminus B^+)\,.
$$
Note that $F_\eps$ is an admissible competitor in the sense of Definition~\ref{def:lambdamin} and thus
\begin{equation}\label{e:osta1}
\Fc_\sigma(E)\leq \Fc_\sigma(F_\eps)+\Lambda|F_\eps\Delta E|\,.
\end{equation}
Note that for a.e. $\e$
\begin{align*}
\Fc_\sigma(F_\eps)\leq& P(F; (B^+\cap \{x_1>\eps\})\setminus \overline{\mathbf{C}}_O)
+\H^{N-1}((F^{(1)}\cap B^+\cap\{x_1=\eps\})\setminus \overline{\mathbf{C}}_O)\\
&+ P(F; B^+\cap \overline{\mathbf{C}}_O)+P(F; H\setminus B^+)
+\sigma P(F; O)\\
&+ \H^{N-1}(F^{(1)}\cap B^+\cap\partial \mathbf{C}_O\cap \{x_1\leq \eps\})+
\H^{N-1}(\partial B\cap F^{(1)}\cap \{x_1\leq \eps\})\,.
\end{align*}
By the continuity of the trace with respect to strict convergence in $BV$ (see \cite[Th.~3.88]{AmbrosioFuscoPallara00}) we have that 
$$
\H^{N-1}((F^{(1)}\cap B^+\cap\{x_1=\eps_n\})\setminus \overline{\mathbf{C}}_O)
\to P(F; \partial H\setminus O)
$$
for a sequence $\eps_n\to 0^+$. Thus,
\begin{multline*}
\liminf_{\eps\to0}\Fc_\sigma(F_{\eps})\leq P(F; B^+\setminus \overline{\mathbf{C}}_O)+P(F; \partial H\setminus O)+P(F; B^+\cap \overline{\mathbf{C}}_O)\\
+P(F; H\setminus \overline B^+)
+\sigma P(F;  O)=\Fc^O_\sigma(F)\,.
\end{multline*}
Then conclusion follows recalling \eqref{e:osta1}.
\end{proof}
\begin{remark}\label{rem:lsc}
Observe that  $\mathcal F_\sigma(\cdot,A)$ is lower semicontinuous with respect to the $L^1_{loc}$ convergence in $\mathcal C_O$, see \cite[Prop. 19.1]{Maggi12}. The same proposition can be applied to show that also $\mathcal F^O_\sigma(\cdot,A)$ is lower semicontinuous with respect to $L^1_{loc}$ in $H$. To this aim it is enough to observe that it is possible to construct an open set $U$ of class $C^{1,1}$   contained in $\R^N\setminus\overline H$ such that $\pa U\cap\pa H=\overline O$ and then to apply \cite[Prop. 19.1 and Prop. 1.3]{Maggi12} to $\R^N\setminus\overline U$.
\end{remark}

We now prove some useful volume and perimeter density estimates.

\begin{proposition}\label{p:densityestimates}
Let \(\sigma \in (-1,1)\), let \(E\) be a \((\Lambda,1)\)-minimizer of \(\Fc_{\sigma}\)  with obstacle \(O\in \mathcal B_R\) 
and assume that \eqref{e:basicassumption} is in force. 
Then there are  universal positive constants \(c_1\) and \(C_1\) such that 
\begin{enumerate}[label=\textup{(\roman*)}]
\item for all $x\in H$ 
$$
P(E; B_r(x))\leq C_1 r^{N-1}\,;
$$
\item for all \(x\in \partial E\)  
\[
|E\cap B_r(x)|\geq c_1 |B_r(x)\cap H|\,;
\]
\item for all \(x\in \overline{\partial E\cap H}\)  
\[
P(E;B_r(x)\cap H)\geq c_1 r^{N-1}\,;
\]
\item
 if  $x\in \partial  (H\setminus E)$ and $B_{2r}(x)\cap\partial H\subset O$, then 
$$
|B_r(x)\setminus E|\geq c_1|B_r(x)\cap H|
$$
\end{enumerate}
for every $r\leq1$. Finally,  $E$ is equivalent to an open set, still denoted by $E$,  such that $\pa E=\pa^eE$, hence $\H^{N-1}(\pa E\setminus\pa^*E)=0$. 
\end{proposition}

\begin{proof} The proof of this proposition follows the lines of \cite[Lemma~2.8]{De-PhilippisMaggi15}; however,  some modifications are needed. Given $x\in H$ and $r<1$, we set $m(r):=|E\cap B_r(x)|$. Recall that for a.e. such $r$ we have  
$m'(r)=\H^{N-1}(E^{(1)}\cap \partial B_r(x))$ and $\H^{N-1}(\partial^*E\cap \partial B_r(x))=0$. For any such $r$ we set $F:=E\setminus B_r(x)$. Then, using Definition~\ref{def:lambdamin}, we have 
\begin{align}
P(E; B_r(x)\cap H)&\leq \H^{N-1}(\partial B_r(x)\cap E^{(1)})+\Lambda|E\cap B_r(x)|+|\sigma|P(E; \partial H\cap B_r(x))\nonumber\\
&\leq   C_1 r^{N-1} \label{e:de1}
\end{align}
for a suitable universal constant $C_1$. 
Observe now that  by an easy application of the divergence theorem we have
$$
P(E; \partial H\cap B_r(x))=P(E\cap B_r(x); \partial H)\leq P(E\cap B_r(x); H)\,.
$$
Thus, using also \eqref{e:de1}, we have
\begin{align}
P(E\cap B_r(x))&=P(E\cap B_r(x); H)+P(E\cap B_r(x); \partial H)\nonumber\\
&\leq 2 P(E\cap B_r(x); H)= 2 P(E; B_r(x)\cap H)+2m'(r)\label{e:de2}\\
&\leq 4 m'(r)+2\Lambda m(r)+2|\sigma| P(E; \partial H\cap B_r(x))\nonumber\\
&\leq 4 m'(r)+2\Lambda m(r)+2|\sigma| P(E\cap B_r(x); H)\nonumber\,.
\end{align}
Comparing the first term in the second line with the fourth line of the previous chain of inequalities we have in particular that 
$$
 P(E\cap B_r(x); H)\leq \frac{1}{1-|\sigma|}(2m'(r)+\Lambda m(r))\,.
$$
In turn, inserting the above estimate in \eqref{e:de2} and using the isoperimetric inequality we get
\begin{align*}
N\omega_N^{\frac1N}m(r)^{\frac{N-1}{N}}& \leq P(E\cap B_r(x))\leq \frac{2}{1-|\sigma|}(2m'(r)+\Lambda m(r))\\
&\leq \frac{2}{1-|\sigma|}(2m'(r)+\Lambda r\omega_N^{\frac1N}m(r)^{\frac{N-1}N})\\
&\leq 
\frac{2}{1-|\sigma|}(2m'(r)+\mathbf{c_0}\omega_N^{\frac1N}m(r)^{\frac{N-1}N})\,,
\end{align*}
where in the last inequality we used \eqref{e:basicassumption}. Now if $\frac{2\mathbf{c_0}}{1-|\sigma|}\leq 1$, then  from the previous inequality we get
 $$
 (N-1)\omega_N^{\frac1N}m(r)^{\frac{N-1}{N}}\leq \frac{4}{1-|\sigma|}m'(r)\,.
 $$
 Observe now that if in addition $x\in \partial^* E$, then $m(r)>0$ for all $r$ as above. Thus, we may divide the previous inequality by $m(r)^{\frac{N-1}{N}}$, and integrate the resulting differential inequality thus getting 
 $$
 |E\cap B_r(x)|\geq c_1 |B_r(x)\cap H|\,,
 $$
 for a suitable positive constant $c_1$ depending only on $N$ and $|\sigma|$. 
 
 To get the lower density estimate on the perimeter,  let $x\in \partial^*E\cap H$ and $r<1$. If 
 $\dist(x, \partial_{\pa H}O)>\frac{r}2$, then either  $B_\rho(x)\cap \partial H\subset \partial H\setminus\overline O$ for all $\rho\in (0, \frac{r}2)$ or $B_\rho(x)\cap \partial H\subset O$ for all $\rho\in (0, \frac{r}2)$. 
 
 In the first case, $E$ is a 
 $(\Lambda, r/2)$-minimizer of the standard perimeter in $B_{\frac{r}{2}}(x)$ under the constraint $E\subset H$. Then an easy truncation argument implies that $E$   is also an unconstrained 
 $(\Lambda, r/2)$-minimizer of the perimeter in the same ball. Then,  the lower density estimate on the perimeter follows from classical results (see \cite[Th. 21.11]{Maggi12}).
 
  In the second case,   it follows that  $H\setminus E$ is a $(\Lambda, r/2)$-minimizer of $\Fc_{-\sigma}$ in $B_{\frac{r}{2}}(x)$ and thus,   arguing as above, we get 
 $$
 |E\cap B_{\frac{r}{2}}(x)|\leq (1-c_1) |B_{\frac{r}{2}}(x)\cap H|
 $$
 and, in turn, by the relative isoperimetric inequality, setting for every $\tau\in (\frac12, 1)$ 
 $$
 \kappa(\tau):=\inf_{t\geq 0}\inf\Bigl\{\frac{P(F; B_1(te_1)\cap H)}{|F|^{\frac{N-1}{N}}}:\, F\subset B_1(te_1)\cap H,\, |F|\leq\tau |B_1(te_1)\cap H|\Bigr\}
 $$
 we obtain
$$
 P(E, B_{\frac{r}{2}}(x)\cap H)\geq \kappa(1-c_1) |E\cap B_{\frac{r}2}(x)|^{\frac{N-1}{N}}\geq  c_1 \kappa(1-c_1) |B_{\frac{r}{2}}(x)\cap H|^{\frac{N-1}{N}}\,.
$$
  Assume now that $\dist(x, \partial_{\pa H} O)\leq \frac{r}2$. Then, there exists  $y=(0, y')\in \partial_{\pa H} O$ such that  the ball $B_{\frac{r}2}(y)\cap H$ is contained in  $B_r(x)\cap H$. 
   Assume that 
 $|E\cap B_{\frac{r}2}(y)|\leq (1-\gamma)|B_{\frac{r}2}(y)\cap H|$ for some small $\gamma\in (0, \frac12)$ to be chosen later. By the relative isoperimetric inequality we get
$$
 P(E; B_{r}(x)\cap H)\geq \kappa_N\min\{|E\cap B_{r}(x)|, |(H\setminus E)\cap B_{r}(x)|\}^{\frac{N-1}{N}}.
$$
If  $\min\{|E\cap B_{r}(x)|, |(H\setminus E)\cap B_{r}(x)|\}=|E\cap B_{r}(x)|$, then (iii) follows from (ii). Otherwise
\begin{align*}
P(E; B_{r}(x)\cap H)&\geq\kappa_N|(H\setminus E)\cap B_{r}(x)|^{\frac{N-1}{N}} \nonumber \\
&\geq\kappa_N|(H\setminus E)\cap B_{\frac{r}{2}}(y)|^{\frac{N-1}{N}}\geq\kappa_N\Big(\frac{\gamma\omega_N}{2^{N+1}}\Big)^{\frac{N-1}{N}}r^{N-1}\,. \label{e:de3}
\end{align*}
 If instead 
 \begin{equation}\label{e:de3.5}
 |E\cap B_{\frac{r}2}(y)|> (1-\gamma)|B_{\frac{r}2}(y)\cap H|\,,
 \end{equation}
  then we denote by $\Pi$ the orthogonal projection on $\partial H$ of 
 $\partial^*E\cap B_{\frac{r}2}(y)\cap H$.  Set $D:= \{(0, z'):\,|z'-y'|<\frac{r}{4} \}\setminus O$ and observe that 
$$
 \H^{N-1}(D)\geq \tilde c\omega_{N-1}\Bigl(\frac{r}4\Bigr)^{N-1}\,,
$$
 for a universal constant $\tilde c>0$. Note that in the above inequality we are using the assumption that  $O\in \mathcal B_R$, with $\frac1R\leq \boldsymbol{c_0}$ (see \eqref{e:basicassumption}).
 Assume first 
  that $\H^{N-1}(\Pi\cap D)\geq \gamma \omega_{N-1}\bigl(\frac{r}4\bigr)^{N-1}$. In this case, we have
  \begin{equation}\label{e:de4}
  P(E; B_{\frac{r}2}(y)\cap H)\geq \gamma \omega_{N-1}\Bigl(\frac{r}4\Bigr)^{N-1}\,.
  \end{equation}
  If  instead $\H^{N-1}(\Pi\cap D)< \gamma \omega_{N-1}\bigl(\frac{r}4\bigr)^{N-1}$, then
  $$
  |(H\setminus E)\cap B_{\frac{r}2}(y)|\geq\frac{r}{4}\H^{N-1}(D\setminus\Pi)\geq (\tilde c-\gamma)\omega_{N-1}\Bigl(\frac{r}4\Bigr)^{N}\,.
  $$
  The last inequality contradicts \eqref{e:de3.5} if $\gamma$ is sufficiently small.  This contradiction shows that \eqref{e:de4} holds under the assumption
  \eqref{e:de3.5}. This completes the proof of (iii).
  
  Property (iv) follows from the volume estimate (ii) recalling that $H\setminus E$ is a $(\Lambda, r_0)$-minimizer for $\Fc_{-\sigma}$ in $B_{2r}(x)$.
  
  In order to show that $\pa E=\pa^eE$ it is enough to prove  that $\pa E\subset\pa^eE$. To this aim fix $x\in\pa E$.  If $x\in H$, from the density estimates (ii) and (iv) we have at once that
 $x\not\in E^{(0)}\cup E^{(1)}$, that is $x\in\pa^eE$. On the other hand, if $x\in\pa H$, then $|B(x,r)\setminus E|\geq \frac12|B(x,r)|$ and again, recalling (ii), we have that  $x\in\pa^eE$.   Hence $\H^{N-1}(\pa E\setminus\pa^*E)=\H^{N-1}(\pa^eE\setminus\pa^*E)=0$, where the last equality follows from Theorem~16.2 in \cite{Maggi12}. 

Finally, since, see \eqref{Euno}, $\pa E^{(1)}=\pa E=\pa^eE$, we have that $E^{(1)}\cap \pa E^{(1)}=\emptyset$, hence $E^{(1)}$ is an open set.
\end{proof}
We are now ready to prove the following compactness theorem for almost minimizers.
\begin{theorem}\label{th:compactness} Let $\Lambda_h\geq 0$, $r_h, R_h\in [1,+\infty]$ satisfying  
$$
\Lambda_h +\frac 1{R_h}\leq \boldsymbol{c_0}\,,
$$
and 
$$
\Lambda_h\to \Lambda_0\in [0,+\infty),\, \quad r_h\to r_0\in [1,+\infty], \quad R_h\to R_0\in [1, +\infty]\,,
$$
with $\boldsymbol{c_0}$ as in \eqref{e:basicassumption}. Let $E_h$ be a $(\Lambda_h, r_h)$-minimizer of $\mathcal F_{\sigma}$
with obstacle $O_h\in \mathcal B_{R_h}$. Then there exist a  (not relabelled) subsequence,  
a set $O\in \mathcal B_{R_0}$, and a set   $E$ of locally finite perimeter, such that $E_h\to E$ in $L^1_{loc}(\R^N)$, 
$O_h\to O$ in $C^{1,\alpha}_{loc}$  for all $\alpha\in (0,1)$, with the property that $E$ is a $(\Lambda_0, r_0)$-minimizer of $\Fc^O_\sigma$.  
Moreover, 
$$
\mu_{E_h}\wtos \mu_E\,,   \quad |\mu_{E_h}|\wtos |\mu_E|\,,
$$
as Radon measures. 
In addition, the following Kuratowski convergence type properties hold:
\[
\begin{split}
&\text{(i)\,\,for every $x\in\pa E$ there exists $x_h\in\pa E_h$ such that $x_h\to x$;} \\
& \text{(ii)\,\,if
 $x_h\in \overline{H\cap\pa E_h}$ and $x_h\to x$, then $x\in \overline{H\cap\pa E}$\,.}
 \end{split}
 \]
 Finally, if $\Lambda_0=0$ and $\partial H\setminus \overline{O}$ is connected, then either 
 $\partial E\cap(\partial H\setminus \overline{O})=\partial H\setminus \overline{O}$ or $\partial E\cap \partial H\subset\overline O$.
 In the latter case, we have that $E$ is a $(0, r_0)$-minimizer with obstacle $O$,  
$$|\mu_{E_h}|\res H\wtos |\mu_E|\res H$$
as Radon measures in $\R^N$, and 
$$
\partial E_h\cap \partial H\to \partial E\cap \partial H\quad\text{ in }L^1_{loc}(\pa H)\,.
$$
\end{theorem} 
\begin{proof}
The proof goes along the lines of the  proof of \cite[Th. 2.9]{De-PhilippisMaggi15}, with some nontrivial modifications that we explain here.
First of all the existence of a non relabelled sequence $E_h$ converging in $L^1_{loc}(\R^N)$ to some set $E$ of locally finite perimeter  and such that $\mu_{E_h}\wtos \mu_E$ as Radon measures  follows at once since the sets $E_h$ have locally equibounded perimeters, thanks to Proposition~\ref{p:densityestimates} (i). The convergence, up to a subsequence, of $O_h$ to $O\in\mathcal B_{R_0}$ in $C^{1,\alpha}_{loc}$ follows from  Remark~\ref{rem:BR}.

We may assume without loss of generality that $r_0<+\infty$. We now want to show that $E$ is a $(\Lambda_0, r_0)$-minimizer of $\mathcal F^O_\sigma$. 
To this end, let us fix a ball $B_{r_0}(x)$,  and consider a competitor $F$ for $E$ such that $E\Delta F\Subset B_{r_0}(x)$. Note that for a.e. $r<r_0$  such that $E\Delta F\Subset B_r(x)$, we have $\H^{N-1}(\partial^*E\cap \partial B_r(x))=\H^{N-1}(\partial^*E_h\cap \partial B_r(x))=0$ for all $h$ and  $\H^{N-1}((E\Delta E_h)\cap \partial B_r(x))\to 0$.   Choose any such $r$ and  set $F_h:=(F\cap B_r(x))\cup (E_h\setminus B_r(x))$. Denoting by 
$O^\delta:=\{x\in \partial H: d_O(x)< -\delta\, \textrm{sign}\sigma\}$ with $\delta>0$ and $d_O$ being the signed distance from $\partial O$ restricted to $\pa H$ and  recalling that $E_h$ is a $(\Lambda_h, r_h)$-minimizer of $\mathcal F^{O_h}_\sigma$, we have by Remark~\ref{rem:lsc}, if $\sigma\leq0$
\begin{align*}
\mathcal F^{O^\delta}_\sigma(E; B_r(x))&\leq \liminf_{h}\mathcal F^{O^\delta}_\sigma(E_h; B_r (x))\leq \liminf_{h}\mathcal F^{O_h}_\sigma(E_h; B_r(x))\\
&\leq 
\limsup_{h}\mathcal F^{O_h}_\sigma(E_h; B_{r_0}(x))\leq \lim_{h}\big[\mathcal F^{O_h}_\sigma(F_h; B_{r_0}(x))+\Lambda_h|F_h\Delta E_h|\big]\\
&=\mathcal F^{O}_\sigma(F; B_{r_0}(x))+\Lambda_0|F\Delta E| \,.
\end{align*}
where we used the fact that $O_h\subset O^\delta$ for $h$ sufficiently large. Instead, if $\sigma>0$,  
\begin{align*}
\mathcal F^{O^\delta}_\sigma(E; B_r(x))&\leq \liminf_{h}\mathcal F^{O^\delta}_\sigma(E_h; B_r (x))\\
&\leq \liminf_{h}\big(\mathcal F^{O_h}_\sigma(E_h; B_r(x))+\H^{N-1}((O_h\setminus O^\delta)\cap B_r(x))\big)\\
&\leq 
\limsup_{h}\mathcal F^{O_h}_\sigma(E_h; B_{r_0}(x))+\H^{N-1}((O\setminus O^\delta)\cap B_r(x))\\
&\leq \lim_{h}\big[\mathcal F^{O_h}_\sigma(F_h; B_{r_0}(x))+\H^{N-1}((O\setminus O^\delta)\cap B_r(x))+\Lambda_h|F_h\Delta E_h|\big]\\
&=\mathcal F^{O}_\sigma(F; B_{r_0}(x))+\H^{N-1}((O\setminus O^\delta)\cap B_r(x))+\Lambda_0|F\Delta E| \,.
\end{align*}
Letting $\delta\to 0^+$ we have in both cases that 
\begin{align*}
\mathcal F^{O}_\sigma(E; B_{r_0}(x))&\leq \liminf_{h}\mathcal F^{O_h}_\sigma(E_h; B_{r_0}(x))\leq 
\limsup_{h}\mathcal F^{O_h}_\sigma(E_h; B_{r_0}(x))\\
&\leq \mathcal F^{O}_\sigma(F; B_{r_0}(x))+\Lambda_0|F\Delta E|\,.
\end{align*}
Thus, 
we have proved  that $E$ is a $(\Lambda_0, r_0)$-minimizer of $\mathcal F^O_\sigma$. 
Choosing $F=E$ in the previous inequality, we obtain
\begin{equation}\label{e:compact1}
\mathcal F^{O_h}_\sigma(E_h; B_{r_0}(x))\to \mathcal F^{O}_\sigma(F; B_{r_0}(x))\,.
\end{equation}
Assume now that $\sigma\leq 0$ and observe that by the lower semicontinuity of perimeter we get
$$
\begin{aligned}
P(E; B_{r_0}(x)\setminus \overline{O^\delta})) 
&\leq \liminf_{h}
P(E_h; B_{r_0}(x)\setminus \overline{O^\delta})) \\
&\leq  \liminf_{h}
P(E_h; B_{r_0}(x)\setminus  \overline{O_h}))
=
\liminf_{h}P(E_h; H\cap B_{r_0}(x))\,.
\end{aligned}
$$
Hence, letting $\delta\to 0^+$ we have 
\begin{equation}\label{e:compact2}
P(E;  B_{r_0}(x)\setminus \overline O )\leq \liminf_{h}P\big(E_h; H\cap B_{r_0}(x)\big)\,.
\end{equation}
On the other hand, by similar arguments
\begin{equation}\label{e:compact2.5}
\begin{aligned}
\limsup_hP(E_h; O_h\cap & B_{r_0}(x))=
\limsup_hP(E_h;  \overline{O_h}\cap B_{r_0}(x))\\
&\leq 
P(E;  \overline{O}\cap B_{r_0}(x))=P(E;  {O}\cap B_{r_0}(x))\,.
\end{aligned}
\end{equation}
From the above inequalities and \eqref{e:compact1} and the fact that $\sigma\leq0$ we deduce that 
\begin{equation}\label{e:compact3}
P(E_h; B_{r_0}(x))\to P(E; B_{r_0}(x))\,.
\end{equation}
From this inequality we deduce that  $|\mu_{E_h}|\wtos |\mu_E|$ as Radon measures in $\R^N$.
If $\sigma>0$, we observe that
$$
\sigma P(E_h;  O_h\cap B_{r_0}(x))=\sigma \H^{N-1}(O_h\cap B_{r_0}(x))-\sigma P(H\setminus E_h;  O_h\cap B_{r_0}(x))\,.
$$
Therefore, arguing as in the proof of \eqref{e:compact2.5} and recalling \eqref{e:compact2} we conclude that \eqref{e:compact3} holds  also in this case. The properties (i) and (ii) then follow by standard arguments using also the perimeter density estimate stated in Proposition~\ref{p:densityestimates} (iii), see also Step 4 of the proof of Theorem 2.9 in \cite{Maggi12}.

Let us now prove the last part the statement, and thus assume $\Lambda_0=0$ and that $\partial H\setminus \overline{O}$ is connected. 
 Assume that there exists  $x\in \partial E\cap(\partial H\setminus \overline{O})$. Since $E$ is $(0, r_0)$-minimizer of the standard perimeter in $\R^N\setminus\overline{O}$ and $E$ lies on one side of $\pa H$, we have that $x$ is a regular point of $\pa E$. Indeed every blow-up is a minimizing cone contained in a half space and thus is a half space (see for instance \cite[Lemma 3]{dm19} for a proof of this well known fact). In turn this  implies the local regularity.  Then the strong maximum principle for the mean curvature equation implies that $\pa E$ coincides with $\pa H$ in the connected component of $\pa H\setminus\overline O$ containing $x$, that is, in the whole $\pa H\setminus\overline O$.    
 
 To conclude the proof observe now that if $\partial E\cap \partial H\subset \overline O $, then from \eqref{e:compact1}
 and \eqref{e:compact2.5}, arguing as above, we  conclude that 
 $$
 P(E_h; B_r(x)\cap H)\to P(E; B_r(x)\cap H)
 $$
 and, in turn, $|\mu_{E_h}|\res H\wtos |\mu_E|\res H$ as Radon measures in $\R^N$. Now the very last part of the statement follows from \cite[Th.~3.88]{AmbrosioFuscoPallara00}.
\end{proof}

\section{ $\eps$-regularity}\label{sec:epsreg}
This section is devoted to the proof of the $\eps$-regularity Theorem~\ref{th:epsreg} for free boundary points lying on the thin obstacle $\pa_{\pa H}O$. Such a proof will split into two subsections.

\subsection{Partial Harnack inequality}

%We assume for simplicity that \(\sigma=0\) and that \(O=\{x_d\le A\}\) for some \(A\in \R\).
In the following for any $a<b$ and $\alpha\in\R$ we set 
$$
S_{a,b}^\alpha:=\Big\{x\in \R^N:a<x_N-\alpha x_1<b\Big\}.
$$
When $\alpha=\frac{\sigma}{\sqrt{1-\sigma^2}}$ we shall simply write
$$
S_{a,b}:=\Big\{x\in \R^N:a<x_N-\frac{\sigma x_1}{\sqrt{1-\sigma^2}}<b\Big\}.
$$
In the following we shall write a generic point $x\in\R^N$ as $(x',x_N)$ with $x'=(x_1,\dots,x_{N-1})$. With a slight abuse of notation we will also denote by $x'$ the generic point of $\{x_N=0\}\simeq\R^{N-1}$.
 
For $R>0$, $x'\in\R^{N-1}$ we set $D_R(x')=\{y'\in \R^{N-1}:\, |x'-y'|<R\}$ and $\mathcal C_R(x'):=D_R(x')\times\R$, $D^+_R(x')=D_R(x')\cap\{x_1>0\}$, $\mathcal C^+_R(x')=\mathcal C_R(x')\cap H$. When $x'=0$ we will omit the center. 

\begin{lemma}\label{lm:savin}
There exist two universal constants \(\eps_0\), \(\eta_0\in (0,1/2)\),  with the following properties: If $E\subset H$ is a $(\Lambda,1)$-minimizer of $\Fc_\sigma$ in $\R^N$ with obstacle $O\in\mathcal B_R$, 
 such that for $0<r\leq1$
 \begin{equation}\label{savin11}
 \partial E\cap \mathcal C_{2r}^+\subset S_{a,b}\qquad \text{with }(b-a)\leq \eps_0 r\,,
 \end{equation}
 \begin{equation}\label{rem:guido1}
 \Big\{x_N<\frac{\sigma x_1}{\sqrt{1-\sigma^2}}+a\Big\}\cap\mathcal C_{2r}^+
\subset E\,,
\end{equation}
 $\Lambda r^2<\eta_0(b-a)$ and 
\begin{equation}\label{116}
 \frac{r^2}{R}\leq \eta_0(b-a)\,,
\end{equation}
\begin{equation}\label{117}
\partial_{\partial H}O\cap \mathcal C_r\cap S_{a,b}\not=\emptyset\,,
\end{equation}
then there exist \(a'\ge a\), \(b'\le b\) with 
\[
b'-a'\le (1-\eta_0) (b-a)  
\]
 such that 
\[
\partial E\cap \mathcal C_{\frac{r}{2}}^+\subset S_{a', b'}.
\]
The same conclusion holds if assumption \eqref{116} is replaced by  
\begin{equation}\label{118}
\partial_{\partial H}O\cap \mathcal C_{2r}\cap S_{a,b}=\emptyset\,.
\end{equation}
\end{lemma}

%\begin{lemma}\label{lm:savin}
%There exist two universal constants \(\eps_0\), \(\eta_0\in (0,1/2)\),  with the following properties: If $E\subset H$ is a $(\Lambda,r_0)$-minimizer of $\Fc_\sigma$ in $\R^N$ with obstacle $O\in\mathcal B_R$
% such that, for $0<r\leq1$
% \begin{equation}\label{savin11}
% \partial E\cap\mathcal C_{2r}^+\subset S_{a,b}\qquad \text{with }(b-a)\leq \eps_0 r\,,
% \end{equation}
% \begin{equation}\label{rem:guido1}
% \Big\{x_N<\frac{\sigma x_1}{\sqrt{1-\sigma^2}}+a\Big\}\cap\mathcal C_{2r}^+
%\subset E\,,
%\end{equation}
% $\Lambda r^2<\eta_0(b-a)$ and 
%$$
%O\cap S_{a, b}=\{x\in S_{a, b}:\, x_1=0,\, x_N<\psi(x_2, \dots, x_{N-1})\}\,,
%$$
%for some $\psi\in C^{1,1}(\R^{N-2})$, with 
%\begin{equation}\label{116}
%r^2\Div \biggl(\frac{\nabla \psi+\xi_0}{\sqrt{1+|\nabla \psi+\xi_0|^2}}\biggl)+r^2\|D^2\psi\|_\infty\leq \eta_0(b-a)\,,
%\end{equation}
%where $\xi_0=\frac{\sigma e_1}{\sqrt{1-\sigma^2}}$,
%then there exist \(a'\ge a\), \(b'\le b\) with 
%\[
%b'-a'\le (1-\eta_0) (b-a)  
%\]
% such that 
%\[
%\partial E\cap C^+_{\frac{r}{2}}\subset S_{a', b'}.
%\]
%\end{lemma}
%
% \begin{equation}\label{rem:guido1}
% \Big\{x_N<\frac{\sigma x_1}{\sqrt{1-\sigma^2}}+a\Big\}\cap\mathcal C_{2r}^+
%\subset E\,,
%\end{equation}

\begin{remark}\label{rem:guido} We start noticing that  if $E\subset H$ is a $(\Lambda,1)$-minimizer of $\Fc_\sigma$ in $\R^N$ with obstacle $O\in\mathcal B_R$
 such that, for $0<r\leq1$ \eqref{savin11} and \eqref{rem:guido1} hold, then

\beq\label{guido1}
 \Big\{x_N<\frac{\sigma x_1}{\sqrt{1-\sigma^2}}+a\Big\}\cap\mathcal C_{2r}^+
\subset E\cap \mathcal C_{2r}^+\subset \Big\{x_N<\frac{\sigma x_1}{\sqrt{1-\sigma^2}}+b\Big\}\cap\mathcal C_{2r}^+
\eeq
provided $\e_0$ is small enough.
Indeed if the above inequalities were not true then $E\cap \mathcal C_{2r}^+$ would be contained in $S_{a,b}$ thus violating the volume density estimates in Proposition~\ref{p:densityestimates}. 
\end{remark}

%Indeed the first assumption implies that \(1_E\) is constant to the two ``halphaspaces'' and  density estimates and compactness  the only possibility is the \(E=\emptyset\) on the upper one and \(E^c=\emptyset\)  on the lower one.

 We will investigate the consequences of the flatness condition:
$$
\partial E\cap \mathcal \mathcal C^+_{2r}\subset S_{-\eps r, \eps r}\,.
$$
Thanks to  Remark~\ref{rem:guido} we may define two functions $u^{\pm}:D^+_{2r}\to \R$ as

%We start noticing the following general fact, if \(E\) is a minimizer and 
%\[
%\partial E \cap B_1^+\subset \{ \abs{x_d-\alpha x_1}\le \eps\big\}
%\]
%for \(\alpha\ge 0\) and \(\eps \ll 1\), then necessarily  
%\[
% \{x_d\le \alpha x_1-\eps\}\cap B_1^+\subset E\cap B_1^+\subset \{x_d\le \alpha x_1+\eps\}\cap B_1^+.
%\]
%Indeed the first assumption implies that \(1_E\) is constant to the two ``halphaspaces'' and  density estimates and compactness  the only possibility is the \(E=\emptyset\) on the upper one and \(E^c=\emptyset\)  on the lower one. In particular if we use as coordinates \(X=(x,x_d)\), the functions  \(u^\pm: D^+_{2}=\{|x|\le 2, x_1\ge 0\}\to \R\) as  
\[
\begin{split}
u^+(x')&=\max \{ x_N: (x',x_N) \in \partial E\}
\\
u^-(x')&=\min \{ x_N: (x',x_N) \in \partial E\}.
\end{split}
\]

Note that $u^+$ is upper semicontinuous and $u^-$ is lower semicontinuous. In particular we may define for every $x'\in D_{2r}\cap\{x_1=0\}$
$$
u^-(x')=\inf\{\liminf_{h\to\infty}u^-(x_h'):\,x_h'\in D_{2r}^+,\,x_h'\to x'\}
$$
and, similarly, $u^+(x')=\sup\{\limsup_{h}u^-(x_h'):\,x_h'\in D_{2r}^+,\,x_h'\to x'\}$. Observe that $\{(x',x_N):\, x'\in D_{2r}^+,\,x_N<u^-(x')\}\subset E$ and thus from the above definition it follows that for 
$$
\{(x',x_N):\, x'\in D_{2r}\cap\{x_1=0\},\,x_N<u^-(x')\}\subset \partial E\cap\partial H\,.
$$
%\begin{lemma}\label{e:upm}
%\(u^+\) is an upper- semicontinuous and \(u^{-}\) is lower-semicontoniuos. Moreover \(u^+\) is a viscosity solution of 
%\[
%\begin{cases}
%H_{u^+}\ge 0 \qquad&\text{in \( D^+_{2}\)}
%\\
%\partial_1 u^+\ge 0 &\text{on \(D_{2}\cap \{x_1=0\}\)}
%\end{cases}
%\]
%and
%\[
%\begin{cases}
%H_{u^-}\le 0 \qquad&\text{in \( D^+_{2}\)}
%\\
%\partial_1 u^-\le 0 &\text{on \(D_{2}\cap \{x_1=0\}\cap \{u^-<A\}\)}.
%\end{cases}
%\]
%Moreover \(u^{\pm}\le A\) on \(D_{2}\cap \{x_1=0\}\). In particular if we let \(v^\pm\) be the even reflection of \(u^\pm\) with respect to \(\{x_1=0\}\) then \(v^+\) is a viscosity solution of 
%\[
%H_{v^+}\ge 0\qquad\text{in \(D_{2}\)}
%\]
%and \(w^-=\min\{v^-, A\}\) is a viscosity solution of 
%\[
%H_{w^-}\le 0\qquad\text{in \(D_{2}\)}.
%\]
%\end{lemma}
%\begin{proof} Exercise...
%\end{proof}
In the following we recall the notions of viscosity super- and subsolutions.
\begin{definition}\label{viscoM}
Let $\Omega\subset\R^{d}$ be an open set and let $v:\Omega\to\R$ be a lower (upper) semicontinuous function. Given  $\xi_0\in\R^{d}$, a constant $\gamma\in\R$  we say that $v$ satisfies the inequality 
\beq\label{visco1}
\Div \biggl(\frac{\nabla v+\xi_0}{\sqrt{1+|\nabla v+\xi_0|^2}}\biggr)  \leq \gamma \,\,(\geq\gamma)
\eeq
in the viscosity sense
if for any function $\varphi\in C^2(\Omega)$ such that $\varphi\leq v$ ($\varphi\geq v$) in a neighborhood of a point $x_0\in\Omega$, $\varphi(x_0)=v(x_0)$ one has
\beq\label{visco2}
\Div \biggl(\frac{\nabla\varphi+\xi_0}{\sqrt{1+|\nabla\varphi+\xi_0|^2}}\biggr)(x_0)\leq\gamma\,\,(\geq\gamma)\,.
\eeq
Moreover,   a lower (upper) semicontinuous function $v$ in $\overline\Omega$ satisfies the Neumann boundary condition
$$
\frac{\nabla v\cdot n}{\sqrt{1+|\nabla v|^2}}\leq \gamma \,\,\,(\geq \gamma)\,, \qquad\text{on $\Gamma$}
$$
in the viscosity sense, where $\Gamma$ is a  subset of $\partial\Omega$ and $n$ stands for the inner normal to $\partial\Omega$, if for every $\varphi\in C^2(\R^d)$ such that $\varphi\leq v$ ($\varphi\geq v$) in a neighborhood of a point $x_0\in\Gamma$, $\varphi(x_0)=v(x_0)$,    one has that
$$
\frac{\nabla\varphi(x_0)\cdot n}{\sqrt{1+|\nabla\varphi(x_0)|^2}}\leq \gamma \,\,\,(\geq \gamma)\,.
$$
\end{definition}
In the following we also need the following restricted notions of viscosity super- and subsolutions.
\begin{definition}\label{viscorestr}
Let $v:\Omega\to\R$ be a lower (upper) semicontinuous function, $\xi_0\in\R^{d}$ and  $\gamma\in\R$.  Given $\kappa>0$, we say that $v$ satisfies the inequality \eqref{visco1} in the $\kappa$-viscosity sense if \eqref{visco2} holds for any function $\varphi\in C^2(\Omega)$ such that $\varphi\leq v$ ($\varphi\geq v$) in a neighborhood of a point $x_0\in\Omega$, $\varphi(x_0)=v(x_0)$ and  $|\nabla\varphi(x_0)|\leq\kappa$.
\end{definition}

We now recall a crucial result, which is essentially contained  in  \cite{Savin}.
\begin{proposition}\label{prop:savin}
Let $\xi\in\R^{d}$ with $|\xi|\leq M$. 
There exist two constants $C_0>1$ and $\mu_0\in(0,1)$, depending only on $M$ and $d$, with the following properties: Let  $k$ be a positive integer and $\nu>0$ such that $C_0^k\nu\leq1$ and let  $v:\overline B_2\to(0,\infty)$ be a lower semicontinuous function, bounded from above, satisfying
\begin{equation}\label{prop:savin1}
\Div \biggl(\frac{\nabla u+\xi}{\sqrt{1+|\nabla u+\xi|^2}}\biggr)  \leq\nu
\end{equation}
in the  $(C_0^k\nu)$-viscosity sense.
If there exists a point $x_0\in B_{1/2}$ such that $v(x_0)\leq\nu$, then
$$
|\{v\leq C_0^k\nu\}\cap B_{1}|\geq(1-\mu_0^k)|B_1|\,.
$$
\end{proposition}
Note that in \cite{Savin} condition \eqref{prop:savin1} is assumed to hold in the usual viscosity sense and not in the $(C_0^k\nu)$-viscosity sense considered here. However the proof of \cite{Savin} extends to our framework without significant changes, see  the Appendix where we provide the details for the reader's convenience.

We now prove a comparison lemma which is a variant of Lemma 2.12 in \cite{De-PhilippisMaggi15}. To this aim, in the following, given $0<a<r$ we set
$$
D_{r,a}:=D_r\cap\{|x_1|<a\}\quad\text{and} \quad D_{r,a}^+:=D^+_r\cap\{x_1<a\}\,.
$$
\begin{lemma}\label{lm:DPCM}
Let $E\subset H$ be a $0$-minimizer  of $\Fc_{-\sigma}$ with obstacle $O\in\mathcal B_R$ for some $R>0$, let $0<\eta<r$ and let $u_0\in C^2(\overline{D_{r,\eta}^+})$ be such that
\beq\label{DPCM1}
\begin{split}
\Div\bigg(\frac{\nabla u_0}{\sqrt{1+|\nabla u_0|^2}}\bigg)\geq0\qquad&\text{in $D_{r,\eta}^+,$} \\
 \frac{\pa_1 u_0}{\sqrt{1+|\nabla u_0|^2}}\geq\sigma\qquad  &\text{on $\pa D_{r,\eta}^+\cap\{x_1=0\}.$}
\end{split}
\eeq
Assume also that $E$ is bounded from below,
\beq\label{DPCM2}
E\cap[(\pa D_{r,\eta}^+\cap\{x_1>0\})\times\R]\subset\{(x',x_N)\in(\pa D_{r,\eta}^+\cap\{x_1>0\})\times\R:\,x_N\geq u_0(x')\}
\eeq
and that 
\beq\label{DPCM3}
\H^{N-1}\big(\pa E\cap[(\pa D_{r,\eta}^+\cap\{x_1>0\})\times\R]\big)=0\,.
\eeq
Then, $$
E\cap(D_{r,\eta}^+\times\R)\subset\{(x',x_N)\in D_{r,\eta}^+\times\R:\,x_N\geq u_0(x')\}\,.
$$
\end{lemma}
\begin{proof} We adapt the argument of \cite[Lemma~2.12]{De-PhilippisMaggi15}. Denote
$$
C:=D_{r,\eta}^+\times\R
$$
and set 
$$
F^{\pm}:=\{(x', x_N)\in C:\, x_N\gtrless u_0(x')\}\,.
$$
Consider the competitor given by
$$
G:=(E\setminus C)\cup (E\cap F^+)\,,
$$
which is  an admissible compact perturbation of $E$ since $E$ is bounded from below. From the minimality of $E$ we then obtain
\beq\label{compa1}
\begin{split}
\H^{N-1}(\pa E&\cap F^-)+\H^{N-1}(\pa E\cap \pa F^-\cap C\cap\{\nu_E=\nu_{F^-}\})\\
&\quad-\sigma\H^{N-1}(\pa E\cap \pa F^-\cap \pa H)\\ 
&\leq \H^{N-1}(E\cap \pa F^-\cap C)\,.
\end{split}
\eeq
Denote
$$
X(x):=\frac{1}{\sqrt{1+|\nabla u_0(x')|^2}}(\nabla u_0(x'), -1)\,,
$$
and observe that $\Div X\geq 0$ in $C$, thanks to the first assumption in \eqref{DPCM1}. Then by the Divergence Theorem we obtain 
\[
\begin{split}
0\leq\int_{E\cap F^-}\Div X\, dx&=\int_{\pa E\cap F^-}X\cdot\nu_E\, d\H^{N-1}+ \int_{\pa F^-\cap E\cap C}X\cdot\nu_{F^-}\, d\H^{N-1}\\
&+\int_{\pa E\cap \pa F^-\cap C\cap \{\nu_E=\nu_{F^-}\}}X\cdot\nu_E\, d\H^{N-1}-\int_{\pa E\cap \pa F^-\cap \pa H}X\cdot e_1\, d\H^{N-1}\,.
\end{split}
\]
Observing that $X\cdot\nu_{F^-}=-1$ on $\pa F^-\cap C$, from the previous inequality, we get, thanks to the second assumption in \eqref{DPCM1},
\[
\begin{split}
\sigma \H^{N-1} (\pa E\cap \pa F^-\cap \pa H)&\leq \int_{\pa E\cap \pa F^-\cap \pa H} \frac{\pa_1 u_0}{\sqrt{1+|\nabla u_0|^2}}\, d\H^{N-1}\\
&\leq \int_{\pa E\cap F^-}X\cdot\nu_E\, d\H^{N-1}+\int_{\pa E\cap \pa F^-\cap C\cap \{\nu_E=\nu_{F^-}\}}X\cdot\nu_E\, d\H^{N-1}\\
&\quad - \H^{N-1}(\pa F^-\cap E\cap C)\\
&\leq  \H^{N-1}(\pa E\cap F^-)+\H^{N-1}(\pa E\cap \pa F^-\cap C\cap \{\nu_E=\nu_{F^-}\})\\
&\quad- \H^{N-1}(\pa F^-\cap E\cap C)\\
&\leq \sigma \H^{N-1} (\pa E\cap \pa F^-\cap \pa H)\,,
\end{split}
\]
where the last inequality follows from \eqref{compa1}.  Thus all the inequalities above are equalities and, in turn, $X\cdot\nu_E=1$ on $\pa E\cap F^-$. We now conclude by applying the Divergence Theorem in $E\cap F^-$ to the vector field $Y(x'):=(x', x'\cdot \nabla u_0(x'))$. Since $Y\cdot \nu_{E\cap F^-}=0$ $\H^{N-1}$-a.e. on $ \pa E\cap  F^-$ and on $\pa F^-\cap C$, recalling \eqref{DPCM2} and \eqref{DPCM3}, we obtain 
$$
(N-1)|E\cap F^-|=\int_{E\cap F^-}\Div Y\, dx=0\,,
$$
and the conclusion follows. 
\end{proof}

\begin{lemma}\label{pace}
Let $E\subset H$ is a $(\Lambda,1)$-minimizer of $\Fc_\sigma$ in $\R^N$ with obstacle $O\in\mathcal B_R$. Assume also that  \eqref{guido1} holds for some $a, b$ and with $\frac{\sigma}{\sqrt{1-\sigma^2}}$ replaced by $\alpha$ for some $\alpha\in\R$, with $r=1$, so that the functions $u^{\pm}$ are well defined on $D^+_2$. Then the functions $u^+$, $u^-$
satisfy   
\begin{equation}\label{irene-1}
\begin{cases}
\displaystyle\Div \bigg(\frac{\nabla u^+}{\sqrt{1+|\nabla u^+|^2}}\bigg)  \geq-\Lambda \,\,\,\, \text{in $D^+_2$,} &\vspace{0.2cm}\cr
\displaystyle\Div \biggl(\frac{\nabla u^-}{\sqrt{1+|\nabla u^-|^2}}\biggr)  \leq\Lambda \,\,\,\, \text{in $D^+_2$,} &\vspace{0.2cm}\cr
\displaystyle\frac{\pa_1 u^+}{\sqrt{1+|\nabla u^+|^2}}\geq\sigma \,\,\,\,\text{on $\Gamma$,} & \vspace{0.2cm}\cr
\displaystyle\frac{\pa_1 u^-}{\sqrt{1+|\nabla u^-|^2}}\leq\sigma \,\,\,\,\text{on $\Gamma\cap\{x':\,(x',u^-(x'))\in O\}$,
} &
\end{cases}
\end{equation}
in the viscosity sense, where $\Gamma:= D_2\cap\{x_1=0\}$.
\end{lemma}
\begin{proof}  

{\bf Step 1.} We prove the statement first for $u^-$.

We  fix $x'_0\in D^+_2$ and take a function $\varphi\in C^2(D^+_2)$ such that $x'_0$ is a  strict minimum point for $u^- -\varphi$ in $D^+_2$  and $u^-(x'_0)=\varphi(x'_0)$.
Assume by contradiction that 
\begin{equation}\label{irene0}
\Div \biggl(\frac{\nabla \varphi}{\sqrt{1+|\nabla \varphi|^2}}\biggr)> \Lambda
\end{equation}
in a neighborhood of $x'_0$. For $\eta>0$  set
$$
E_\eta:=E\cup \{x\in E^c:\, x_N<\varphi(x')+\eta\}
$$
and note that if $\eta$ is sufficiently small is an admissible competitor for $E$.
Note that by $\Lambda$-minimality, for all but countably many such $\eta$, 
\begin{align}
&\Fc_\sigma (E; \mathcal C_2) - \Fc_\sigma (E_\eta; \mathcal C_2)\nonumber \\
&=
\H^{N-1}(\partial E\cap \{x_N< \varphi(x')+\eta\})-\H^{N-1}(E^c\cap \{x_N
= \varphi(x')+\eta\})\leq \Lambda |E_\eta\setminus E|\,.\label{irene1}
\end{align}
On the other hand, setting 
$$
X:=\biggl(\frac{\nabla \varphi}{\sqrt{1+|\nabla \varphi|^2}}, -\frac{1}{\sqrt{1+|\nabla \varphi|^2}}\biggr)\,,
$$
by \eqref{irene0} and \eqref{irene1} we have, if $\eta$ is small enough, 
\begin{align*}
&\Lambda |E_\eta\setminus E|<\int_{E_\eta\setminus E} \Div X\, dx\\
&= 
-\int_{\partial E\cap \{x_N< \varphi(x')+\eta\}} X\cdot \nu_{E}\, d\H^{N-1}+
\int_{ E^c\cap \{x_N= \varphi(x')+\eta\}} X\cdot \nu_{E_\eta}\, d\H^{N-1}\\
&\leq  \H^{N-1}(\partial E\cap \{x_N< \varphi(x')+\eta\})-\H^{N-1}(E^c\cap \{x_N
= \varphi(x')+\eta\})\leq \Lambda |E_\eta\setminus E|\,,
\end{align*}
which yields a contradiction.

{\bf Step 2.} 
Let $\varphi\in C^2( D_2)$ such that $\varphi\leq u^-$ in a neighborhood of $x'_0\in D_2\cap\{x_1=0\}$ in $D^+_2$,  $\varphi(x'_0)=u^-(x'_0)$  and $(x'_0,u^-(x'_0))\in O$ and thus it lies strictly below the boundary of the obstacle. We claim that 
\beq\label{decay2}
\frac{\partial_1\varphi(x'_0)}{\sqrt{1+|\nabla\varphi(x'_0)|^2}}\leq \sigma\,.
\eeq
We argue by contradiction assuming that
\beq\label{decay3}
\frac{\partial_1\varphi(x'_0)}{\sqrt{1+|\nabla\varphi(x'_0)|^2}}> \sigma\,.
\eeq
In this case we set 
\[
E_n:=n\big(E-(x'_0, u^-(x'_0))\big),\quad\! \varphi_n(x'):=n\Big[\varphi\Big(x_0'+\frac{x'}{n}\Big)-\varphi(x_0')\Big], \quad\!  O_n:=n\big(O-(x'_0, u^-(x'_0))\big)\,.
\]
Observe that $E_n$ is a $(\Lambda/n, n)$-minimizer of $\Fc_\sigma$ with obstacle $O_n$. By Theorem~\ref{th:compactness} we may assume that up to a not relabelled subsequence $E_n\to E_\infty$ in $L^1_{loc}(\R^N)$, where $E_\infty$ is a $0$-minimizer of $\Fc_\sigma$ (with obstacle $\pa H$). Moreover, from property (ii) of Theorem~\ref{th:compactness} we have that 
\beq\label{zeroin}
0\in \overline{\pa E_\infty\cap H}\,.
\eeq
Note also that $\varphi_n\to \varphi_\infty$ locally uniformly, where $\varphi_\infty(x'):=\nabla \varphi(x'_0)\cdot x'$, and that
\beq\label{decay6}
E_\infty\supset\{(x', x_N)\in H:\, x_N<\varphi_\infty(x')\}\,.
\eeq
 Set now for $\e>0$ small (to be chosen) 
$$
\psi_\e(x'):=\varphi_\infty(x')-\e x_1-\frac{\e^2}2|x'|^2+\frac\e2 x_1^2\,.
$$

A direct calculation shows that 
$$
\frac1\e\Div \biggl(\frac{\nabla \psi_\e}{\sqrt{1+|\nabla \psi_\e|^2}}\biggr)\to \frac{1+|\nabla \varphi(x'_0)|^2-| \pa_1\varphi(x'_0)|^2}{\big(1+|\nabla \varphi(x'_0)|^2\big)^{3/2}}>0\,,
$$
locally uniformly. Therefore,  we may fix $\e>0$ so small that 
\beq\label{decay4}
\Div \biggl(\frac{\nabla \psi_\e}{\sqrt{1+|\nabla \psi_\e|^2}}\biggr)>0\quad \text{in $\overline D_1^+$}
\eeq
and that, recalling  \eqref{decay3}, 
\beq\label{decay5}
\frac{\pa_1 \psi_\e}{\sqrt{1+|\nabla \psi_\e|^2}}>\sigma \quad\text{on $\pa D_1^+\cap\{x_1=0\}$.}
\eeq
Observe that we may now choose $\eta, r\in (0,1)$ such that 
\beq\label{psieta}
\psi_\e<\varphi_{\infty}\quad\text{on }\pa D_{r,\eta}\,\quad\text{and}\quad \H^{N-1}\big(\pa E_\infty \cap[(\pa D_{r,\eta}^+\cap\{x_1>0\})\times\R]\big)=0\,.
\eeq
Finally, let $w\in C_{c}^{\infty}(D_{r,\eta})$, with $w(0)>0$, and note that by \eqref{decay4} and \eqref{decay5} we may choose $\de\in (0, 1)$ so small that, 
setting $\psi_{\e, \de}:=\psi_\e+\de w$, we have
 \beq\label{decay7.1}
 \begin{split}
 & \Div \biggl(\frac{\nabla \psi_{\e, \de}}{\sqrt{1+|\nabla \psi_{\e, \de}|^2}}\biggr)>0\quad \text{in $ D_{r,\eta}^+$,}\\
 & \frac{\pa_1 \psi_{\e,\de}}{\sqrt{1+|\nabla \psi_{\e,\de}|^2}}>\sigma \quad\text{on $\pa D_{r,\eta}^+\cap\{x_1=0\}$.}
 \end{split}
 \eeq
Recall  that $H\setminus E_\infty$ is a $0$-minimizer of $\Fc_{-\sigma}$. Note that this minimality property is a consequence of the fact that $E_\infty$ is a $0$-minimizer of $\Fc_{\sigma}$, which in turn follows from the assumption that  $(x'_0,u^-(x'_0)$ does not touch the boundary of $O$ . Moreover, by \eqref{decay6} and \eqref{psieta}, we have 
$$
(H\setminus E_\infty) \cap[(\pa D_{r,\eta}^+\cap\{x_1>0\})\times\R]\subset\{(x',x_N)\in(\pa D_{r,\eta}^+\cap\{x_1>0\})\times\R:\,x_N\geq \psi_{\e,\de}(x')\}\,.
$$
Therefore, taking into account also \eqref{decay7.1}, we can apply Lemma~\ref{lm:DPCM}, with $E$, $u_0$ replaced by $H\setminus E_\infty$, $\psi_{\e,\de}$, respectively, to conclude that
$$
(H\setminus E_\infty) \cap(D_{r,\eta}^+\times\R)\subset\{(x',x_N)\in D_{r,\eta}^+\times\R:\,x_N\geq \psi_{\e,\delta}(x')\}\,.
$$
In particular, since $\psi_{\e,\de}(0)>0$ the latter inclusion contradicts \eqref{zeroin}.  This concludes the proof of \eqref{decay2}.

{\bf Step 3.} Concerning the proof of the statement for $u^+$, the case where $x'_0\in D^+_2$ is proved with the same argument used in Step~1. Instead, when $x'_0\in D_2\cap\{x_1=0\}$ we argue as In Step~2, replacing $E$ by $\widetilde E:=\{(x',-x_N):\, (x',x_N)\in E\}$ with the only difference that at the end of the proof we apply Lemma~\ref{lm:DPCM} to $\widetilde E_\infty$ instead of $H\setminus \widetilde E_\infty$. 
\end{proof}

In the following for $x'\in D^+_{2r}$ we set
$$
u^{\pm}_\sigma(x'):=u^{\pm}(x')-\frac{\sigma x_1}{\sqrt{1-\sigma^2}}\,.
$$
\begin{lemma}\label{lm:decay}
There exist two universal constants  $C_1>1$, $\mu_1\in(0,1)$ with the following property: Let $E\subset H$ be a $(\Lambda,1)$-minimizer of $\Fc_\sigma$ in $\R^N$ with obstacle $O\in\mathcal B_R$. Assume  that \eqref{guido1} holds for $a=-\eps, b=\eps$,  $\eps\in (0, 1)$, $r=1$, so that the functions $u^{\pm}$ are well defined on $D^+_2$. Assume also that  $\Lambda<\eta \eps$ with $\eta \in (0, 1)$ and 
$$
O\cap S_{-\eps, \eps}\cap\mathcal C_2=\{x\in S_{-\eps, \eps}:\, x_1=0,\, x_N<\psi(x_2, \dots, x_{N-1})\}\cap\mathcal C_2\,,
$$
where $\psi\in C^{1,1}(\partial H\cap D_2)$ with 
%$\|\nabla \psi\|_{L^\infty}\leq1$ and 
$$
\Div \biggl(\frac{\nabla \psi+\xi_0}{\sqrt{1+|\nabla \psi+\xi_0|^2}}\biggl)\leq \eta\eps\,,\qquad \xi_0:=\frac{\sigma}{\sqrt{1-\sigma^2}}e_1\,. 
$$
 Then the following holds:
\begin{enumerate}[\textup{(\roman*)}]
\item If  there exists \(\bar x'\in D_{1/2}^+\) such that  \(u^+_\sigma(\bar x')\ge (1-\eta)\eps\), then 
\begin{multline*}
\H^{N-1}(\{x'\in D_{1}^+: u^+_\sigma(x')\ge (1-C_1^k\eta)\eps \})\ge (1-\mu_1^k)\H^{N-1} (D^+_1)
\\
 \text{for all \(k\ge 1\) such that \(C_1^k\sqrt{\eta\eps}\le 1\)}\,.
\end{multline*}
\vskip 0.2cm
\item If   there exists \(\bar x'\in D_{1/2}^+\) such that  \(u^-_\sigma(\bar x')\le -(1-\eta)\eps\), then 
\begin{multline*}
\H^{N-1}(\{x'\in D_{1}^+: \min\{u^-_\sigma(x'),  \psi(x')\}\leq -(1-C_1^k\eta)\eps \})\ge (1-\mu_1^k)\H^{N-1} (D^+_1)
\\
 \text{for all \(k\ge 1\) such that \(C_1^k\sqrt{\eta\eps}\le 1\)}\,,
\end{multline*}
where we have set $\psi(x')=\psi(x_2,\dots,x_{N-1})$ when $x_1>0$.
\end{enumerate}
\end{lemma}

\begin{proof} 
We start by proving (ii). To this aim we split the proof in three steps.
 \par\noindent
{\bf Step 1.} Set $w^-(x')=\min\{u^-(x'),\psi(x')+\xi_0\cdot x'\}=\min\big\{u^-(x'),\psi(x')+\frac{\sigma x_1}{\sqrt{1-\sigma^2}}\big\}$.
We claim  that $w^-$ satisfies
$$
\Div \biggl(\frac{\nabla w^-}{\sqrt{1+|\nabla w^-|^2}}\biggr)\leq {\eta\eps}\qquad\text{ in }D^+_2
$$
in the viscosity sense.
To this aim assume that $x'_0\in D^+_2$ is a  strict minimum point for $w^- -\varphi$ in $D^+_2$, with $\varphi\in C^2(D^+_2)$  and $w^-(x'_0)=\varphi(x'_0)$. If $w^-(x'_0)=\psi(x'_0)+\xi_0\cdot x'_0$, then clearly we also have that $x'_0$ is a minimum point of $x'\to \psi(x')+\xi_0\cdot x'-\varphi(x')$ and thus
$$
\Div \biggl(\frac{\nabla \varphi}{\sqrt{1+|\nabla \varphi |^2}}\biggr)(x'_0)\leq 
\Div \biggl(\frac{\nabla \psi+\xi_0}{\sqrt{1+|\nabla \psi+\xi_0|^2}}\biggr)(x'_0)\leq{\eta\eps}\,.
$$
If otherwise $w^-(x'_0)=u^-(x'_0)$ the claim follows from Lemma~\ref{pace} since $\Lambda<\eta\eps$.

{\bf Step 2.} 
 We now denote by $w_\sigma^-$ the function defined on $D_2$ obtained by even reflection of 
 $$
\min\{u^-_\sigma, \psi\}-L_kx_1,\qquad L_k:=\Big(\frac{\sigma^+}{2\sqrt{1-\sigma^2}}+1\Big)C_1^{2k}\eta^2\eps^2\,,
$$
 with  respect to $\partial H$, where $C_1>1$ will be chosen later and $k$ is an integer such that $C_1^k\sqrt{\eta\eps}\leq1$. We claim that $w_\sigma^-$ satisfies the inequality
 \begin{equation}\label{decay7}
\Div \biggl(\frac{\nabla w^-_\sigma+\xi_0+L_ke_1}{\sqrt{1+|\nabla w^-_\sigma+\xi_0+L_ke_1|^2}}\biggr)\leq \eta\eps\quad\text{ in }D_2
\end{equation}
in the   $(C_1^k\eta\eps)$-viscosity sense, see Definition~\ref{viscorestr}.
To this aim let $\varphi\in C^2(D_2)$ such that $w_\sigma^- -\varphi$ has a minimum at $x'_0\in D_2$ and $|\nabla \varphi(x'_0)|\leq C_1^k\eta\eps$.
We first show that  $x'_0\not \in \partial H$. Indeed, assume by contradiction the opposite and assume in addition that 
 $u^-_\sigma(x'_0)<\psi(x'_0)$, that is $u^-(x'_0)<\psi(x'_0)$. Then, by Lemma~\ref{pace}, using as a test function $\varphi+\Big(\frac{\sigma }{\sqrt{1-\sigma^2}}+L_k\Big)x_1$, we infer that
 $$
 \frac{\pa_1\varphi(x'_0)+A_k}{\sqrt{1+(\pa_1\varphi(x'_0)+A_k)^2+|\nabla'\varphi(x'_0)|^2}}\leq\sigma\,,
 $$
 where we set $A_k:=\frac{\sigma }{\sqrt{1-\sigma^2}}+L_k$ and $\nabla'\varphi=\nabla\varphi-(\pa_1\varphi)e_1$. From this inequality, using that $|\nabla'\varphi|\leq C_1^k\eta\eps$, we easily get
 $$
 \pa_1\varphi(x'_0)+A_k\leq\frac{\sigma }{\sqrt{1-\sigma^2}}+\frac{\sigma^+ }{\sqrt{1-\sigma^2}}\frac{C_1^{2k}\eta^2\eps^2}{2}\,.
 $$
 In turn, the last inequality implies that $\pa_1\varphi(x'_0)<0$.
 
 On the other hand the symmetric argument in $D_2^-$ shows that 
 $\partial_1\varphi (x'_0)>0$, thus leading to a contradiction. 
 If instead $u^-(x'_0)=\psi(x'_0)$, i.e., $u^-_\sigma(x'_0)=\psi(x'_0)$ then $x'_0$ is a minimum for $\psi- L_kx_1
-\varphi$ in $\overline D^+_2$ and thus, in particular, 
 $$
 \partial_1\varphi (x'_0)\leq\partial_1\psi(x'_0)-L_k=-L_k<0.
 $$
 Arguing symmetrically in $D^-_2$ we also get $\partial_1\varphi (x'_0)>0$, which is again a contradiction.

Thus,  $x'_0\in D^+_2\cup D^-_2$ and the fact that $w_\sigma^-$ satisfies \eqref{decay7} in the viscosity sense now follows easily from Step 1, since on $D^+_2$ we have $w^-_\sigma(x')=w^-(x')-\big(\frac{\sigma}{\sqrt{1-\sigma^2}}+L_k\big)x_1$.

\noindent
{\bf Step 3.} 
 Observe that from our assumptions  we have that  $-\eps\leq\min\{u^-_\sigma,  \psi\}\leq\eps$. Thus, 
$0<w_\sigma^- +\eps+2L_k$ in $D_2$. Moreover, by assumption we have
$$
w_\sigma^- (\bar x')+\eps+2L_k\leq\Big(1+\Big(\frac{\sigma^+}{\sqrt{1-\sigma^2}}+2\Big)C_1^{2k}\eta\eps\Big)\eta\eps\leq\Big(3+\frac{\sigma^+}{\sqrt{1-\sigma^2}}\Big)\eta\eps:=\nu\,.
$$ 
Therefore from Step 2 and by Proposition~\ref{prop:savin}, applied with $d=N-1$, we have that
$$
\H^{N-1}(\{w_\sigma^- +\eps+2L_k\leq C_0^k\nu\}\cap D_1)\ge (1-\mu_0^k) \H^{N-1}(D_1)\,,
$$
provided that  $C_0^k\nu\leq C_1^k\eta\eps\leq1$, where $\mu_0$ and $C_0$ are the constants provided by Proposition~\ref{prop:savin} corresponding to $M=\frac{|\sigma|}{\sqrt{1-\sigma^2}}+\frac{\sigma^+}{2\sqrt{1-\sigma^2}}+1$. Note that the inequality 
$C_0^k\nu\leq C_1^k\eta\eps$ is satisfied if we take $C_1\geq C_0(3+\sigma^+(1-\sigma^2)^{-1/2})$. Thus we finally have
$$
\H^{N-1}(\{ \min\{u^-_\sigma,  \psi\}\leq -\eps+L_kx_1-2L_k+C_0^k\nu \}\cap D_1^+)\ge (1-\mu_0^k)\H^{N-1} (D^+_1)\,,
$$
from which the conclusion (ii)  follows since $L_kx_1-2L_k+C_0^k\nu\leq C_1^k\eta\eps$ in $D^+_1$.

\medskip

Concerning the proof of (i) we argue as in the previous steps with $w^-$ replaced by $-u^+$ and $w^-_\sigma$ replaced by the even reflection of $-u^+_\sigma-L_k x_1$.
\end{proof}

\begin{remark}\label{rm:stupida}
Note that if $\partial_{\partial H}O\cap \mathcal C_2\cap S_{-\eps,\eps}=\emptyset$  the conclusion of   Lemma~\ref{lm:decay} holds with $\min\{u^-_\sigma,\psi\}$ replaced by $u^-_\sigma$.
\end{remark}

\begin{remark}\label{rm:decay}
Observe that the following interior version of the previous lemma holds:
Let $\kappa$ be a positive number. There exist a  constant  $C_1>1$ and $\mu_1\in(0,1)$ depending only on $\kappa$ with the following property: if $E\subset H$ is a $(\Lambda,1)$-minimizer of the perimeter in $\R^N$ such that
$$
\partial E\cap \mathcal C_2(y')\subset S_{-\eps, \eps}^\alpha\,,
\subset E\,,
$$
for some $y'=(y_1,\dots,y_{N-1})$ with $y_1\geq2$, 
 $\alpha\in[-\kappa,\kappa]$,  $\eps\in (0, 1)$, $\Lambda<\eta \eps$ with $\eta \in (0, 1)$,  then the following holds:
\begin{enumerate}[\textup{(\roman*)}]
\item If  there exists \(\bar x'\in D_{1/2}(y')\) such that  \(u^+(\bar x')-\alpha\bar x_1\ge (1-\eta)\eps\), then 
\begin{multline*}
\H^{N-1}(\{x'\in D_{1}(y'): u^+(x')-\alpha x_1\ge (1-C_1^k\eta)\eps \})\ge (1-\mu_1^k)\H^{N-1} (D^+_1)
\\
 \text{for all \(k\ge 1\) such that \(C_1^k\sqrt{\eta\eps}\le 1\)}.
\end{multline*}
\vskip 0.2cm
\item If   there exists \(\bar x'\in D_{1/2}(y')\) such that  \(u^-(\bar x')-\alpha\bar x_1\le -(1-\eta)\eps\), then 
\begin{multline*}
\H^{N-1}(\{x'\in D_{1}(y'): u^-(x')-\alpha x_1\leq -(1-C_1^k\eta)\eps \})\ge (1-\mu_1^k)\H^{N-1} (D^+_1)
\\
 \text{for all \(k\ge 1\) such that \(C_1^k\sqrt{\eta\eps}\le 1\)}\,.
\end{multline*}
\end{enumerate}
This clearly follows from the interior version of the previous arguments. We only remark that the uniformity of the estimates with respect to $\alpha$ varying in a bounded interval relies on the estimates provided by Proposition~\ref{prop:savin} which are uniform with respect to $\xi$ varying in a bounded set.
\end{remark}

We need also the following lemma which relies on the classical interior regularity theory for $(\Lambda,r_0)$-minimizers of the perimeter.

\begin{lemma}\label{covid}
For every $\delta\in(0,1)$, $\kappa>0$, there exists $\eps>0$ such that if  $E$ is a $(\Lambda,1)$-minimizer of  the perimeter in $\mathcal \mathcal C^+_2$ with $\Lambda\leq1$ such that 
$$
 \partial E\cap \mathcal C_2^+\subset S_{-\eps,\eps}^\alpha,
$$
with $\alpha\in[-\kappa,\kappa]$, 
then the corresponding functions $u^+$ and $u^-$ coincide in $D^+_1$ outside a set of measure less than $\delta$.
\end{lemma}
\begin{proof}
It is enough to observe that given $\delta'\in(0,1)$, a sequence $\eps_n>0$ converging to zero and a sequence $E_n$ of $(\Lambda_n,r_n)$-minimizers, with $\Lambda_n+\frac{1}{r_n}\leq2$ and  $\partial E_n\cap \mathcal C_2^+\subset S_{-\eps_n,\eps_n}^{\alpha_n}$, with $\alpha_n\in[-\kappa,\kappa]$, $\alpha_n\to\alpha$,  then from classical regularity results we have that $\partial E_n\cap \mathcal \mathcal C^+_1\cap\{x_1>\delta'\} $ converge in $C^1$ to the plane $\{x_N-\alpha x_1=0\}\cap \mathcal \mathcal C^+_1\cap\{x_1>\delta'\}$.
\end{proof}

\begin{lemma}\label{apparte}
For every $\delta>0$ there exists $\eta=\eta(\delta)\in(0,1/2)$ with the following property: Let $O\in\mathcal B_R$ be such that 
\begin{equation}\label{apparte1}
\partial H\cap\{x_N=-\eps r\}\cap \mathcal C_{2r}\subset {\overline O} \,,
\end{equation}
for some $\eps\in(0,1]$, and  for some $r>0$ with $\frac{r}{R}\leq\eta(\delta)\eps$. Assume also that 
 there exists a point $x'\in \partial_{\partial H}O\cap \mathcal C_r\cap\{-\eps r\leq x_N\leq\eps r\}$. Then the connected component of $\partial_{\partial H}O\cap \mathcal C_{2r}$ containing $x'$ is the graph of a  function $\psi\in C^{1,1}(\partial H\cap  D_{2r})$ such that
$$
\|D\psi\|_{L^\infty(\partial H\cap D_{2r})}\leq\eps,\quad r\|D^2\psi\|_{L^\infty(\partial H\cap D_{2r})}\leq\delta\eps\,.
$$
\end{lemma}
\begin{proof} By rescaling it is enough to prove the statement with $r=1$.

Observe that, under our assumptions, if $\eta(\delta)$ is sufficiently small, hence $R$ is large, then the connected component of $\partial_{\partial H}O\cap \mathcal C_{2}$ containing $x'$ is the graph of a $C^{1,1}(\partial H\cap  D_{2})$ function.
We argue by contradiction assuming that there exist a sequence $\eps_h\in(0,1]$, a sequence $O_h\in\mathcal B_{R_h}$, with  $\frac{1}{R_h}\leq\frac{\eps_h}{h}$ such that  $\partial H\cap\{x_N=-\eps_h\}\cap \mathcal C_2\subset \overline O_h$ for all $h$ and that there exist $x'_h\in \partial_{\partial H}O_h\cap \mathcal C_1\cap\{-\eps_h\leq x_N\leq\eps_h\}$ such that
\begin{equation}\label{apparte2}
\|D\psi_h\|_{L^\infty(\partial H\cap D_{2})}>\eps_h\quad\text{or}\quad\|D^2\psi_h\|_{L^\infty(\partial H\cap D_{2})}>\eps_h\delta\,,
\eeq
where the graph of $\psi_h\in C^{1,1}(\partial H\cap  D_{2})$ describes the connected component of $\partial_{\partial H}O_h\cap \mathcal C_2$.

We now set $\widetilde O_h=\Phi_h(O_h)$ where $\Phi_h(x)=(x_1,\dots,x_{N-1},\eps_h^{-1}x_N)$. Since $\Phi_h$ maps balls of radius $1$ into ellipsoids with maximal semiaxis of length $\frac{1}{\eps_h}$ it is easy to check that  $\widetilde O_h\in\mathcal B_{\eps_hR_h}$. Without loss of generality we may assume that $\Phi_h(x'_h)$ converges to a point $\bar x'\in\partial H\cap\overline{\mathcal C}_1$ with $-1\leq\bar x_N\leq1$.
Then, recalling that $\eps_hR_h\to+\infty$ and thus the uniform inner and outer ball condition for $\widetilde O_h$ hold for larger and larger radii, by a compactness argument we have that the connected components of $\partial_{\partial H} \widetilde O_h$ containing $x'_h$ converge locally (up to a subsequence) in the Hausdorff sense   to a $(N-2)$-dimensional plane $\pi$  passing through $\bar x'$ and such that $\pi\cap\mathcal C_2\subset\{x_N\geq-1\}$. Note that the latter inclusion follows from \eqref{apparte1} applied to $\widetilde O_h$ and $\eps=1$. Note also that this inclusion together with the fact that $\bar x'_N\leq1$ yields that the slope of $\pi$ is strictly less than $1$. Hence the functions $\frac{\psi_h}{\eps_h}$ converge locally uniformly (actually in $C^{1,1}$) to an affine function whose gradient has norm strictly less than 1.  This contradicts \eqref{apparte2} for $h$ sufficiently large.
\end{proof}

\begin{proof}[Proof of \cref{lm:savin}]
By a simple rescaling argument it is enough to prove the statement for $r=1$.
Up to renaming the coordinates  we can assume that \(a+b=0\), so that  if we set \(\eps=(b-a)/2\) our assumption becomes 
\[
\partial E\cap \mathcal C^+_{2}\subset S_{-\eps, \eps}.
\]
Let \(0< \eta_0<1\) to be fixed in an universal way. If \(\sup_{D_{1/2}^+} u^+_\sigma\le \eps-\eta_0 \eps\) (resp.  if \(\inf_{D_{1/2}^+} u^-_\sigma\ge -\eps+\eta_0 \eps\))   we are done by choosing  \(a'=a=-\eps\) and \(b'=\eps-\eta_0 \eps=b-\eta_0(b-a)/2\) (resp.  \(b'=b=\eps\) and \(a'=-\eps+\eta_0 \eps=a+\eta_0(b-a)/2\)). 

Hence we can assume by contradiction that there are \(\bar x', \hat x' \in D_{1/2}^+\) such that 
\begin{equation}\label{e:trieste}
u^+_\sigma(\bar x') > \eps-\eta_0 \eps\qquad\text{and}\qquad u_\sigma^-(\hat x')< -\eps+\eta_0 \eps.
\end{equation}
Assume that \eqref{117} holds. Then there exists   $\bar y'\in\partial_{\partial H}O\cap \mathcal C_1\cap\{-\eps<y_N<\eps\}$.  Then, we may apply Lemma~\ref{apparte} to conclude that the connected component of $\pa_{\pa H}O\cap \mathcal C_2$ containing $\bar y'$ is the graph of a function $\psi\in C^{1,1}(\partial H\cap D_2)$, where $\|D^2\psi\|_{L^\infty(\partial H\cap D_2)}\le\delta\eps$ with $\delta=\delta(\sigma)$ so small that
$$
\Div \biggl(\frac{\nabla \psi+\xi_0}{\sqrt{1+|\nabla \psi+\xi_0|^2}}\biggl)\leq \eta_0\eps\,.
$$
This is certainly true provided that $O\in\mathcal B_R$ with $\frac{1}{R}\leq\eta(\delta(\sigma))\eps$. Thus we can apply Lemma~\ref{lm:decay} with $k_0$ large to infer that
\[
 \H^{N-1}(\{x'\in D^+_1: u^+_\sigma(x') \ge (1-C_1^{k_0}\eta_0)\eps \})\ge (1-\mu_1^{k_0})\H^{N-1} (D^+_1)\ge \frac{3}{4}\H^{N-1} (D^+_1)
\]
and
 \begin{multline*}
\H^{N-1}(\{x'\in D_1^+: \min\{u_\sigma^-,\psi\}\le -(1-C_1^{k_0}\eta_0)\eps \})\\
\ge (1-\mu_1^{k_0})\H^{N-1} (D^+_1)\ge \frac{3}{4}\H^{N-1} (D^+_1),
\end{multline*}
provided that $C_1^{k_0}\sqrt{\eta_0}\leq1$.
If \(\psi\ge  -(1-C^{k_0}_1\eta_0)\eps\) in $D_1\cap\partial H$ we are done since from Lemma~\ref{covid} we may assume that $\H^{N-1}(\{u^+\not=u^-\})<\H^{N-1} (D^+_1)/4$ and therefore from the two previous inequalities we get
\[
 \H^{N-1}(\{x'\in D^+_1: u^+_\sigma(x') \ge (1-C_1^{k_0}\eta_0)\eps \}\cap\{u^+=u^-\})>  \frac{1}{2}\H^{N-1} (D^+_1)
\]
and
\[
\H^{N-1}(\{x'\in D_1^+: \min\{u_\sigma^-,\psi\}\le -(1-C_1^{k_0}\eta_0)\eps \}\cap\{u^+=u^-\})>  \frac{1}{2}\H^{N-1} (D^+_1)
\]
which is impossible.

Hence we only have to deal with the case where there exists a point in $D_1\cap\pa H$  at which $\psi\le  -(1-C^{k_0}_1\eta_0)\eps\le -3\eps/4\) (provided \(\eta_0\) is small). Thus, up to replacing $O$ with $O+\frac{3\eps}{4}e_N$, we may apply Lemma~\ref{apparte} in $\mathcal C_2$ with $\eps$ replaced by $\eps/4$ (and taking $\eta(\delta)$ smaller if needed)  to conclude that $\|\nabla\psi\|_{L^\infty(D_2)}\leq\frac{\eps}{4}$. In turn this estimate implies that $\psi\leq0$ in $D_2\cap \partial H$. 

Let $\Omega\subset \R^{N-1}$ be a smooth set such that $D^+_{4/3}\subset \Omega\subset D^+_{3/2}$ and consider the solution to  the following problem: 
\[
\begin{cases}
\Delta w_\mu(1+|\xi_0|^2)-D_{11}w_\mu|\xi_0|^2=-\mu\qquad &\text{in \(\Omega\),}
\\
w_\mu=\varphi\qquad &\text{on \(\partial \Omega\),}
\end{cases}
\]
where  $\varphi$ is a smooth function such that $\varphi\equiv 1$ on $\partial \Omega\cap H$ and $\varphi\equiv 1/4$ on $D_1\cap \partial H$, and $\varphi\geq 1/4$ elsewhere on $\partial \Omega$ and $\mu>0$ is to be chosen. Note that $w_\mu$ is smooth on $\overline \Omega$ and that it converges in $C^2(\overline \Omega)$, as $\mu\to0$,  to the function $w_0$ such that $w_0=\varphi$ on $\partial\Omega$ and
$$
\Delta w_0(1+|\xi_0|^2)-D_{11}w_0|\xi_0|^2=0\qquad \text{in \(\Omega\),}
$$
 By the maximum principle there exists $\tau\in (0,1/2)$ such that $\max_{\overline D_1^+} w_0\leq 1-2\tau$. Therefore there exists $\mu_0>0$ such that
\begin{equation}\label{eq:tau}
\max_{\overline D_1^+} w_{\mu_0}\leq 1-\tau\,.
\end{equation}
We now claim that for $\eps\leq \eps_0$, with $\eps_0$ depending only on $w_0$, the function $v_\eps:=\eps w_{\mu_0}$ satisfies
\begin{multline*}
\Div \biggl(\frac{\nabla  v_\eps+\xi_0}{\sqrt{1+|\nabla  v_\eps+\xi_0|^2}}\biggr)\\
=\frac{\eps}{(1+|\eps\nabla w_{\mu_0}+\xi_0|^2)^{\frac32}}[\Delta w_{\mu_0} (1+|\eps\nabla w_{\mu_0}+\xi_0|^2)
-D^2w_{\mu_0}(\eps\nabla w_{\mu_0}+\xi_0)(\eps\nabla w_{\mu_0}+\xi_0)]\\
\leq \eps\Big(-\mu_0+C\eps(1+\|w_0\|_{C^2(\overline \Omega)}^3)\Big)<-\frac{\eps\mu_0}{2}< -2\eta_0\eps\,,
\end{multline*}
with $C>0$ universal, provided  $\eps_0$ and $\eta_0$ are chosen  small enough.

By  our assumptions, \(u^+_\sigma\le v_\eps\) on \(\partial \Omega\). Thus, recalling the first inequality in \eqref{irene-1} and the assumption on $\Lambda$, we get that 
$$
\Div \biggl(\frac{\nabla  u^+_\sigma+\xi_0}{\sqrt{1+|\nabla  u^+_\sigma+\xi_0|^2}}\biggr)\geq- 2\eta_0\eps
$$
in the viscosity sense. By the comparison principle we conclude that 
 \(u^+_\sigma\le v_\eps\) in \(\Omega\).  In turn, recalling \eqref{eq:tau}, we infer
 \[
\sup_{D_1^+} u^+_\sigma\le \sup_{D_1^+} v_\eps\le (1-\tau) \eps\,. 
\]
 This is in contradiction with \eqref{e:trieste} if \(\eta_0<\tau\).

Finally, if \eqref{118} holds, we may argue as before taking $\psi\equiv\eps$.
\end{proof}

\begin{lemma}\label{lm:savinglobal}
There exist \(\eps_1\), \(\eta_1\in (0,1/2)\) universal with the following property: if $E\subset H$ is a $(\Lambda,1)$-minimizer of $\Fc_\sigma$ in $\R^N$ with obstacle $O\in\mathcal B_R$,
 such that 
 \begin{equation}\label{savinglobal1}
 \partial E\cap \mathcal C_{2r}^+(x')\subset S_{a,b}\qquad \text{with }b-a\leq \eps_1r\,,
 \end{equation}
  \begin{equation}\label{rem:guido1bis}
 \Big\{x_N<\frac{\sigma x_1}{\sqrt{1-\sigma^2}}+a\Big\}\cap\mathcal C_{2r}^+(x')
\subset E\,,
\end{equation}
 for some $x'\in\R^{N-1}$ with $x_1\geq0$, $0<r\leq1$,
 $\Lambda r^2<\eta_1 (b-a)$ and 
$$
 \frac{r^2}{R}\leq \eta_1(b-a)\,,
$$
then there exist \(a'\ge a\), \(b'\le b\) with 
\[
b'-a'\le (1-\eta_1) (b-a)  
\]
 such that 
\[
\partial E\cap \mathcal \mathcal C^+_{\frac{r}{32}}(x')\subset S_{a', b'}.
\]
\end{lemma}
\begin{proof}
We start by observing that a simplified version of the same arguments in the proof of Lemma~\ref{lm:savin}, see Remark~\ref{rm:decay}, (or, alternatively,  the standard interior regularity results for $(\Lambda,r_0)$-minimizers of the perimeter) lead to the following interior version of the Lemma: there exist $\eps_1\leq\eps_0$, $\eta_1\leq\eta_0$ , such that if \eqref{savinglobal1} and \eqref{rem:guido1bis} hold with $x_1\geq2r$ and $\Lambda r^2<\eta_1 (b-a)$, then  there exist \(a'\ge a\), \(b'\le b\) with 
$
b'-a'\le (1-\eta_1) (b-a)  
$
 such that 
\begin{equation}\label{savinglobal2}
\partial E\cap \mathcal C_{\frac{r}2}(x')\subset S_{a', b'}.
\end{equation}
Therefore it is enough to prove the statement in the case $0< x_1<2r$. By rescaling, we may assume $r=1$ .

If $0< x_1\leq\frac{1}{8}$, we may apply Lemma~\ref{lm:savin} with the origin replaced by the point  $\overline x'=(0,x_2,\dots,x_N)$ and $r=\frac12$, provided $\eta_1$ is sufficiently small, thus getting $\partial E\cap \mathcal C^+_{\frac{1}{4}}(\bar x')\subset S_{a', b'}$, hence in particular $\partial E\cap \mathcal C^+_{\frac{1}{8}}(x')\subset S_{a', b'}$. 

Finally, if $x_1\geq1/8$ we just get the interior estimate \eqref{savinglobal2} with $r=\frac{1}{16}$.
\end{proof}

\begin{remark}\label{rm:savinglobal}
Let $\kappa>0$ be given. Observe that there exist  possibly smaller $\eps_1,\eta_1$, depending on $\kappa$, such that if $D_{2r}(x')\subset\{x_1>0\}$ and 
$$
 \partial E\cap \mathcal C_{2r}(x')\subset S_{a,b}^\alpha\qquad \text{with }b-a\leq \eps_1r\,,
$$
for some $\alpha\in(-\kappa,\kappa)$ and $\Lambda r^2<\eta_1 (b-a)$, then 
\[
\partial E\cap \mathcal C_{\frac{r}{32}}(x')\subset S_{a', b'}^\alpha,
\]
with \[
b'-a'\le (1-\eta_1) (b-a)\,.  
\] 
Indeed this follows with the same arguments of the proof of Lemma~\ref{lm:savin}, taking into account Remark~\ref{rm:decay} and Lemma~\ref{covid} which holds uniformly with respect to $\alpha\in[-\kappa,\kappa]$.
\end{remark}

\subsection{Flatness improvement and $\eps$-regularity} We start with the following crucial barrier argument, which forces the solution to coincide with the boundary of the obstacle if the solution is too close to a plane which does not satisfy the exact optimality condition, compare with \cite{Fernandez-RealSerra20} for a similar argument.

\begin{lemma}\label{lm:barrier}
 There exist universal constants \(\eps_2, M_0>0\) with the following property. Assume that $E\subset H$ is a $(\Lambda,1)$-minimizer of $\Fc_\sigma$ in $\R^N$ with obstacle $O\in\mathcal B_R$. Assume also  that 
 \[
\overline{\partial E \cap H} \cap \mathcal C_{2r}\subset \{ \abs{x_N-\alpha x_1}< \eps r\big\},\quad  \{x_N<\alpha x_1-\eps r\}\cap\mathcal C_{2r}^+
\subset E\,,
\]
for some  \(0<\eps \le \eps_2\),  $0<r\leq1$ with $\Lambda r<\eps$ and $\frac{r}{R}\leq\eta(1)\eps$ (where $\eta(1)$ is as in Lemma~\ref{apparte}), and for some $\alpha\ge M_0\eps+\frac{\sigma}{\sqrt{1-\sigma^2}}$. Then 
\[
\overline{\partial E \cap H}\cap \partial H\cap\mathcal C_{r/2}=\partial_{\partial H}O\cap\mathcal C_{r/2}\cap\{|x_N|<\eps r\}.
\]
\end{lemma}

\begin{proof} Without loss of generality, by a rescaling argument, we may assume $r=1$.

We start by assuming that there exists $\bar y'\in\partial_{\partial H}O\cap\mathcal C_1\cap S_{-\eps,\eps}$. From the assumption on $O$ we get, thanks to Lemma~\ref{apparte} that  the connected component of $\partial_{\partial H}O\cap\mathcal C_{2}$ containing $\bar y'$ is a graph of a $C^{1,1}$ function $\psi$ such that
$\|D\psi\|_{L^\infty(\partial H\cap D_{2})}\leq\eps$, $\|D^2\psi\|_{L^\infty(\partial H\cap D_{2})}\leq\eps$.
In particular
\begin{equation}\label{barrier0.5}
\|\psi\|_{L^\infty(\partial H\cap D_{1})}\leq3\eps\,.
\end{equation}
As said, the proof is obtained via a barrier construction. To this end we fix $z'=(0,z_2,\dots,z_{N-1})$ with $|z'|\leq1/2$ and assume by contradiction that $u^-(z')<\psi(z')$.  We define for $x'\in\R^{N-1}$
\[
w(x')=\beta x_1+\gamma\eps (L x_1^2-|x'-z'|^2)+\psi(z')+\nabla\psi(z')\cdot(x'-z')+c
\]
where 
\[
\beta= \frac{M_0\eps}{2}+\frac{\sigma}{\sqrt{1-\sigma^2}}\,,
\]
$L$ and $\gamma$ are positive constants to be chosen 
and $c$ is the first constant such that the function $w$ touches from below  $u^-(x')$ in $\overline D_{1/2}^+(z')$.
 Observe that $c\leq0$, since $ w(z')=\psi(z')+c\leq u^-(z')\leq\psi(z')$, and that
\begin{equation}\label{barrier1-}
\frac1\eps\Div \biggl(\frac{\nabla w}{\sqrt{1+|\nabla w|^2}}\biggr)\to 2\gamma\frac{(L-N+1)(1+|\xi_0|^2)-(L-1)|\xi_0|^2}{(1+|\xi_0|^2)^{3/2}}
\end{equation}
uniformly in $D_2$ as $\eps\to0$, where $\xi_0=\sigma e_1/\sqrt{1-\sigma^2}$. Therefore, if $L$ is chosen so large that the right hand side of \eqref{barrier1-} is bigger than $2$,  there exists $\eps_2$ depending on $M_0$ such that if $0<\eps<\eps_2$ then 
\begin{equation}\label{barrier1}
\Div \biggl(\frac{\nabla w}{\sqrt{1+|\nabla w|^2}}\biggr)>\eps\,.
\end{equation}
Observe that if $x'\in\partial D_{1/2}(z')\cap\{x_1\geq0\}$, then
$$
w(x')\leq\beta x_1+\frac{\gamma L\eps}{2} x_1-\frac{\gamma}{4}\eps+\psi(z')+\nabla\psi(z')\cdot(x'-z')< \alpha x_1-\eps\leq u^-(x)\,,
$$
provided $M_0>\gamma L$ and $\gamma\geq18$, where in the last inequality we used \eqref{barrier0.5}.  We claim that $w$ does not touch $u^-$ at a point $x_0\in D^+_{\frac12}(z')$. Indeed if this is the case,  by Lemma~\ref{pace}, recalling that  $\Lambda$ by assumption is smaller than $\eps$ we would get that
$$
\Div \biggl(\frac{\nabla w}{\sqrt{1+|\nabla w|^2}}\biggr)(x_0)\leq\eps
$$
a contradiction to \eqref{barrier1}.
Thus we may conclude that $w$ touches $u^-$ at a point $\overline x'\in\{x_1=0\}\cap\partial D^+_{\frac12}(z')$. Assume by contradiction that $u^-(\overline x')<\psi(\overline x')$, then from the last inequality in \eqref{irene-1}, observing that 
$$
|\nabla w(\overline x')-\partial_1w(\overline x') \,e_1|\leq C\eps\,,
$$
for a constant depending only on $N$, we have
$$
\partial_1w(\overline x')\leq \frac{\sigma}{\sqrt{1-\sigma^2}}+\frac{\sigma^+}{\sqrt{1-\sigma^2}}\frac{C^2\eps^2}{2}\,.
$$
which is impossible, provided that $M_0$ is sufficiently large, since 
$$
\partial_1w(\overline x')=\frac{M_0\eps}{2}+\frac{\sigma}{\sqrt{1-\sigma^2}}\,.
$$
Therefore, at the touching point we have
$$
0=u^-(\overline x')-w(\overline x')=\psi(\overline x')-w(\overline x')\,,
$$
hence
$$
\psi(\overline x')=\psi(z')+\nabla\psi(z')\cdot(\overline x'-z')+c-\gamma\eps|\overline x'-z'|^2\,.
$$
On the other hand, recalling that $|D^2\psi|\leq \eps$, we have
$$
\psi(\overline x')\geq\psi(z')+\nabla\psi(z')\cdot(\overline x'-z')-\eps|\overline x'-z'|^2\,.
$$
Combining the two inequalities we get  that $0\geq c\geq(\gamma-1)\eps|\overline x'-z'|^2$. Therefore $\overline x'=z'$ hence $\psi(z')=u^-(z')$.

If instead $\partial_{\partial H}O\cap\mathcal C_1\cap S_{-\eps,\eps}=\emptyset$, we may repeat the same argument as before with $\psi$ replaced by $\tilde\psi\equiv\eps$, obtaining that $u^\pm=\tilde\psi=\eps$ in $\partial H\cap\mathcal C_{1/2}$ which is impossible by assumption.
\end{proof}
\begin{remark}\label{rm:barrier}
Note that the assumptions of the previous lemma force the obstacle $\psi$ to satisfy the inequality
$|\psi|<\eps r $ in $\partial H\cap \mathcal C_{r/2}$.
\end{remark}

%\begin{corollary}\label{cor:savin}
%Let \(\eps_0\), \(\eta_0\in (0,1/2)\) be as in Lemma~\ref{lm:savin}. If $E\subset H$ is a $(\Lambda,r_0)$-minimizer of $\Fc_\sigma$ in $\R^N$ with obstacle $O$
% such that, for $0<r\leq1$
% \begin{equation}\label{savin11}
% \partial E\cap C_{r}^+\subset S_{a,b}\qquad \text{with }2(b-a)\leq \eps_0 r\,,
% \end{equation}
% $\Lambda r^2<4\eta_0(b-a)$ and 
%$$
%O\cap S_{a, b}=\{x\in S_{a, b}:\, x_1=0,\, x_N<\psi(x_2, \dots, x_{N-1})\}\,,
%$$
%for some $\psi\in C^{2}(\R^{N-2})$, with $r\|\nabla\psi\|_\infty\leq 2\eta_0(b-a)$ and $r^3\|D^2\psi\|_\infty\leq 8\eta^2_0(b-a)^2$, 
%then there exist \(a'\ge a\), \(b'\le b\) with 
%\[
%b'-a'\le (1-\eta_0) (b-a)  
%\]
% such that 
%\[
%\partial E\cap C^+_{\frac{r}{4}}\subset S_{a', b'}.
%\]
%\end{corollary}

\begin{lemma}\label{itera}
For all \(\tau\in (0,1/2)\), \(M>0\),  there exist constants \(\lambda_0=\lambda_0(M,\tau)>0\), \(C_2=C_2(M,\tau)>0\) such that for all \(\lambda\in (0, \lambda_0)\) one can find  \( \eps_3=\eps_3(M,\tau,  \lambda)>0\),  \( \eta_2=\eta_2(M,\tau,  \lambda)>0\) with the following property: Assume  $E\subset H$ is a $(\Lambda,1)$-minimizer of $\Fc_\sigma$ in $\R^N$ with obstacle $O\in\mathcal B_R$. Assume also that $0\in\overline{\partial E\cap H}$ and
$$
\overline{\partial E\cap H} \cap \mathcal C_r\subset \{ \abs{x_N-\alpha x_1}< \eps r\big\}, \quad  \{x_N<\alpha x_1-\eps r\}\cap\mathcal C_{r}^+\subset E\,,
$$
for some \(0<\eps \le \eps_3\), $0<r\leq1$ with $\Lambda r<\eta_2\eps $ and $\frac{r}{R}\leq \eta_2\eps$, and for some
 \(\frac{\sigma}{\sqrt{1-\sigma^2}}\le \alpha \le M \eps +\frac{\sigma}{\sqrt{1-\sigma^2}}\).
%$$
%O\cap S_{-\eps r, \eps r}^\alpha=\{x\in S_{-\eps r, \eps r}^\alpha:\, x_1=0,\, x_N<\psi(x_2, \dots, x_{N-1})\}\,,
%$$
%for some $\psi\in C^{1,1}(\R^{N-2})$ with  
%$$
%r\Div \biggl(\frac{\nabla \psi+\xi_0}{\sqrt{1+|\nabla \psi+\xi_0|^2}}\biggl)+r\|D^2\psi\|_\infty\leq \eta_2\eps\,,
%$$
 Then  there exist \(\bar \alpha\ge\frac{\sigma}{\sqrt{1-\sigma^2}} \), a rotation \(R\) about the \(x_1\) axis,  with  \(\|R-\Id\|\le C_2\eps\),  \(|\bar \alpha-\alpha|\le C_2 \eps\),  such that 
\[
R( \overline{\partial E\cap H}) \cap\mathcal C_{\lambda r}\subset \{ \abs{x_N-\bar\alpha x_1}< \lambda^{1+\tau} \eps r\big\}.
\]
\end{lemma}

\begin{proof}  By rescaling it is enough to show the statement for $r=1$.
We argue by contradiction and we assume that there exist  \(\tau \) and \(M\)  and sequences \(\epsilon_n\to 0\),  \(\frac{\sigma}{\sqrt{1-\sigma^2}}\le \alpha_n \le M\epsilon_n+\frac{\sigma}{\sqrt{1-\sigma^2}}\),  $(\Lambda_n,1)$-minimizers \(E_n\) with $0\in\overline{\partial E_n\cap H}$, $\overline{\partial E_n\cap H} \cap \mathcal C_1\subset \{ \abs{x_N-\alpha_n x_1}< \epsilon_n\big\}$ and $\{x_N<\alpha x_1-\eps r\}\cap\mathcal C_{r}^+\subset E$, $\Lambda_n$, $O_n\in\mathcal B_{R_n}$, such that 
 $\Lambda_n\leq\frac{\epsilon_n}{n}$, $\frac{1}{R_n}\leq\frac{\epsilon_n}{n}$  for which the conclusion fails for all \(\bar \alpha \), \(R\)   with \(|\bar \alpha-\alpha_n|\le C_2 \epsilon_n\), \(\|R-\Id\|\le C_2\epsilon_n\),  and for a certain  \(\lambda \le \lambda_0\). We will show that for a suitable choice of  \(C_2\)  and \(\lambda_0\)  depending only on \(M\), \(\tau\) and the dimension this will lead to a contradiction.
 
 We claim that there exist universal constants $C>0$, $\beta\in(0,1]$ and $\gamma_n>0$, $\gamma_n\to0$, such that for all $x',y'\in \overline D^+_{\frac12}$, with $|x'-y'|\geq\gamma_n$,
 \begin{equation}\label{itera0}
 |u_{n,\sigma}^{\pm}(x')-u_{n,\sigma}^{\pm}(y')|\leq C\epsilon_n|x-y|^\beta,\qquad |u_{n,\sigma}^{\pm}(x')-u_{n}^{\mp}(y')|\leq C\epsilon_n|x-y|^\beta\,,
 \end{equation}
 where  $u^{\pm}_{n,\sigma}$ are the functions defined as $u^\pm_{\sigma}$  with $E$ replaced by $E_n$.
 
 To prove the claim we  fix  $x'\in\overline D^+_{\frac12}$. Note that our assumptions imply that if $0<\varrho\leq
\frac12$
\[
\overline{\partial E_n\cap H} \cap\mathcal C_{\varrho}(x')\subset S_{a,b}\,,
\]
with $b-a= 2(1+M)\epsilon_n$.
 We now apply Lemma~\ref{lm:savinglobal}  to conclude that for all $y'\in \overline D^+_{\frac{\varrho}{64}}(x')$
\begin{equation}\label{itera1}
|u^{\pm}_{n,\sigma}(x')-u^{\pm}_{n,\sigma}(y')| <2(1+M)(1-\eta_1)\epsilon_n,\,\,\,\,\,|u^{\pm}_{n,\sigma}(x')-u^{\mp}_{n,\sigma}(y')|< 2(1+M)(1-\eta_1)\epsilon_n,
\end{equation}
provided that
\begin{equation}\label{itera2}
4(1+M)\epsilon_n\leq \eps_1 \varrho\quad\text{and}\quad   \frac{\varrho^2}{n}\leq8\eta_1(1+M)\,.
\end{equation}
Thus,    inequalities \eqref{itera1} hold for $n$ sufficiently large.
Note that \eqref{itera1} implies that 
$$
\overline{\partial E_n\cap H} \cap \mathcal C_{\frac{\varrho}{64}}(x')\subset  S_{a',b'}\quad\text{with \,\,$b'-a'\leq 2(1+M)(1-\eta_1)\epsilon_n$.}
$$
Therefore, by applying Lemma~\ref{lm:savinglobal} $m_n-1$ more times   we conclude that for all $y'\in \overline D^+_{\frac{\varrho}{64^{j+1}}}(x')$, $j=0,1,\dots,m_n-1$,
\begin{equation}\label{itera4}
\begin{split}
&|u^{\pm}_{n,\sigma}(x')-u^{\pm}_{n,\sigma}(y')|< 2(1+M)(1-\eta_1)^{j+1}\epsilon_n,  \\
&|u^{\pm}_{n,\sigma}(x')-u^{\mp}_{n,\sigma}(y')|< 2(1+M)(1-\eta_1)^{j+1}\epsilon_n,
\end{split}
\end{equation}
provided that
$$
4\cdot64^{j}(1+M)(1-\eta_1)^j\epsilon_n\leq\eps_1\varrho, \quad \frac{\varrho^2}{64^{2j}n}\leq8\eta_1(1-\eta_1)^{j}(1+M)\,.
$$
Note that the last  inequality follows from the second one in \eqref{itera2}, since $\eta_1<\frac12$, while the first holds provided $m_n$ is the largest integer for which
$$
4\cdot64^{m_n-1}(1+M)(1-\eta_1)^{m_n-1}\epsilon_n\leq\eps_1\varrho\,.
$$
In particular, taking $\varrho=\frac12$ and  applying \eqref{itera4} with $y\in\overline D^+_{\frac{\varrho}{64^j}}(x')\setminus D^+_{\frac{\varrho}{64^{j+1}}}(x')$  and $j=1,\dots m_n-1$   we easily get  that if $y'\in \overline D^+_{1/128}(x')\setminus D^+_{\gamma_n}(x')$, with $\gamma_n=\frac{1}{2\cdot64^{m_n}}$, then
$$
|u^{\pm}_{n,\sigma}(x')-u^{\pm}_{n,\sigma}(y')|\leq C\epsilon_n|x'-y'|^\beta,\qquad |u^{\pm}_{n,\sigma}(x')-u^{\mp}_{n,\sigma}(y')|\leq C\epsilon_n|x'-y'|^\beta\,,
$$
for a constant $C$ depending only on $M$  and $\beta=-\frac{\log(1-\eta_1)}{\log64}$. This proves \eqref{itera0} when $x'\in\overline D^+_{\frac12}$ and $y'\in D^+_{\frac{1}{128}}(x')$, with $|x'-y'|\geq\gamma_n$. Clearly, $\gamma_n\to0$, since $m_n\to\infty$. 

In particular if we consider the functions 
 \[
 v^{\pm}_{n}(x')=\frac{u^{\pm}_{n,\sigma}(x')-\Big(\alpha_n-\frac{\sigma}{\sqrt{1-\sigma^2}}\Big) x_1}{\epsilon_n}=\frac{u^{\pm}_{n}(x')-\alpha_nx_1}{\epsilon_n}
 \]
they converge (up to a subsequence)  to the same  H\"older continuous function \(v\) defined on \(D_{1/2}^+\) (note that \(\big|\alpha_n-\frac{\sigma}{\sqrt{1-\sigma^2}}\big|\le M\epsilon_n\)). Furthermore \(\|v\|_{\infty}\le 1\) and $v(0)=0$, since the assumption $0\in\overline{\partial E_n\cap H}$ implies that $u^{-}_{n,\sigma}(0)\leq0\leq u^{+}_{n,\sigma}(0)$. We also assume that (up to further subsequence), 
\beq\label{gamma1000}
\frac{1}{\epsilon_n}\Big(\alpha_n-\frac{\sigma}{\sqrt{1-\sigma^2}}\Big) \to\gamma \in [0,M]\,.
\eeq

We now consider two possible cases. Let us assume first that for all $n$ there exists $y'_n\in \partial_{\partial H}O_n\cap\mathcal C_{\frac12}\cap\{-\epsilon_n<x_N<\epsilon_n\}$. Then by Lemma~\ref{apparte},  we have that for $n$ large the connected component of $\partial_{\partial H}O_n\cap\mathcal C_{1}$ containing $y_n'$ is a graph of a $C^{1,1}$ function $\psi_n$ such that
$$
\|D\psi_n\|_{L^\infty(\partial H\cap D_{1})}\leq2\eps_n,\quad\|D^2\psi_n\|_{L^\infty(\partial H\cap D_{1})}\leq\delta_n\eps_n\,
$$
for a suitable sequence $\delta_n$ converging to zero.
Recall that
$$
\frac{\psi_n(y'_n)}{\epsilon_n}\in(-1,1)\,.
$$
Therefore we may assume that, up to a subsequence, 
$$
\frac{\psi_n(x')}{\epsilon_n}\to\psi_\infty+\omega\cdot x'\quad\text{uniformly with respect to $x'\in D_{1}\cap\partial H$}
$$
for some $\psi_\infty\in[-2,2]$ and  $\omega\in\R^{N-1}$ of the form $(0,\omega_2,\dots,\omega_{N-1})$ with $|\omega|\leq2.$ 

We now claim that \(v-\omega\cdot x'\) satisfies 
\begin{equation}\label{itera5}
\begin{cases}
Lw:=\Delta w(1+|\xi_0|^2)-D_{11}w|\xi_0|^2 =0\qquad &\text{in \(D_{1/2}^+\)}
\\
w\le \psi_\infty  &\text{on \(D_{1/2}\cap\{x_1=0\}\)}
\\
\partial_1 w\ge -\gamma &\text{on \(D_{1/2}\cap\{x_1=0\}\)}
\\
(\psi_\infty-w)(\partial_1w+\gamma)=0 &\text{on \(D_{1/2}\cap\{x_1=0\}\)}
\end{cases}
\end{equation}
in the viscosity sense. To this aim observe that the functions $u^\pm_{n}$, are subsolutions, respectively supersolutions, of 
\[
\displaystyle\Div \Biggl(\frac{\nabla w}{\sqrt{1+|\nabla  w|^2}}\Biggr) =\mp\frac{\epsilon_n}{n}
\]
thanks to Lemma~\ref{pace}. Therefore $v_{n}^-$ and $v_{n}^+ $ are supersolutions, respectively subsolutions, of 
$$
 L_nw:=\Delta w-\dfrac{D^2 w (\alpha_n e_1+\epsilon_n \nabla w)(\alpha_n e_1+\epsilon_n \nabla w)}{1+|\alpha_n e_1+\epsilon_n \nabla w|^2}=\pm\frac1n\sqrt{1+|\alpha_n e_1+\epsilon_n \nabla w|^2}\,.
$$
%\partial_1 v^-_{n,\sigma}\leq-\frac{1}{\epsilon_n}\Big(\alpha_n-\frac{\sigma}{\sqrt{1-\sigma^2}}\Big)+\frac{\sigma^+}{\sqrt{1-\sigma^2}}\frac{n^2\epsilon_n}{2}\,.
Passing to the limit in the previous equation, recalling that $v_{n}^\pm$ converge to $v$ uniformly in $D^+_{\frac12}$ and using the stability of viscosity super- and subsolutions, we get that $v$, hence $v-\omega\cdot x'$, is a viscosity solution of the first equation in \eqref{itera5}. Recalling that $v^{\pm}_{n}\leq\frac{\psi_n}{\epsilon_n}$ on $D_{\frac12}\cap\partial H$, we get that $v(x')\leq\psi_\infty+\omega\cdot x'$ for all $x'\in  D_{\frac12}\cap\partial H$, hence the second inequality in \eqref{itera5} follows for $v-\langle\omega,\cdot\rangle$. 

To prove the third inequality we take a $C^2$ test function $\varphi$ touching $v$ from above in $D^+_{\frac12}$ at a  point $\overline x'\in D_{\frac12}\cap\{x_1=0\}$. Without loss of generality we may assume that $\overline x'$ is the unique touching point. We argue by contradiction assuming that $\partial_1\varphi(\overline x')<-\gamma$.  If this is the case the function
\begin{equation}\label{itera7}
\tilde\varphi(x')=-\tilde\gamma x_1- bx_1^2+\varphi(0,x_2,\dots,x_{N-1}),\qquad\text{with $\tilde\gamma\in(\gamma,-\partial_1\varphi(\overline x'))$}
\end{equation}
and $b>0$ to be chosen,
stays above $\varphi$ and hence above $v$ in a neighborhood of $\overline x'$. In particular $\overline x'$ is the unique touching point between $\tilde\varphi$ and $v$ in such a neighborhood. Therefore there exists a sequence $x_n'\in\overline D^+_{\frac12}$ converging to $\overline x'$ such that $x'_n$ is a local maximizer of $v_{n}^+-\tilde\varphi$. If $x_n'\in D^+_{\frac12}$ for infinitely many $n$, thus recalling the subsolution property of $v^+_n$ we have  
$$
L_n\tilde\varphi(x_n')\geq-\frac1n\sqrt{1+|\alpha_n e_1+\epsilon_n \nabla\tilde\varphi(x_n')|^2}. 
$$
Thus, passing to the limit, $L\tilde\varphi(\overline x')\geq0$ which is impossible if we choose $b$ so large that  $L\tilde\varphi(\overline x')<0$. Otherwise,  $x_n'\in D_{\frac12}\cap\{x_1=0\}$ for infinitely many $n$. In particular for all such $n$ the function $u^+_n-(\epsilon_n\tilde\varphi+\alpha_nx_1)$ has a local maximum in $x'_n$. Hence by the third inequality in \eqref{irene-1}   we have
$$
\epsilon_n\partial_1\tilde\varphi(x'_n)+\alpha_n\geq \frac{\sigma}{\sqrt{1-\sigma^2}}-\frac{\sigma^-}{\sqrt{1-\sigma^2}}\frac{\epsilon_n^2\|\nabla\tilde\varphi\|^2_{L^\infty(D^+_1)}}{2}\,.
$$
Dividing the previous inequality by $\epsilon_n$ and recalling the definition of $\gamma$ we have
$$
-\tilde\gamma+\gamma\geq0\,,
$$
which is a contradiction to \eqref{itera7}. This shows the third inequality in \eqref{itera5}.

To show the last equality we take a test function $\varphi$ such that $v-\langle\omega,\cdot\rangle-\varphi$ has a strict local minimum at a point $\overline x'\in  D_{\frac12}\cap\{x_1=0\}$ such that $v(\overline x')-\omega\cdot\overline x'<\psi_\infty$. Then, arguing by contradiction as before and using $u^-_n$ and the fourth inequality in \eqref{irene-1} in place of $u^+_n$ and the third inequality in \eqref{irene-1} we infer that  $\partial_1\varphi(\overline x')\leq-\gamma$, thus getting also the last equality in \eqref{itera5} in the viscosity sense.

Thus,  the function $\overline w(x'):=\psi_\infty-v(x')+\omega\cdot x'-\gamma x_1$ solves the following Signorini type problem
\[
\begin{cases}
\Delta\overline w(1+|\xi_0|^2)-D_{11}\overline w|\xi_0|^2 =0\qquad &\text{in \(D_{1/2}^+\)}
\\
\overline w\ge 0 &\text{on \(D_{1/2}\cap\{x_1=0\}\)}
\\
\partial_1\overline w\le0  &\text{on \(D_{1/2}\cap\{x_1=0\}\)}
\\
\overline w\,\partial_1\overline w=0 &\text{on \(D_{1/2}\cap\{x_1=0\}\)}
\end{cases}
\]
in the viscosity sense. 
In particular by the regularity estimates  proved in \cite{Misi} (see also \cite{Guillen09}) we infer that there exist a universal constant \({C}\), such that for all \(\lambda < 1/4\)
\[
\sup_{D^+_{\lambda}} \abs[\big]{v(x')-\nabla v(0) \cdot x'}=\sup_{D^+_{\lambda}} \abs[\big]{\overline w(x')-\overline w(0)-\nabla \overline w(0) \cdot x'}\le {C} \lambda^{\frac{3}{2}}\|\overline w\|_{L^\infty(D^+_{\frac12})}\le C(1+M) \lambda^{\frac{3}{2}}
\]
and 
\[
\abs{\nabla v(0)}\le  C\|\overline w\|_{L^\infty(D^+_{\frac12})}\leq C (1+M)\,.
\]
We first choose \(\lambda_0\) so that \(C(1+M) \lambda^{\frac{3}{2}}<\frac14\lambda^{1+\tau}\) for all \(\lambda\le \lambda_0\). Therefore, by the above estimate and uniform convergence of $v_{n}^\pm$ to $v$ we get for $n$ large, recalling that $v(0)=0$,
\begin{align*}
& (\epsilon_n\partial_1v(0)+\alpha_n)x_1+\epsilon_n(\nabla v(0)-\partial_1v(0)e_1)\cdot x'-\frac14\lambda^{1+\tau}\epsilon_n
\\
&\quad <
u^-_n(x')\leq u^+_n(x')< (\epsilon_n\partial_1v(0)+\alpha_n)x_1+\epsilon_n(\nabla v(0)-\partial_1v(0)e_1)\cdot x'+\frac14\lambda^{1+\tau}\epsilon_n\,.
\end{align*}
In particular, setting 
\[
 \bar \alpha_n=\epsilon_n\Big(\partial_1 v( 0)+\frac{\lambda^{1+\tau}}{4}\Big) +\alpha_n
\]
and
\[
 \boldsymbol v_n=\epsilon_n(\nabla v (0)-\partial_1 v_1(0) e_1)
\]
 $\partial E_n\cap\mathcal C_\lambda$ is contained in the strip 
$$
S_n:=\Big\{|x_N-\bar\alpha_nx_1-\boldsymbol v_n\cdot x'|<\frac12\lambda^{1+\tau}\epsilon_n\Big\}\,.
$$
Note that, since by \eqref{itera5} $\partial_1v(0)\geq-\gamma$, recalling \eqref{gamma1000}, we have that for $n$ sufficiently large  $\bar\alpha_n>\frac{\sigma}{\sqrt{1-\sigma^2}}$. 
Hence, if  \(R_n\) is the rotation around  \(x_1\) which maps the vector
\[
\frac{e_N-\boldsymbol v_n}{\abs{e_N-\boldsymbol v_n}}
\]
in $e_N$, we conclude that
$$
R_n(\overline{\partial E_n\cap H})\cap \mathcal C_\lambda\subset  \{ \abs{x_N-\bar\alpha_n x_1}< \lambda^{1+\tau} \epsilon_n\big\}\,,
$$
thus giving a contradiction in this case.

If instead for infinitely many $n$ we have that $\partial_{\partial H}O_n\cap\mathcal C_{\frac12}\cap\{-\epsilon_n<x_N<\epsilon_n\}=\emptyset$, 
 then reasoning as above we may infer that the function $\tilde w(x'):=v(x')+\gamma x_1$ solves the Neumann problem\[
\begin{cases}
\Delta \tilde w(1+|\xi_0|^2)-D_{11}\tilde w|\xi_0|^2 =0\qquad &\text{in \(D_{1/2}^+\)}
\\
\partial_1\tilde w=0  &\text{on \(D_{1/2}\cap\{x_1=0\}\)}\,.
\end{cases}
\]

The same argument, now using the more standard elliptic estimates for the Neumann problem, leads to a contradiction also in this case.
\end{proof}

 Next Lemma is the interior case of the above estimate and the proof  follows from the interior version of the previous arguments taking into account Remark~\ref{rm:savinglobal}.

\begin{lemma}\label{iteraintern}
For all \(\tau\in (0,1)\), $\kappa>0$ there exist constants \(\lambda_1=\lambda_1(\tau, \kappa)>0\), \(C_3=C_3(\tau, \kappa)>0\) such that for all \(\lambda\in (0, \lambda_1)\) there exist  \( \eps_4=\eps_4(\tau,\kappa, \lambda)>0\),  \( \eta_3=\eta_3(\tau,\kappa,   \lambda)>0\) with the following property. If $E\subset H$ is a $(\Lambda,1)$-minimizer of $\Fc_\sigma$ in $\R^N$ with obstacle $O$ such that if $\bar x\in\partial E$,  $D_r(\bar x')\subset\{x_1>0\}$, $0<r\leq1$,
\[
\partial (E-\bar x) \cap \mathcal C_r\subset \{ \abs{x_N-\alpha x_1}< \eps r\big\}\,,
\]
with \(\eps \le \eps_4\), for some $\alpha\in[-\kappa,\kappa]$,
 $\Lambda r<\eta_3\eps$, 
 then  there exist \(\bar \alpha\in\R \), a rotation \(R\)  about the $x_1$ axis,  with  \(\|R-\Id\|\le C_3\eps\),  \(|\bar \alpha-\alpha|\le C_3 \eps\),  such that 
\[
\partial R(E-\bar x) \cap \mathcal C_{\lambda r}\subset \{ \abs{x_N-\bar\alpha x_1}< \lambda^{1+\tau} \eps r\big\}.
\]
\end{lemma}

 Next Lemma deals with the case when  \(\partial_{\partial_{H}}E\) fully coincides with the obstacle and it is a simple application of the boundary regularity result. This would be needed for the situations where Lemma \ref{lm:barrier} applies.
\begin{lemma}\label{iteraostac}
For all \(\tau\in (0,1)\), $\kappa>0$,  there exist constants \(\lambda_2=\lambda_2(\tau, \kappa)>0\), \(C_4=C_4(\kappa)>0\) such that for all \(\lambda\in (0, \lambda_2)\) one can find  \( \eps_5=\eps_5(\tau, \kappa, \lambda)>0\),   with the following property:  Assume that $E\subset H$ is a $(\Lambda,1)$-minimizer of $\Fc_\sigma$ in $\R^N$ with obstacle $O\in\mathcal B_R$. Assume also that 
$\bar x\in\partial H\cap\overline{\partial E\cap H}$ and
\[
(\partial E-\bar x) \cap \mathcal C_r^+\subset \{ \abs{x_N-\alpha x_1}< \eps r\big\}, \quad  \{x_N<\alpha x_1-\eps r\}\cap\mathcal C_{r}^+\subset E-\bar x\,,
\]
for some \(\eps \le \eps_5\), 
 $0<r\leq 1$ with $\Lambda r<\eps $ and  $\frac{r}{R}<\eta(1)\eps$  (where $\eta(1)$ is the constant provided in Lemma~\ref{apparte}),  and for some $ \alpha \in[-\kappa,\kappa]$. Finally assume that 
$$
\partial H\cap\overline{(\partial E-\bar x)\cap H}\cap\mathcal C_r= \partial_{\partial H}(O-\bar x)\cap\mathcal C_r\cap\{-\eps r<x_N<\eps r\}\,.
$$
Then  there exist \(\bar \alpha\geq\frac{\sigma}{\sqrt{1-\sigma^2}} \), a rotation \(R\) about the \(x_1\) axis,  with  \(\|R-\Id\|\le C_4\eps\),  \(|\bar \alpha-\alpha|\le C_4 \eps\),  such that 
\[
 R( \partial E-\bar x) \cap\mathcal C_{\lambda r}^+\subset \{ \abs{x_N-\bar\alpha x_1}< \lambda^{1+\tau} \eps r\big\}.
\]
\end{lemma}

\begin{proof}
By rescaling and translating we may assume $r=1$ and $\bar x=0$. Denote by $\mathcal E$ the cylindrical excess
$$
\mathcal E(E;r)=\frac{1}{r^{N-1}}\int_{\partial E\cap \mathcal C_r^+}|\nu_E(x)-\nu_\alpha|^2\,d\H^{N-1}(x)\,,
$$
where $\nu_\alpha$ is the normal to the hyperplane $\{x_N-\alpha x_1=0\}$ pointing upward. Fix $\delta\in(0,1)$. We claim that there exists $\eps_5$ such that if the assumptions are satisfied  for $r=1$ then
\begin{equation}\label{iteraostac1}
\mathcal E\Big(E;\frac12\Big)\leq\delta\,.
\end{equation}
To prove this claim we observe that if $E_n$ is a sequence of $(\Lambda_n,1)$-minimizers satisfying the assumptions with $\eps=\eps_n\to0$, then by Theorem~\ref{th:compactness} we have that, up to subsequence, $|(E_n\Delta F)\cap \mathcal C^+_{\frac12}|\to0$, $|\mu_{E_n}|\res \mathcal C^+_{\frac12}\wtos |\mu_F|\res \mathcal C^+_{\frac12}$, where  $F=\{x_N-\alpha x_1\leq0\}$. Thus, 
$$
\mathcal E\Big(E_n;\frac12\Big)\to \mathcal E\Big(F;\frac12\Big)=0\,.
$$
Hence \eqref{iteraostac1} follows by a compactness argument. 

We recall that, thanks to Lemma~\ref{apparte},  there exists $\eta(1)>0$ such that if $\frac{1}{R}<\eta(1)\eps$,  $\partial_{\partial H}O\cap \mathcal C_1$ is the graph of a function $\psi\in C^{1,1}(D_1)$ such that
$$
\|\nabla\psi\|_{L^\infty(\partial H\cap D_1)}\leq2\eps,\quad \|D^2\psi\|_{L^\infty(\partial H\cap D_1)}\leq2\eps\,.
$$
Observe that,  there exists $\mu(\kappa)\in(0,1)$ such that if $\Sigma\subset\mathcal C^+_{\frac14}$ can be described as a graph over $\{x_N-\alpha x_1=0\}$ of a $C^1$ function $g$ such that $\|\nabla g\|_{\infty}\leq\mu(\kappa)$, then
$\Sigma$ can be also written as a graph over $D^+_{\frac14}$ of a $C^1$ function $f$ with $\|\nabla f\|_{\infty}\leq2\kappa$.

By   \cite[Theorem~6.1]{DuzaarSteffen02}\footnote{Theorem 6.1 in \cite{DuzaarSteffen02} is stated and proved for almost minimizing currents. However it is well known to the experts that the methods of the proof extend without significant changes also to the framework of almost minimizing sets of finite perimeter.}  we have that there exists $\bar\delta=\bar\delta(\kappa)$ such that if 
$$
\mathcal E\Big(E;\frac12\Big)+\Lambda+ \|D^2\psi\|_\infty \leq \bar\delta\,,
$$
then $\partial E\cap \mathcal C^+_{\frac14}$ is a graph with respect to the hyperplane $\{x_N-\alpha x_1=0\}$ of a  $C^{1,\gamma}$ function $g$ with $\|\nabla g\|_{C^{0,\gamma}}\leq\mu(\kappa)$, for some universal $\gamma>0$.
In particular $\partial E\cap \mathcal C^+_{\frac14}$ is the graph over $D^+_{\frac14}$ of a $C^{1,\gamma}$ function $f$ with 
$\|\nabla f\|_{C^{0,\gamma}}\leq C(\kappa)$. Then, observing that the function $f(x')-\alpha x_1$ is a solution of
$$
\Div \biggl(\frac{\nabla w+\alpha e_1}{\sqrt{1+|\nabla w+\alpha e_1|^2}}\biggr) =h\,
$$
with $|h|\leq\Lambda$, 
standard regularity estimates for solutions of the mean curvature equation imply that for all $s\in(0,1)$ there exists a constant $C_{s,\kappa}$, depending only on $s$ and $\kappa$ such that
\begin{equation}\label{iteraostac2}
\|f-\alpha x_1\|_{C^{1,s}(D^+_{\frac18})}\leq C_{s,\kappa}\big(\|f-\alpha x_1\|_{L^\infty(D^+_{\frac14})}+\|\psi\|_{C^{1,s}(D^+_{\frac14})}\big)\leq C'_{s,\kappa}\eps\,,
\end{equation}
where the last inequality follows from the fact that $\|\psi\|_{C^{1,1}(D^+_{\frac14})}\leq C(\|D^2\psi\|_{L^\infty(D^+_{\frac14})}+\|\psi\|_{L^\infty(D^+_{\frac14})})$ for a universal constant $C$. 
Let us fix $\tau\in(0,1)$ and take $s=(1+\tau)/2$. From the previous estimate we have that  for all \(\lambda < 1/8\), since $f(0)=0$,
\[
\sup_{D^+_{\lambda}} \abs[\big]{f(x')-\nabla f(0) \cdot x'}\le C'_{s,\kappa} \lambda^{1+s}\eps<\frac14\lambda^{1+\tau}\eps\,,
\]
provided that $\lambda<\lambda_2(\tau,\kappa)$.
We take 
\[
 \tilde \alpha=\partial_1f(0)
 \]
and
\[
 \boldsymbol v=\nabla f (0)-\partial_1 f(0) e_1=\nabla\psi(0)\,,
\]
where the last equality follows from the fact that by assumption $f=\psi$ on  $\partial H\cap\mathcal  C_1$. Thus
 $R(\partial E)\cap \mathcal C_\lambda^+$ is contained in the strip 
$$
S:=\{|x_N-\tilde\alpha x_1|<\frac12\lambda^{1+\tau}\eps\}\,,
$$
where  \(R\) is the rotation around  \(x_1\) which maps the vector
\[
\frac{e_N-\boldsymbol v}{\abs{e_N-\boldsymbol v}}
\]
in $e_N$, provided $\e$, hence $\e_5$, is sufficiently small, depending on $\lambda$. Note that, recalling that $|\boldsymbol v|=|\nabla\psi(0)|\leq C\eps$, the choice of $\tilde\alpha$ and \eqref{iteraostac2}, we have
$$
\|R-\Id\|\le C_4\eps,  \quad |\tilde\alpha-\alpha|\le C_4 \eps\,,
$$
for a sufficiently large $C_4$ depending only on $\kappa$.

From the fact that $E$ is a  $(\Lambda,1)$-minimizer of $\Fc_\sigma$ in $\mathcal C^+_{\frac14}$  of class $C
^{1,\tau}$ up to the boundary via a standard first variation argument, see also Lemma~\ref{pace}, we get that
$$
\frac{\partial_1f(0)}{\sqrt{1+|\nabla f(0)|^2}}\geq \sigma\,.
$$
From this inequality we get
$$
\tilde\alpha=\partial_1f(0)\geq \frac{\sigma}{\sqrt{1-\sigma^2}}-\frac{\sigma^-}{\sqrt{1-\sigma^2}}\frac{|\nabla\psi(0)|^2}{2}\,,
$$
since $|\nabla f(0)-\tilde\alpha e_1|=|\nabla\psi(0)|\leq2\eps$.
From the last inequality, setting $\bar\alpha=\max\big\{\tilde\alpha,\frac{\sigma}{\sqrt{1-\sigma^2}}\big\}$, we finally get
$$
R(\partial E)\cap\mathcal C_{\lambda}^+\subset S\cap\mathcal C_{\lambda}^+\subset\Big\{|x_N-\bar\alpha x_1|<\frac12\lambda^{1+\tau}\eps+\frac{2\sigma^-\eps^2}{\sqrt{1-\sigma^2}}\Big\}\subset\{|x_N-\bar\alpha x_1|<\lambda^{1+\tau}\eps\}\,,
$$
provided $\eps_5$ is chosen sufficiently small.
\end{proof}

%We denote by $\mathcal R$ the set of all rotations about the $x_1$ axis.
%Given a point $\bar x\in \overline H$ and $\alpha\in\R$ we define the excess at $\bar x$ with respect to the hyperplane parallel to $x_N-\alpha x_1=0$ passing through $\bar x$ in the cylinder $C_r^+(\bar x')$ the following quantity
%$$
%\mathcal E(E,\bar x,\alpha,r)=\frac{1}{r}\min\{\eps\geq0:\, (\partial E-\bar x)\cap C_r^+(\bar x') \subset S_{-\eps,\eps}^\alpha\}\,.
%$$
%Then we set 
%$$
%\mathcal E(E,\bar x,r)=\inf\{\mathcal E(R(E),R(\bar x),\alpha,r):\, \alpha\in\R,\, R\in\mathcal R\}\,.
%$$
%We now prove the following regularity criterion
%\begin{lemma}
%Let $E\subset H$ be a set of finite perimeter such that there exists a cylinder $C_R^+(\tilde x')$, $\tilde x'\in\partial H$ with the property that for every $x\in\partial^*E\cap C_R^+(\tilde x')$ and every $0<r<R$ such that $C^+_r(x')\subset C^+_R(\tilde x')$, the inequality
%$$
%\mathcal E(E,x,r)\leq Cr^{\tau}
%$$  
%holds
%for some $C>0$ and $\tau\in(0,1)$. Then $\overline{\partial^*E\cap H}\cap {\overline C_R^+}(\tilde x')$ is a $C^{1,\tau}$ hypersurface.
%\end{lemma}
%\begin{proof}
%\end{proof}

We can now prove the following

\begin{lemma}\label{lm:epsilon} Let $\tau\in(0,1/2)$. There exist $\bar\lambda=\bar\lambda(\tau)\in(0,1/2)$ and $\bar C=\bar C(\tau)$ such that for all $\lambda\in(0,\bar\lambda)$ it is possible to find  \( \bar\eps=\bar\eps(\lambda,\tau)\in(0,\frac12)\), $\bar\eta=\bar\eta(\lambda,\tau)\in(0,\frac12)$ with the following property:
Assume  $E\subset H$ is a $(\Lambda,1)$-minimizer of $\Fc_\sigma$ in $\R^N$ with obstacle $O\in\mathcal B_R$  and let $y\in \overline{\partial E\cap H}$ be such that
%$\bar x\in\partial E\cap \overline H$ such that 
%\[
%\partial (E-\bar x) \cap \mathcal C_r^+\subset S_{-\bar\eps r,\bar\eps r}\,,
%\]
 $$(\overline{\partial E\cap H}-y)\cap \mathcal C_\varrho\subset S_{-\eps\varrho,\eps\varrho}^\alpha,
 \quad \quad  \{x_N<\alpha x_1-\eps r\}\cap\mathcal C_{\varrho}\subset E-y\,,
$$
 for some $0<\varrho\leq1$, $0<\eps\leq\bar\eps$, $\alpha\in[\frac{\sigma}{\sqrt{1-\sigma^2}}-1,\frac{\sigma}{\sqrt{1-\sigma^2}}+1]$ (and $\alpha\geq\frac{\sigma}{\sqrt{1-\sigma^2}}$ if $0\leq y_1<\frac{\lambda\varrho}{16}$),
with $\Lambda\varrho<\bar\eta\eps $,  $\frac{\varrho}{R}\leq\bar\eta\e$. Then there exist a  sequence of rotations $R_k=R_k(y)$, $R_0=I$, a sequence $\alpha_k=\alpha_k(y)\in\R$, $\alpha_0=\alpha$, such that, setting $\varrho_k=\frac{\lambda^{k+1}\varrho}{16}$,
\begin{equation}\label{epsilon1}
\|R_{k+1}-R_k\|\leq \frac{\bar C}{\lambda}\lambda^{k\tau}\eps,\qquad |\alpha_{k+1}-\alpha_k|\leq \frac{\bar C}{\lambda}\lambda^{k\tau}\eps
\end{equation} 
and 
\begin{equation}\label{epsilon2}
R_{k}(\overline{\partial E\cap H}-y) \cap\mathcal C_{\varrho_k}\subset \{ \abs{x_N-\alpha_k x_1}< \lambda^{k(1+\tau)}\varrho\eps \big\}.
\end{equation} 
Moreover $\alpha_k(y)\geq\frac{\sigma}{\sqrt{1-\sigma^2}}$ whenever $0\leq y_1<\frac{\lambda^k\varrho}{16}$.
\end{lemma}

\begin{proof} By rescaling we may assume $\varrho=1$. We fix $\tau\in(0,1/2)$. 

{\bf Step 1.} We first assume that $y\in\partial H$ and that $\alpha\leq\frac{\sigma}{\sqrt{1-\sigma^2}}+M_0\eps$, where $M_0$ is the constant in Lemma~\ref{lm:barrier}, $\eps\leq\bar\eps$. Thus we may apply Lemma~\ref{itera} taking $\lambda\leq\lambda_0(M_0,\tau)$ assuming that $\bar\eps\leq\eps_3$ and $\bar\eta\leq\eta_2$. We recall that both these constants depend on $M_0,\tau$ and $\lambda$. Then we get a rotation $R_1$ and $\alpha_1\geq\frac{\sigma}{\sqrt{1-\sigma^2}}$ satisfying 
$$
\|R_{1}-I\|\leq C_2\eps,\qquad |\alpha_{1}-\alpha|\leq C_2\eps\,,
$$
\[
R_1( \overline{\partial E\cap H}-y) \cap\mathcal C_{\lambda}\subset \{ \abs{x_N-\alpha_1 x_1}< \lambda^{1+\tau}\eps\big\}.
\]
 If $\alpha_1\leq
\frac{\sigma}{\sqrt{1-\sigma^2}}+M_0\lambda^\tau\eps$, we can apply Lemma~\ref{itera} again, with $r=\lambda$ and $\eps$ replaced by $\lambda^\tau\eps$,  to get a rotation $\tilde R$ such that $\|\tilde R-I\|\leq C_2\lambda^\tau\eps$ and $\alpha_2\geq\frac{\sigma}{\sqrt{1-\sigma^2}}$ such that, setting $R_2=\tilde R\circ R_1$, $\|R_{2}-R_1\|\leq C_2\lambda^{\tau}\eps$,  $|\alpha_{2}-\alpha_1|\leq C_2\lambda^{\tau}\eps$ and
\[
R_2( \overline{\partial E\cap H}-y) \cap\mathcal C_{\lambda^2}\subset \{ \abs{x_N-\alpha_2 x_1}<\lambda^{2(1+\tau)} \eps\big\}.
\]
We may now iterate this procedure and get a sequence of rotations $R_k$ and a sequence $\alpha_k\geq\frac{\sigma}{\sqrt{1-\sigma^2}}$ such that
$$
\|R_{k+1}-R_{k}\|\leq C_2\lambda^{k\tau}\eps,\qquad |\alpha_{k+1}-\alpha_{k}|\leq C_2\lambda^{k\tau}\eps
$$
 and 
 $$
R_{k+1}( \overline{\partial E\cap H}-y) \cap\mathcal C_{\lambda^{k+1}}\subset \{ \abs{x_N-\alpha_{k+1} x_1}< \lambda^{(k+1)(1+\tau)}\eps\big\}.
 $$
  hold as long as $\alpha_{k}\leq\frac{\sigma}{\sqrt{1-\sigma^2}}+M_0\lambda^{k\tau}\eps$. If the latter inequality is satisfied for every $k$, we get the conclusion with $\frac{\bar C}{\lambda}$ replaced by  $C_2$ and with $\varrho_{k}$  replaced by $\lambda^{k+1}$ and thus also by  $\lambda^{k+1}/16$.

Otherwise let $\bar k\in\mathbb N\cup\{0\}$ be the first integer such that $\alpha_{\bar k}>\frac{\sigma}{\sqrt{1-\sigma^2}}+M_0\lambda^{\bar k\tau}\eps$. In this case assuming that $\bar\eps\leq\eps_2$, Lemma~\ref{lm:barrier} yields that
\[
R_{\bar k}( \overline{\partial E\cap H}-y) \cap \partial H\cap\mathcal C_{\frac{\lambda^{\bar k}}{4}}=\partial_{\partial H}O\cap\mathcal C_{\frac{\lambda^{\bar k}}{4}}\cap\{|x_N|<\lambda^{\bar k(1+\tau)}\e\}\,,
\]
provided that we also enforce $\bar\eta\leq\eta(1)$.
Observe that from the previous iteration argument we know in particular that
 $$
R_{\bar k}( \overline{\partial E\cap H}-y) \cap\mathcal C_{\frac{\lambda^{\bar k}}{4}}\subset \{ \abs{x_N-\alpha_{\bar k} x_1}< \lambda^{{\bar k}(1+\tau)}\eps \big\}.
 $$
 We may now use Lemma~\ref{iteraostac} with $\tilde\eps=4\lambda^{\bar k \tau}\eps$ and $r=\lambda^{\bar k}/4$, provided that $4\bar\eps\leq \eps_5$ and that we have chosen from the beginning $\lambda\leq\lambda_2(\tau,\kappa)$, where $\kappa:=\frac{|\sigma|}{\sqrt{1-\sigma^2}}+2$. Indeed, since $\alpha\leq\frac{|\sigma|}{\sqrt{1-\sigma^2}}+1$, we have that $\alpha_{\bar k}\in[-\kappa,\kappa]$, since
\begin{equation}\label{epsilon3}
| \alpha_{\bar k}-\alpha|\leq C_2\bar\eps\sum_{n=0}^\infty\lambda^{\tau n}<1
\end{equation}
by taking $\bar\eps$ smaller if needed. So we get  there exist  a rotation \(R_{\bar k+1}\)   with  \(\|R_{\bar k+1}-R_{\bar k}\|\le 4C_4\lambda^{\bar k \tau}\eps\),  \(|\alpha_{\bar k+1}-\alpha_{\bar k}|\le 4C_4\lambda^{\bar k \tau}\eps\), $\alpha_{\bar k+1}\geq\frac{\sigma}{\sqrt{1-\sigma^2}}$, such that 
\[
R_{\bar k+1}( \overline{\partial E\cap H}-y) \cap\mathcal C_{\frac{\lambda^{\bar k+1}}{4}}^+\subset \{ \abs{x_N-\bar\alpha x_1}<\lambda^{(\bar k+1)(1+\tau)} \eps \big\}.
\]
At this point we keep iterating the previous argument by applying Lemma~\ref{iteraostac} to get for all $k>\bar k$ a sequence $R_k$ and a sequence $\alpha_k\geq\frac{\sigma}{\sqrt{1-\sigma^2}}$ satisfying \eqref{epsilon2}  (even with  $\varrho_k$ replaced by $\lambda^k/4$) and \eqref{epsilon1} with $\frac{\bar C}{\lambda}$ replaced by $4C_4$. Note that, arguing as for \eqref{epsilon3},  we may ensure that during this iteration process $\alpha_k\in[-\kappa,\kappa]$ provided that we choose $\bar\eps$ smaller if needed, depending on $\lambda$ and $\kappa$.

{\bf Step 2.} Let us now assume that $y\in\partial E\cap H$. If $y_1\geq\frac{\lambda}{16}$, since  by assumption we have
\[
\partial (E-y) \cap\mathcal C_{\frac{\lambda}{16}}\subset \{ \abs{x_N-\alpha x_1}<\eps\big\}\,,
\]
we may apply iteratively Lemma~\ref{iteraintern} choosing $\lambda<\lambda_1(\tau,\kappa)$, with $\kappa$ as above, taking $\bar\eps\leq\frac{\lambda}{16}\eps_4$ and $\bar\eta\leq\eta_3$, where we recall that the constants $\eta_3, \eps_4$ depend on $\tau,\kappa$ and $\lambda$. In this way we get the conclusion  $\varrho_k=\frac{\lambda^{k+1}}{16}$ and with a sequence of rotations $R_k$, and a sequence $\alpha_k$ such that \(\|R_{k+1}-R_k\|\le \frac{16}{\lambda}C_3\lambda^{k\tau}\eps\),  \(|\alpha_{k+1}-\alpha_k|\le \frac{16}{\lambda}C_3\lambda^{k\tau}\eps\).
Note that, arguing as above, choosing $\bar\eps$ smaller if needed, we may ensure that all the $\alpha_k$ remain in the interval $[-\kappa,\kappa]$.

{\bf Step 3.} If $y_1<\frac{\lambda}{16}$, we denote by $\hat k$ the last integer such that $y_1<\frac{\lambda^{\hat k}}{16}$. We denote by $\bar y$ the point $\bar y=(0,y_2\,\dots,y_{N-1},y_N-\alpha y_1)$. Observe that by assumption this point satisfies
$$
( \overline{\partial E\cap H}-\bar y)\cap \mathcal C_{\frac34}\subset S_{-\eps,\eps}^\alpha\,.
$$
Hence from Step 1 we have that for all $k=0,1,\dots,\hat k$ there exist a radius $r_k\in\{\frac{3}{16}\lambda^k,\frac34\lambda^k\}$, a rotation $R_k$ and a number $\alpha_k$, 
such that
$$
\|R_{k+1}-R_k\|\leq \frac43\bar C\lambda^{k\tau}\eps,\qquad |\alpha_{k+1}-\alpha_k|\leq \frac43\bar C\lambda^{k\tau}\eps,
$$
with $\bar C=\max\{C_2,4C_4\}$, 
and 
$$
R_{k}( \overline{\partial E\cap H}-\bar y) \cap\mathcal C_{r_k}\subset \{ \abs{x_N-\alpha_k x_1}< \lambda^{k(1+\tau)} \eps \big\}.
$$
In particular we have that that for $k=1,\dots,\hat k$
$$
R_{k}(\overline{\partial E\cap H}- y) \cap\mathcal C_{\frac{\lambda^k}{8}}\subset \{ \abs{x_N-\alpha_k x_1}< \lambda^{k(1+\tau)}\eps \big\}\,.
$$
In particular we have that
$$
R_{\hat k}( \overline{\partial E\cap H}- y) \cap\mathcal C_{\frac{\lambda^{\hat k+1}}{16}}\subset \{ \abs{x_N-\alpha_{\hat k} x_1}< \lambda^{\hat k(1+\tau)}\eps \big\}\,.
$$
Note that, as already observed in Step1, $\alpha_k\geq\frac{\sigma}{\sqrt{1-\sigma^2}}$ for all $k=1,\dots,\hat k$.
Observing that the cylinder $\mathcal C_{\frac{\lambda^{\hat k+1}}{16}}(y)\subset H$, we may start from this cylinder arguing as in the proof of Step~2 to conclude.
 \end{proof}

\begin{theorem}[\(\eps\)-regularity theorem]\label{th:epsreg}
 There exists    \(\widehat\eps >0\) with the following property. If $E\subset H$ is a $(\Lambda,1)$-minimizer of $\Fc_\sigma$ in $\R^N$ with obstacle $O\in\mathcal B_R$, $\bar x\in\overline{\partial E\cap H}\cap\partial H$ such that 
\[
(\overline{\partial E\cap H}-\bar x) \cap \mathcal C_r\subset S_{-\widehat\eps r,\widehat\eps r}\quad  \big\{x_N<\frac{\sigma x_1}{\sqrt{1-\sigma^2}}-\widehat\eps r\Big\}\cap\mathcal C_{r}^+\subset E-\bar x\,,
\]
where $0<r\leq1$, 
 $\Lambda r<\widehat\eps$,  $\frac{r}{R}\leq\widehat\eps$,
 then $M:=\overline{\partial E\cap H} \cap\mathcal C_{\frac r2}(\bar x')$ is a hypersurface (with boundary) of class $C^{1,\tau}$ for all $\tau\in(0,\frac12)$. Moreover, 
\begin{align*}
\nu_{E}\cdot \nu_{H}\ge \sigma \quad &\text{on} \quad   M\cap \partial H; \\
\nu_{E}\cdot \nu_{H}= \sigma \quad &\text{on} \quad   (M\cap \partial H)\setminus \partial_{\partial H} O,
\end{align*}
 \end{theorem}
\begin{proof}
We may assume  $\bar x=0$.

{\bf Step 1.} We claim that given $\tau$ there exists $\hat\eps=\hat\eps(\tau)$ such that if the assumptions are satisfied for such $\hat\eps$ then $\partial E$ is of  class $C^{1,\tau}$ in $\mathcal C^+_{\frac{r}{2}}(\bar x')$ uniformly up to $\partial H$.  

To this aim we may assume without loss of generality that $r=1$.
We fix $\tau\in(0,1/2)$ and let $\bar\lambda$ and $\bar C$ as in Lemma~\ref{lm:epsilon}. Fix $\lambda\in(0,\bar\lambda)$. Let $\bar\eps$ and $\bar\eta$ be the corresponding constants provided once again by Lemma~\ref{lm:epsilon} and set $\hat\eps=\frac{1}{2}\bar\eta\tilde\eps$, with $\tilde\eps\leq\bar\eps$ to be chosen. 

We fix $y\in\partial E\cap H\cap\mathcal C_{\frac12}$, $y=(y',y_N)$ and observe that 
$$
(\overline{\partial E\cap H}-y) \cap \mathcal C_{\frac12}\subset S_{-\frac{\tilde\eps}{2},\frac{\tilde\eps}{2}}
$$
and that $\frac{1}{2}\Lambda<\bar\eta\tilde\eps $,  $\frac{1}{2R}\leq\bar\eta\tilde\eps$. Therefore, from \eqref{epsilon1} and \eqref{epsilon2} we have that there exist a sequence of rotations $R_k(y)$ converging to $R(y)$ and a sequence $\alpha_k(y)$ converging to $\alpha(y)$, for suitable $R(y)$ and $\alpha(y)$ such that
\begin{equation}\label{epsreg1}
\|R_{k}(y)-R(y)\|\leq C(\lambda,\tau)\lambda^{k\tau}\tilde\eps,\qquad |\alpha_{k}(y)-\alpha(y)|\leq C(\lambda,\tau)\lambda^{k\tau}\tilde\eps\,,
\end{equation} 
with $C(\lambda,\tau)=\frac{\overline C}{\lambda(1-\lambda^\tau)}$ and 
$$
R_{k}(y)(\overline{\partial E\cap H}-y) \cap\mathcal C_{\varrho_k}\subset \Big\{ \abs{x_N-\alpha_k(y) x_1}< \lambda^{k(1+\tau)}\frac{\tilde\eps}{2} \Big\}\,,
$$
with $\varrho_k=\frac{\lambda^{k+1}}{32}$. Note now that by the classical interior regularity results $\partial E\cap H$ is a locally $C^{1,\gamma}$-hypersurface for all $\gamma\in(0,1)$, provided $\tilde\eps$ is sufficiently small. Therefore the hyperplane $$y+R(y)^{-1}(\{x_N-\alpha(y)x_1=0\})$$ coincides with the tangent plane to $\partial E$ at $y$.

Let now $y,z\in\partial E\cap H \cap \mathcal C_{\frac12}$ with $0<|y-z|<\frac{\lambda}{32}$ and let $h$ an integer $h\geq0$ such that $\frac{\varrho_{h+1}}{2}\leq|y-z|<\frac{\varrho_{h}}{2}$. Assume that $0< y_1\leq z_1$. Since $\mathcal C_{\frac{\varrho_h}{2}}(z')\subset \mathcal C_{\varrho_h}(y')$, we have 
$$
R_{h}(y)(\overline{\partial E\cap H}-z) \cap\mathcal C_{\frac{\varrho_h}{2}}\subset \Big\{ \abs{x_N-\alpha_h(y) x_1}< \lambda^{h(1+\tau)}\tilde\eps \Big\}\,.
$$
Thus we may apply Lemma~\ref{lm:epsilon} with $\varrho=\frac{\varrho_h}{2}$, $\eps=\frac{64\lambda^{h\tau}}{\lambda}\tilde\eps\leq\bar\eps$ provided we have chosen $\tilde\eps$ sufficiently small. Thus we get for $k\geq h$ a sequence of radii $r_k=\frac{\lambda^{k-h+1}\varrho_h}{32}=\frac{\lambda^{k+2}}{32^2}$, a sequence of rotations
$S_k(z)$ converging to $S(z)$ and a sequence $\beta_k(z)$ converging to $\beta(z)$ such that
$$
S_{k}(z)(\overline{\partial E\cap H}-z) \cap\mathcal C_{r_k}\subset \{ \abs{x_N-\beta_k(z) x_1}< \lambda^{k(1+\tau)}\tilde\eps\}\,.
$$
%At the same time we have from  \eqref{epsreg2} 
%$$
%R_{k}(z)(\overline{\partial E\cap H}-z) \cap\mathcal C_{\varrho_k}\subset \Big\{ \abs{x_N-\alpha_k(z) x_1}< \lambda^{k(1+\tau)}\frac{\tilde\eps}{2} \Big\}\,.
%$$
Clearly  $S(z)= R(z)$ and $\beta(z)=\alpha(z)$ by the uniqueness of the tangent plane.  Note also that, arguing as for \eqref{epsreg1}, we  have
$$
\|S_{k}(z)-R(z)\|\leq C(\lambda,\tau)\lambda^{k\tau}\tilde\eps,\qquad |\beta_{k}(z)-\alpha(z)|\leq C(\lambda,\tau)\lambda^{k\tau}\tilde\eps,
$$
for a possibly larger constant $C(\lambda,\tau)$. Therefore, since $R_h(y)=S_h(z)$, and since $\frac{\lambda^{h+2}} {64}\leq|y-z|$ by our choice of $h$, we have
$$
\|R(y)-R(z)\|\leq\|R(y)-R_h(y)\|+\|S_h(z)-R(z)\|\leq 2C(\lambda,\tau)\lambda^{h\tau}\tilde\eps\leq \widetilde C(\lambda,\tau)\tilde\eps|y-z|^\tau\,.
$$
A similar estimate holds also for $|\alpha(y)-\alpha(z)|$, showing that both the maps $\alpha$ and $R$ are $\tau$-H\"older continuous uniformly up to $\partial H$. This proves that $\partial E$ is of class $C^{1,\tau}$ up to $\partial H$, where $\tau$ is the exponent fixed at the beginning of Step 1. 

Finally observe that if $y\in\overline{\partial E\cap H}\cap\partial H$, we may choose a sequence $y_k\in\partial E\cap H$ converging to $y$ and such that $y_k\cdot e_1<\frac{\lambda^{k}}{32}$. Then, from Lemma~\ref{lm:epsilon} we have that $\alpha_k(y_k)\geq\frac{\sigma}{\sqrt{1-\sigma^2}}$ and in turn using \eqref{epsreg1} and passing to the limit we get that also 
\begin{equation}\label{epsreg3}
\alpha(y)\geq\frac{\sigma}{\sqrt{1-\sigma^2}}\,.
\end{equation}

\noindent
{\bf Step 2.} Let us now show that if $\overline {\partial E\cap H}$ is of class $C^{1,\tau}$ in $\mathcal C_{\frac12}$ for some $\tau\in(0,\frac12)$, then it is also of class $C^{1,\gamma}$ for all $\gamma\in(0,1/2)$.
To this aim we take a point $y\in\overline {\partial E\cap H}\cap\partial H\cap\mathcal C_{\frac12}$ and consider two cases.

Assume first that $\alpha(y)>\frac{\sigma}{\sqrt{1-\sigma^2}}$. Exploiting the $C^1$ regularity of $\overline {\partial E\cap H}$ up to $\partial H$  we may find $\eps<\eps_2$ and $\varrho$ so small that $\alpha(y)>\frac{\sigma}{\sqrt{1-\sigma^2}}+M_0\eps$, where $M_0$ and $\eps_2$ are the constants of Lemma~\ref{lm:barrier}, and that
$$
(\partial E-y)\cap \mathcal C_{2\varrho}^+\subset \{ \abs{x_N-\alpha x_1}< \eps\varrho\big\}\,.
$$
Then we have that $
\overline{(\partial E -y)\cap H}\cap \partial H\cap\mathcal C_{\varrho/2}=\partial_{\partial H}O\cap\mathcal C_{\varrho/2}\cap\{|x_N|<\eps\varrho \}$. Therefore, see for instance \cite{DuzaarSteffen02}, we may conclude that $\overline {\partial E\cap H}$ is of class $C^{1,\gamma}$ in $\mathcal C_{\varrho/2}(y')$ for all $\gamma\in(0,1)$.

Otherwise, recalling \eqref{epsreg3}, 
 we have $\alpha(y)=\frac{\sigma}{\sqrt{1-\sigma^2}}$. In this case, given $\gamma\in(0,1/2)$, we may choose $\varrho$ so small that the assumptions of the claim in Step 1 are satisfied in $\mathcal C_{\varrho}(y')$ with $\hat\eps(\tau)$ replaced by $\hat\eps(\gamma)$ and the conclusion follows from Step 1.
\end{proof}

%\subsection{Height bound}
%
%\subsection{Lipschitz approximation}
%
%\subsection{Caccioppoli inequality}
%
%\subsection{Excess decay}

%\begin{proposition}\label{prop:coincidence}
%Let \(\sigma \in (-1,1)\) then there exists a universal constant \(C\) such that for all \(\beta \in (0,1/2)\) there exists \(\eps_{0}=\eps_0(d, \beta, \sigma)\), \(M_0=M_0(d, \beta, \sigma)\)  and \(\rho_0=\rho (d, \beta, \sigma)\) such that  \(E\) is a \((\Lambda, r_0)\) minimizer in \(B_{4}\), \eqref{e:basicassumption} are in force, \(0\in \partial H \cap \partial_{\partial H} O\), \(\nu \in \Sb^{d-1}\) satisfies
%\[
%\sigma \le \nu\cdot e_1\le \sigma+\frac{\eps_0}{M_0}
%\]
% and 
%\[
%\Lambda+\norm{\II_{ \partial_{\partial H} O}}_{C_0}+\e_{\nu}(E,0,2)\le \eps_{0}
%\]
%then there exists \(\hat{\nu}\in \Sb^{d-1}\) with 
%\[
%\abs{\hat{\nu}-\nu}^2\le C \e_{\nu}(E,0,2)\qquad \sigma \le \hat{\nu}\cdot e_1
%\]
%and such that 
%\[
%\e_{\hat{\nu}}(E,0,\rho)\le \rho^{2\beta} \e_{\nu}(E,0,2).
%\]
%\end{proposition}

\section{A monotonicity formula and proof of the main  results}\label{sec:mainthm}

%\subsection{Monotonicity formula}\label{subsec:monotonicity}

In this section, in view of the applications to the model for nanowire growth  discussed in Subsection~\ref{nanosec},  we consider also the case of convex polyhedral obstacles. Thus, in order to
deal at the same time with convex and smooth obstacles, we introduce the following definition where we identify \(O\) as a subset of \(\mathbb R^{N-1}\).

\begin{definition}

Let \(O\subset \partial H\approx \R^{N-1}\), we say that \(O\) is locally  \emph{semi-convex} at scale \(\bar r>0\) and with constant \(C\ge 0\) such that, with if for all \(\bar x \in \partial_{\partial H} O\) there exists a radius \(C\)-semiconvex function \(\psi: \mathbb R^{N-2} \to \mathbb R\) with \(\psi(0)=0\) such that, up to a change of of coordinates:
\begin{equation}\label{e:graphrep}
O\cap \partial H \cap B_{\bar r}(\bar x)=\Bigl(\bar x+\bigr\{ (0,x'',x_N): x_N \ge \psi (x'')\bigl\}\Bigr)\cap B_{\bar r}(\bar x).
\end{equation}
\end{definition}
Recall that a function \(\psi: \mathbb R^{N-2} \to \mathbb R\) is said to be \(C\) semiconvex if the function \(\psi(x'')+C|x''|^2/2\) is convex. In particular the sub-differential of \(\partial \psi(x'')\) for all \(x''\)  is non-empty, where
\[
\partial \psi(x'')=\Bigl\{p\in\R^{N-2}:  \psi (y'')\ge \psi (x'') +p \cdot(y''-x'')-\frac{C|x''-y''|^2}{2}\text{ for all \(y''\in\R^{N-2}\)}\Bigr\}. 
\]
Note in particular that if \(\psi (0)=0\) and \(p \in \partial \psi (x'')\) then 
\begin{equation}\label{e:conv}
0\ge \psi (x'')-p \cdot x''-\frac{C|x''|^2}{2} . 
\end{equation}
By taking \(O\) as above 
%(and identifying it as a subset of \(R^{N-1}\) )
 we define  for \(x =(0,x'',x_N)\in \partial_{\pa H} O\) the  normal and the tangent cones  at \(x\) in $\pa H$ as  
\[
N_{x} O=\{  \lambda (0,p, -1):\, p \in  \partial \psi (x''), \lambda \in [0,+\infty)\}
\]
and 
\[
T_{ x} O=\{ v \in \R^{N}:\, v\cdot e_1=0,\,v \cdot \nu\le 0 \quad\text{for all \(\nu \in N_{x} O\)}\}.
\]
It is well known  that for \(x =(0,x'',x_N)\in \partial_{\pa H} O\), the sets 
$$
O_{x, r}=\frac{O- x}{r} \to T_{ x} O
$$
as \(r\to 0\) where the convergence is the Kuratowski sense.
We start with the following technical lemma.
\begin{lemma}\label{lm:vectorfield}
Let \(O\subset \partial H\) be such that \eqref{e:graphrep} is satisfied. Then there exist a smooth  vector field \(X: \R^N \to \mathbb R^{N}\) such that
\begin{enumerate}
\item[(i)] \(X(x)\cdot e_1=0 \) for all \(x \in \partial H\)
\item[(ii)] \(X(x) \cdot \nu\ge 0\) for all \(x \in \partial_{H} O\cap B_{\bar r}(\bar x)\)  and all \(\nu \in N_{x} O\)
\item[(iii)] For all \(x \in B_{\bar r} (\bar x) \) it holds:
\begin{align}
X(x)&=x+ O(|x|^2)\label{e:st1}\,,
\\
\nabla X(x)&=\mathrm{Id}+O(|x|)\label{e:st2}\,.
\end{align}
\end{enumerate}
\end{lemma}
\begin{proof}
We may assume that \(\bar x=0\) and we define 
\[
X(x)=\Bigl(x_1,x'',x_N-\frac{C|x''|^2}{2}\Bigr)
\]
where \(C\) is the semiconvexity constant of \(\psi\). Clearly (i) and (iii) are satisfied. To check (ii), note that if $\nu\in N_xO$ then
\[
\nu=\lambda (0,p,-1)
\]
for \(p \in \partial \psi (x')\) and \(\lambda \ge 0\) so that 
\[
X(x) \cdot \nu=\lambda\Bigl(p\cdot x'-\psi(x')+\frac{C|x'|^2}{2}\Bigr)\ge 0
\]
by \eqref{e:conv}.
\end{proof}

We can now state the desired monotonicity formula for \((\Lambda,r_0)\) minimizers.
\begin{theorem}\label{thm:montonicity}
Let   $E\subset H$  be a $(\Lambda,r_0)$-minimizer of $\Fc_\sigma^O$  with obstacle $O$. Assume that \(\bar x \in \partial_{\partial H} O\) and that \eqref{e:graphrep} is satisfied with \(0<\bar r\le  r_0\) and a \(C\)-semiconvex function \(\psi\) with \(\psi(0)=0\). Then there exists a constant \(c_0=c_0(C,\Lambda,N)\) such that  for all $0<s<r$, with $r<\bar r$ sufficiently small,
\beq\label{monotone-1}
e^{c_0r}\frac{\Fc_\sigma^O(E, B_{r}(\bar x))}{r^{N-1}}-e^{c_0s}\frac{\Fc_\sigma^0(E, B_{s}(\bar x))}{s^{N-1}}\geq\frac12\int_s^rt^{1-N}e^{c_0t}\frac{d}{dt}\bigg[\int_{\partial E	\cap H\cap B_t(\bar x) }\frac{(x\cdot \nu_E)^2}{|x|^2}\bigg]\,dt\,. 
\eeq
%Here \(\mathcal B_{r}(\bar x)\) is the set 
%\[
%\bar x+\{x: |X(x)|\le r\}
%\]
%where \(X\) is the vector field constructed in Lemma \ref{lm:vectorfield}.
\end{theorem}
%Note that, by \eqref{e:st1}, the set \(\mathcal B_{r}(\bar x))\) is asymptotic to  \( B_{r}(\bar x)\) in the sense that 
%\[
% B_{r(1-o(1))}(\bar x)\subset \mathcal B_{r}(\bar x)\subset  B_{r(1+o(1)))}(\bar x)
%\]
%where \(o(1) \to 0\) as \(r \to 0\).
\begin{proof} 
In this proof by $O(r)$ we mean any function bounded by $cr$ for $r$ small, where the constant $c$ depends only on $\Lambda$, $N$ and the semiconvexity constant $C$.

We assume that \(\bar x=0\) and we let \(\phi\in C_c^\infty([0,1);[0,1))\) be a smooth decreasing function with \(\phi (0)=1\)  and we consider for \(r\le \bar r/2\)  the vector field \(T: \R^N\to \R^N\)  defined as 
\[
T(x)=-\phi \Biggl(\frac{|x|}{r}\Biggr)X(x)\,,
\]
where $X$ is as in  Lemma~\ref{lm:vectorfield}. By (i) and (ii) of the lemma, if we let \(\varphi_t\) be the flow generated by \(T\), then for all \(t>0\)
\[
\varphi_t(\partial H)=\partial H,\qquad  \varphi_t( O)\subset  O\qquad \text{and} \qquad \varphi_t(x)=x\quad\text{for $x\not\in B_r$\,.} 
\]
In particular  \(\varphi_t(E)\Delta E\Subset B_{\bar r}\) and thus the set \(\varphi_t(E)\) is a competitor for the $(\Lambda, r_0)$-minimality of $E$.  Hence  
\beq\label{monotone0}
\begin{split}
\Fc_\sigma^O (E;B_{\bar r})&
\le \Fc_\sigma^O (\varphi_t(E);B_{\bar r})+\Lambda |\varphi_t(E)\Delta E| \\
&\leq P(\varphi_t(\pa E\setminus O))+\sigma P(\varphi_t(\pa E\cap O))+\Lambda |\varphi_t(E)\Delta E|
\end{split}
\eeq
for all \(t>0\).
By using the coarea formula and recalling that  \(|X(x)|=O(r)\) on \(\mathrm{spt} \,T\), it is easy to check that for $t>0$ small enough
$$
|\varphi_t(E)\Delta E|= t\int_{\partial E	\cap H} |T(x)\cdot\nu_E(x)|\,d\mathcal H^{N-1} +o(t)\leq
tO(r)\int_{\partial E	\cap H} \phi \biggl(\frac{|x|}{r}\biggr)\,d\mathcal H^{N-1}\,.
$$
Therefore, differentiating the inequality in \eqref{monotone0} we get
\begin{equation}\label{e:der}
O(r)\int_{\partial E	\cap H} \phi \biggl(\frac{|x|}{r}\biggr)\,\mathcal H^{N-1}\le \int_{\partial E \setminus O} \Div_{\tau} T d \mathcal H^{N-1}+\sigma  \int_{\partial E	\cap O} \Div_{\tau} T d \mathcal H^{N-1}\,,
\end{equation}
where 
\[
\Div_{\tau} T =\Div T-\nabla T [\nu_E] \cdot \nu_E
\]
is the tangential divergence of \(T\). Since 
\[
\nabla T =-\phi \biggl(\frac{|x|}{r}\biggr)\nabla X(x)-\phi' \biggl(\frac{|x|}{r}\biggr)\frac{|x|}{r} \frac{x}{|x|}\otimes\frac{X(x)}{|x|}
\]
by exploiting \eqref{e:st1}, \eqref{e:st2}  we have
\[
\begin{split}
\Div_{\tau} T =&-\phi \biggl(\frac{|x|}{r}\biggr)(N-1)-\phi' \biggl(\frac{|x|}{r}\biggr)\frac{|x|}{r}
\\
&+\phi' \biggl(\frac{|x|}{r}\biggr)\frac{|x|}{r}\biggl(\frac{(x\cdot \nu_E)^2}{|x|^2}\biggr)
+O(r)\phi \biggl(\frac{|x|}{r}\biggr)+O(r)\phi' \biggl(\frac{|x|}{r}\biggr)\frac{|x|}{r}.
\end{split}
\]
which can be written as 
\[
\begin{split}
\Div_{\tau} T &=(1+O(r))r^N\frac{d}{d r} \Biggl(r^{1-N}\phi \biggl(\frac{|x|}{r}\biggr)\Biggr)
\\
&+\phi' \biggl(\frac{|x|}{r}\biggr)\frac{|x|}{r}\biggl(\frac{(x\cdot \nu_E)^2}{|x|^2}\biggr)
+O(r)\phi \biggl(\frac{|x|}{r}\biggr)
\end{split}
\]
Combining the above inequality with \eqref{e:der} we infer, after easy computations,
\[
\begin{split}
(1+O(r))&\frac{d}{d r} \Biggl(r^{1-N}  \int_{\partial E	\setminus O} \phi \biggl(\frac{|x|}{r}\biggr)\mathcal H^{N-1}+\sigma r^{1-N}   \int_{\partial E	\cap O} \phi \biggl(\frac{|x|}{r}\biggr) d \mathcal H^{N-1}\Biggr)
\\
\ge& -c \Biggl(r^{1-N}  \int_{\partial E	\setminus O} \phi \biggl(\frac{|x|}{r}\biggr)\mathcal H^{N-1}+\sigma r^{1-N}  \int_{\partial E	\cap O} \phi \biggl(\frac{|x|}{r}\biggr) d \mathcal H^{N-1} \Biggr)
\\
&- r^{-N} \int_{\partial E	\setminus O}\phi' \biggl(\frac{|x|}{r}\biggr)\frac{|x|}{r}\biggl(\frac{(x\cdot \nu_E)^2}{|x|^2}\biggr)-\sigma r^{-N} \int_{\partial E	\cap O}\phi' \biggl(\frac{|x|}{r}\biggr)\frac{|x|}{r}\biggl(\frac{(x\cdot \nu_E)^2}{|x|^2}\biggr).
\end{split}
\]
In turn, noticing that
$$
\int_{\partial E	\cap H}\phi' \biggl(\frac{|x|}{r}\biggr)\frac{|x|}{r}\frac{(x\cdot \nu_E)^2}{|x|^2}=0\,,
$$
setting
\begin{align*}
h(r)
&=r^{1-N}  \int_{\partial E	\setminus O} \phi \biggl(\frac{|x|}{r}\biggr)\,d\mathcal H^{N-1}+\sigma r^{1-N}   \int_{\partial E	\cap O} \phi \biggl(\frac{|x|}{r}\biggr) d \mathcal H^{N-1}, \\
k(r)&=- r^{-N} \int_{\partial E	\setminus O}\phi' \biggl(\frac{|x|}{r}\biggr)\frac{|x|}{r}\frac{(x\cdot \nu_E)^2}{|x|^2}-\sigma r^{-N} \int_{\partial E	\cap O}\phi' \biggl(\frac{|x|}{r}\biggr)\frac{|x|}{r}\frac{(x\cdot \nu_E)^2}{|x|^2}\\
&=r^{1-N}\frac{d}{dr}\biggl[\int_{\partial E	\cap H}\phi \biggl(\frac{|x|}{r}\biggr)\frac{(x\cdot \nu_E)^2}{|x|^2}\,d\mathcal H^{N-1}\bigg]\,, \\
\end{align*}
the inequality above implies that
$$
h'(r)\geq -c_0h(r)+\frac{1}{2}k(r)
$$
provided that $2\geq1+O(r)\geq1/2$, where $c_0$ is a costant depending only on $\Lambda, N$ and $C$. 
Note also that $k(r)\geq0$.
Multiplying both sides of this inequality by $e^{c_0r}$ and integrating in $(s,r)$, we have
\begin{align*}
&h(r)e^{c_0r}-h(s)e^{c_0s}
\geq \frac12\int_s^re^{c_0t}k(t)\,dt\,.
\end{align*}
%&\qquad=\frac12\bigg[r^{1-N}e^{c_0r}\int_{\partial E	\cap H}\phi \biggl(\frac{|x|}{r}\biggr)\frac{(x\cdot \nu_E)^2}{|x|^2}-s^{1-N}e^{c_0s}\int_{\partial E	\cap H}\phi \biggl(\frac{|x|}{s}\biggr)\frac{(x\cdot \nu_E)^2}{|x|^2}\bigg]
Then we conclude by letting $\phi\to1$.
\end{proof}

By classical argument, the above monotonicity formula allows for the study of blow-ups of \(\Lambda\)-minimizers. To this aim we preliminary observe that the following compactness property for blow-ups in the case of convex polyhedral domains holds.

\begin{theorem}\label{th:compactnessbis}
Let   $E\subset H$  be a $(\Lambda,r_0)$-minimizer of $\Fc_\sigma $  with obstacle $O$, where $O$ is a convex polyhedron. Let \(\bar x\in \partial_{\partial H}O\cap\pa E\) and set
$$
E_h=\frac{E-\bar x}{r_h}\,,
$$
where $r_h\to 0^+$.
 Then there exist a  (not relabelled) subsequence  
 and a set   $E_\infty$ of locally finite perimeter, such that $E_h\to E_\infty$ in $L^1_{loc}(\R^N)$ with the property that $E_\infty$ is a $0$-minimizer of $\Fc^{\tilde O}_\sigma$, where $\tilde O=T_{\bar x} O$.
Moreover, 
$$
\mu_{E_h}\wtos \mu_E\,,   \quad |\mu_{E_h}|\wtos |\mu_E|\,,
$$
as Radon measures. 
In addition, the following Kuratowski convergence type properties hold:
\[
\begin{split}
&\text{(i)\,\,for every $x\in\pa E_\infty$ there exists $x_h\in\pa E_h$ such that $x_h\to x$;} \\
& \text{(ii)\,\,if
 $x_h\in \overline{\pa E_h\cap H}$ and $x_h\to x$, then $x\in \overline{\pa E_\infty\cap H}$\,.}
 \end{split}
 \]
 Finally,  either 
 $\partial E_\infty\cap(\partial H\setminus{\tilde O})=\partial H\setminus{\tilde O}$ or $\partial E_\infty\cap \partial H\subset\overline{\tilde O}$.
\end{theorem} 
The proof of this theorem follows as in the proof of Theorem~\ref{th:compactness} observing that the density estimates proved in Proposition~\ref{p:densityestimates} still hold when $O$ is a convex polyhedron.

\begin{proposition}\label{cor:bu}
Let   $E\subset H$  be a $(\Lambda,r_0)$-minimizer of $\Fc_\sigma$  with obstacle $O$, where  $O$ is either of class $C^{1,1}$ or a convex polyhedron. If \(\bar x\in \partial_{\partial H}O\cap\pa E\) then the sets 
\[
E_{\bar x,r}=\frac{E-\bar x}{r}
\]
are pre-compact in \(L^1\) and every limit point \(E_\infty\) as $r\to0$ is a conical minimizer of \(\mathcal F_{\sigma}^{\tilde O}\) with obstacle \(\tilde O=T_{\bar x} O\). Moreover if \(N=3\) either \(\partial E_\infty=\pa H\)  or, after a rotation of coordinates in $\pa H$,  
\[
\partial E_\infty\cap H=\bigl\{ x: x_N=\alpha x_1\bigr\}\cap H
\]
with 
\[
\alpha \ge \frac{\sigma}{\sqrt{1-\sigma^2}}
\]
and $\tilde O$ is a half space.
%In particular if \(N=3\) and \(T_{\bar x} O\) is not an half-space, then \(\bar x \notin \partial E\).
\end{proposition}

\begin{proof}
 Let $E_\infty$ be a limit point of of $E_{\bar x,r}$ as $r\to0$. By Theorem~\ref{th:compactness} or Theorem~\ref{th:compactnessbis} $E_\infty$ is a $0$-minimizer of $\mathcal F_{\sigma}^{\tilde O}$.
Observe that by Theorem~\ref{thm:montonicity} there exists 
$$
\lim_{r\to0}\frac{\Fc_\sigma(E, B_{r}(\bar x))}{r^{N-1}}
$$
and it is finite. From this a standard argument, see for instance the proof of Theorem~28.6 in \cite{Maggi12}, shows that 
$$
\frac{\Fc_\sigma^{\tilde O}(E_\infty, B_{r}(\bar x))}{r^{N-1}}
$$
is constant with respect to $r$.
Thus, the right-hand side of  \eqref{monotone-1} with $E_\infty$ in place of $E$ is zero.  This immediately  implies that $x\cdot\nu_{E_\infty}(x)=0$ for $\H^{N-1}$-a.e. $x$, hence, see \cite[Prop. 28.8]{Maggi12}, $E_\infty$ is a cone.

Note that if \(N=3\) the only cones with zero mean curvature are planes and this forces \(\partial E_{\infty}\) to be a union of planes. These planes can not intersect in \(H\) by the interior regularity theory. In particular since \(0\in \partial E_{\infty} \) (which holds true since \(\bar x \in \partial E\)), either $\pa E_\infty=\pa H$ or, after a rotation in the hyperplane $H$,
\[
\partial E_{\infty}\cap H=\bigcup_{i=1}^k \bigl\{ x: x_1=\alpha_i x_N\bigr\}\cap H.
\]
Minimality forces \(k=1\) and $\pa E_\infty\cap \pa H=\tilde O$ and thus $\tilde O$ is a half space.
\end{proof}

We are now in a position to prove Theorem~\ref{th:reg}.
\begin{proof}[Proof of Theorem~\ref{th:reg}]
We start by proving part (i). 
In view of the result \cite[Theorem 1.2]{De-PhilippisMaggi15} it is enough to prove the regularity in a neighborhood of a point $\bar x\in\pa E\cap\pa_{\pa H}O$. Given such a point, by Proposition~\ref{cor:bu} we know that there exists a sequence $E_h=E_{\bar x,r_h}=\frac{E-\bar x}{r_h}$, $r_h\to0$, of blow-ups of $E$ converging in $L^1$ to a $0$-minimizer $E_\infty$ of \(\mathcal F_{\sigma}^{\tilde O}\), where \(\tilde O=T_{\bar x} O\). Moreover, either \(\partial E_\infty=\pa H\)  or, after a rotation of coordinates in $\pa H$,  
$\partial E_\infty\cap H=\bigl\{ x: x_1=\alpha x_N\bigr\}\cap H$
with $\alpha \ge \frac{\sigma}{\sqrt{1-\sigma^2}}$. Note that if $E_\infty=\pa H$ or if $\alpha> \frac{\sigma}{\sqrt{1-\sigma^2}}$ by conclusion (ii) of Theorem~\ref{th:compactness} we get that $\overline{\pa E_h\cap H}$ satisfies the assumptions of Lemma~\ref{lm:barrier} in a neighborhood of $\bar x$ and thus $\overline{\pa E\cap H}\cap\pa H$ coincides with $\pa_{\pa H}O$ in such a neighborhood. In turn, the conclusion follows by \cite[Theorem~6.1]{DuzaarSteffen02}. If instead $\alpha=\frac{\sigma}{\sqrt{1-\sigma^2}}$, using again (ii) of Theorem~\ref{th:compactness} we have that for $E_h$ satisfies the assumptions of Theorem~\ref{th:epsreg} for $h$ sufficiently large. Hence the conclusion follows also in this case.

Now the proof of part (ii) follows combining (i) with the classical Federer's dimension reduction argument, see for instance \cite[Appendix A]{Simon83} or \cite[Sections 28.4-28.5]{Maggi12}. We leave the details to the reader.
\end{proof}

We conclude the section by proving Theorems~\ref{cinqueuno} and \ref{cinquedue}. In the following we make use of the notation introduced in Subsection~\ref{nanosec}.

\begin{proof}[Proof of Theorem~\ref{cinqueuno}]
We start by recalling that by a standard argument, see for instance Example 21.3 in \cite{Maggi12}, there exists $\Lambda, r_0>0$ such that
$$
J_{\sigma,\omega}(E)\leq J_{\sigma,\omega}(F)+\Lambda||E|-|F||
$$
for all $F\subset\R^3\setminus\cc$ with diam$(E\Delta F)<r_0$. It easily follows that $E$ is a $(\Lambda,r_0)$-minimizer of $J_{\sigma,\omega}$ with obstacle $\omega$ in the sense introduced in the previous sections. Thus the conclusion follows from Theorem~\ref{th:reg}.
\end{proof}

Before proving Theorem~\ref{cinquedue} we  set the following definition.

\begin{definition}\label{nontang}
We say that $E\subset\R^3\setminus\cc$ has {\it a nontangential contact with $\cc_{top}$ at a point $\bar x\in\overline{\pa E\cap\mathscr H}\cap\gamma$} if for any subsequence $E_{\bar x, r_h}$ of blow-ups of $E$ with  $r_h\to0$, converging to $E_\infty$, we have that $\pa E_\infty$ does not contain the plane $\{x_3=0\}$.
\end{definition}

\begin{proof}[Proof of Theorem~\ref{cinquedue}] 
We start by observing that, arguing as in the proof of Theorem~\ref{cinqueuno} we have that $E$ is $(\Lambda,r_0)$-minimizer of $J_{\sigma,\omega}$ with ostacle $\omega$.

Assume first that $E$ has a nontangential contact at all points of $\overline{\pa E\cap\mathscr H}\cap\gamma$. Fix a point  $\bar x\in\overline{\pa E\cap\mathscr H}\cap\gamma$ and consider a sequence $E_{\bar x, r_h}$ of blow-ups of $E$ with  $r_h\to0$, converging to some set $E_\infty$ of locally finite perimeter.  By Proposition~\ref{cor:bu} it follows that  $E_\infty$ is a $0$-minimizer of $J_{\sigma,\omega}$ with obstacle $T_{\bar x}\omega$, that
after a rotation in the plane $\{x_3=0\}$
$$
\partial E_\infty\cap\{x_3>0\}=\bigl\{ x: x_1=\alpha x_3\bigr\}\cap \{x_3>0\}\,,
$$
with $\alpha\geq \frac{\sigma}{\sqrt{1-\sigma^2}}$ and  thus $\pa E_\infty\cap\{x_3=0\}$ is a half plane.  In turn $T_{\bar x}\omega$ is a half plane and thus it cannot be a vertex of the polygon. Hence, thanks to Theorem~\ref{th:reg},   $\overline{\pa E\cap\mathscr H}$ is of class $C^{1,\tau}$ for all $\tau\in(0,\frac12)$  in a neighborhood of $\bar x$. 

Finally, we claim that if $\sigma<0$ then $E$ has a nontangential contact at all points of $\overline{\pa E\cap\mathscr H}\cap\gamma$. Indeed assume by contradiction that $E_\infty=\{x_3>0\}$ for some point $\bar x\in\overline{\pa E\cap\mathscr H}\cap\gamma$ and some sequence of blow-ups $E_{\bar x, r_h}$ with $r_h\to0$. Then, arguing similarly as for Theorem~\ref{th:compactnessbis}, see also \cite[Theorem~2.9]{De-PhilippisMaggi15}, we get that $E_\infty$ is a $0$-minimizer of $J_{\sigma,T_{\bar x}\omega}$, that is 
$$
J_{\sigma,T_{\bar x}\omega}(E_\infty)\leq J_{\sigma,T_{\bar x}\omega}(F)
$$
for all $F\subset\R^3\setminus(T_{\bar x}\omega\times(-\infty,0])$ with $E_\infty\Delta F$ bounded. From this it easily follows that $F_\infty:=(\R^3\setminus(T_{\bar x}\omega\times(-\infty,0]))\setminus E_\infty$ is a $0$-minimizer of $J_{-\sigma,T_{\bar x}\omega}$. By a rotation in the plane $\{x_3=0\}$ we may assume without loss of generality that the half line $\ell=\{(0,x_2,0):\,x_2>0\}$ is contained in $\pa T_{\bar x}\omega\times\{0\}$ and that $F_\infty$ (locally around each point of $\ell$) is contained in the half space $H=\{x_1>0\}$. Hence, locally around any point $y\in\ell$, $F_\infty$ is a $0$ minimizer of $\mathcal F_{-\sigma}$ with obstacle $O=\{x_3<0\}\cap\pa H$. Thus, by the third inequality in \eqref{irene-1} (with $\sigma$ replaced by $-\sigma$) it follows that $\pa F_\infty$ must form with the vertical plane $\pa H$ an angle which is strictly larger then $\pi/2$, thus leading to a contradiction.
\end{proof}

\section{Appendix}
Here we give the proof of Proposition~\ref{prop:savin}. We shall closely follow the proofs of Lemmas~2.1 and 2.2 and of Theorem~1.1 in \cite{Savin} with the necessary changes. 

We denote by $\mathbb{M}^{d\times d}_{\rm sym}$ the space of symmetric $d\times d$ matrices and by ${{\rm Tr}\,(A)}$ the trace of $A$. Let $F:\mathbb{M}^{d\times d}_{\rm sym}\times\R^d\to\R$ be the function
$$
F(A,p):=\frac{{\rm Tr}\,(A)}{(1+|p+\xi|^2)^{1/2}}-\sum_{i,j=1}^d\frac{A_{ij}(p_i+\xi_i)(p_j+\xi_j)}{(1+|p+\xi|^2)^{3/2}}\,,
$$
where $\xi\in\R^d$ is a  vector, with $|\xi|\leq M$.
Observe that there exist two positive constants $\tilde\lambda,\lambda>0$, depending only on $M$ and $d$ such that if $|p|\leq1$, $A\in \mathbb{M}^{d\times d}_{\rm sym}$ and $A\geq0$, then
$$
\tilde\lambda\|A\|\geq F(A,p)\geq\lambda\|A\|\,,
$$
where $\|A\|:=\sqrt{\sum_{i,j}A^2_{ij}}$. Since throughout this section $M$ will be fixed, in the following with a slight abuse of language we will say that a constant is universal if it depends only on $d$ and $M$. 
The next result is essentially as \cite[Lemma 2.1]{Savin}.
\begin{lemma}\label{lemma2.1}
Let $r\in(0,1)$ and $a\in(0,1/2)$. Let $v:B_1\to(0,\infty)$ be a viscosity $(2a)$-supersolution of the equation
\begin{equation}\label{lemma2.10}
\Div \biggl(\frac{\nabla u+\xi}{\sqrt{1+|\nabla u+\xi|^2}}\biggr)  =a\,,
\end{equation}
bounded from above.
There exists a universal constant $c_0$ with the following property:
Let $\overline B_r(x_0)\subset B_1$ and let $B\subset\overline B_1$ be a closed set.  For every $y\in B$ consider the paraboloid 
\begin{equation}\label{lemma2.11}
-\frac{a}{2}|x-y|^2+c_y
\end{equation}
staying below $v$ in $\overline B_r(x_0)$ and touching $v$ in $\overline B_r(x_0)$. Assume that the set  $A$  of all such touching points for $y\in B$ is contained in  $B_r(x_0)$. Then $|A|\geq c_0|B|$.
\end{lemma}
\begin{proof}
Observe that since $v$ is a bounded lower semicontinuous function, then $A$ is a closed set. 

Assume that $v$ is semi-concave in $B_r(x_0)$, that is there exists $b>0$ such that $v-b|x|^2$ is concave in $B_r(x_0)$. Note that this implies in particular that $v$ is differentiable at all touching points $z\in A$. Moreover it is not difficult to show that $D v$ is Lipschitz in $A$ with a Lipschitz constant depending only on $a$ and $b$.
By assumption if $z\in A$ there exists a paraboloid of vertex $y\in B$ and opening $a$ as in \eqref{lemma2.11} touching $v$ at $z$ from below. Moreover,
\begin{equation}\label{lemma2.12}
y=z+\frac{1}{a}Dv(z)\,.
\end{equation}
Let us denote by $Z$ the set of points in $B_r(x_0)$ such that $v$ is twice differentiable at $z$. By Alexandrov theorem $|B_r(x_0)\setminus Z|=0$. Fix $z\in Z$.  For all $x\in B_r(x_0)$ we have
$$
v(x)=P(x;z)+o(|x-z|^2)=v(z)+Dv(z)\cdot(x-z)+\frac12(x-z)^TD^2v(z)(x-z)+o(|x-z|)^2\,.
$$
We claim that there exists a universal constant $C>0$ such that
$$
-aI\leq D^2v(z)\leq CaI\qquad\text{for all $z\in A\cap Z\,.$}
$$
The left inequality follows from the fact that the paraboloid in \eqref{lemma2.11} touches $v$ from below. To prove the second inequality assume that there exists a unit vector $e$ such that
$$
D^2v(z)\geq Ca\,e\otimes e-aI\,.
$$
We will prove that if $C$ is sufficiently large this leads to a contradiction.

Consider now the test function $\varphi(x)=P(x;z)-\displaystyle\frac{\eps}{2}|x-z|^2$, with $\eps>0$. Note that $\varphi$ lies below $v$ in a neighborhood of $z$ and touches $v$ at $z$ and  that by \eqref{lemma2.12} $|D\varphi(z)|=|Dv(z)|=a|y-z|\leq2a$. Therefore $\varphi$ is an admissible test function and we have
\begin{align*}
a
& \geq F(D^2\varphi(z),D\varphi(z))=F(D^2v(z)-\eps I,D\varphi(z))\geq F(Ca\,e\otimes e-(a+\eps)I,D\varphi(z)) \\
&= F(Ca\,e\otimes e,D\varphi(z))-(a+\eps)F(I,D\varphi(z))\geq Ca\lambda-(a+\eps)\tilde\lambda\sqrt{n}\,
\end{align*}
which is a contradiction if $C>(\tilde\lambda\sqrt{n}+1)/\lambda$ and $\eps$ is sufficiently small.

Now, from the area formula, using \eqref{lemma2.12} we obtain
\begin{equation}\label{lemma2.13}
|B|\leq\int_A\Bigl|{\rm det}\Bigl(I+\frac{1}{a}D^2v(x)\Big)\Big|\,dx\leq C|A|,
\end{equation}   
where $C$ is another universal constant. 

If $v$ is not semi-concave for $\eps>0$ we define the inf-convolution, setting for $x\in\overline B_r(x_0)$
$$
v_\eps(x)=\inf_{y\in\overline B_r(x_0)}\Big\{v(y)+\frac1\eps|y-x|^2\Big\}.
$$
It is easily checked that $v_\eps$ is semi-concave: Since  $v$ is  lower semicontinuous and bounded we have also that $v_\eps$ converges pointwise to $v$ in $\overline B_r(x_0)$. Moreover  each $v_\eps$ is a viscosity $(2a)$-supersolution of the equation \eqref{lemma2.10} in $B_r(x_0)$. In fact, if $\varphi\in C^2(B_r(x_0))$ lies below $v_\eps$ in a neighborhood of  $\overline x\in B_r(x_0)$ and touches $v_\eps$ at $\overline x$, let $\overline y$ the point in $\overline B_r(x_0)$ such that 
$$
\inf_{y\in\overline B_r(x_0)}\Big\{v(y)+\frac1\eps|y-\overline x|^2\Big\}=v(\overline y)+\frac1\eps|\overline y-\overline x|^2.
$$
Then the function
$$
\varphi(x+\overline x-\overline y)+v(\overline y)-\varphi(\overline x)
$$
touches $v$ from below at $\overline y$ and thus
$$
F(D^2\varphi(\overline x),D\varphi(\overline x))\leq a.
$$
Observe that for $\eps$ small enough the contact set $A_\eps$ of $v_\eps$ is contained in $B_r(x_0)$. Indeed if this is not true there exist a sequence $\eps_n$ converging to $0$ and points $x_n\in A_{\eps_n}$, the contact set of $v_{\eps_n}$, such that $x_n\in\partial B_r(x_0)$. Therefore there exist $y_n\in B$ and $c_n\in\R$ such that
$$
-\frac{a}{2}|x-y_n|^2+c_n\leq v_{\eps_n}(x)\,\,\,\forall x\in\overline B_r(x_0),\quad-\frac{a}{2}|x_n-y_n|^2+c_n= v_{\eps_n}(x_n).
$$
Since $B$ is closed and the $v_{\eps_n}$ are equibounded we may assume that $y_n\to \bar y\in B$, $c_n\to\bar c$ and $x_n\to\bar x\in\partial B_r(x_0)$. Thus, from the first inequality above we have that the paraboloid $-\frac{a}{2}|x-\bar y|^2+\bar c$ stays below $v$ in  $B_r(x_0)$. Moreover, by lower semicontinuity we have
\[
\begin{split}
-\frac{a}{2}|\bar x-\bar y|^2+\bar c&= \lim_{n\to\infty}v_{\eps_n}(x_n)\\
&=\lim_{n\to\infty}\Big(v(z_n)+\frac{1}{\eps_n}|z_n-x_n|^2\Big)\geq \liminf_{z\to \bar x} \geq v(z)\geq v(\bar x).
\end{split}
\]
Hence $\bar x\in A\cap\partial B_r(x_0)$ which is impossible by assumption. In particular we may assume that \eqref{lemma2.13} holds with $A$ replaced by $A_\eps$ for $\eps$ sufficiently small. 

Using the fact that $B$ is closed and the pointwise convergence of $v_\eps$ to $v$, a similar argument shows
$$
\bigcap_{h=1}^\infty\bigcup_{k=h}^\infty A_{1/k}\subset A.
$$
From the above inclusion and \eqref{lemma2.13} we then conclude that $|B|\leq C|A|$.
\end{proof}
Let $v:\overline B_1\to(0,\infty)$ be a lower semicontinuous function. If $a>0$ we denote by $A_a$ the set of points where $v$ can be touched from below by a paraboloid of opening $a$ and vertex in $\overline B_1$ and where the value of $v$ is also smaller than $a$.
\begin{align*}
A_a:
&=\Big\{x\in\overline  B_1:\,\,v(x)\leq a\,\,\text{and there exists}\,\,y\in\overline  B_1\,\,\text{such that}\\
&\qquad\qquad\qquad \min_{z\in\overline B_1}\Big(v(z)+\frac{a}{2}|z-y|^2\Big)=v(x)+\frac{a}{2}|x-y|^2\Big\}\,.
\end{align*}
\begin{lemma}\label{lemma2.2}
There exist two constants $c_1>0$, $C_1>1$,  with the following properties: Let $0<a<1/C_1$ and let $v:\overline B_1\to (0,\infty)$ be a viscosity $(C_1a)$-supersolution of the equation \eqref{lemma2.10}, bounded from above. If 
$$
\overline B_r(x_0)\subset B_1,\qquad A_a\cap\overline B_r(x_0)\not=\emptyset\,,
$$
then
$$
\big|A_{C_1a}\cap B_{r/8}(x_0)|\geq c_1|B_r|\,.
$$
\end{lemma}
\begin{proof}
Up to replace $r$ with $r+\eps$ and then letting $\eps\to0^+$ we may assume that there exists 
$$
x_1\in A_a\cap B_r(x_0)\,.
$$
Thus there exists $y_1\in\overline B_1$ such that the paraboloid
$$
P(x;y_1)=v(x_1)+\frac{a}{2}|x_1-y_1|^2-\frac{a}{2}|x-y_1|^2
$$
 touches $v$ in $x_1$ from below. We claim that there exists a universal constant $C$ such that there exists a point $z\in \overline B_{r/16}(x_0)$ such that
\begin{equation}\label{lemma2.20}
v(z)\leq P(z;y_1)+Car^2\,.
\end{equation}
If $x_1\in \overline B_{r/16}(x_0)$ then we may take trivially $z=x_1$. Otherwise, 
 we consider the function
\[
\varphi(x)=
\begin{cases}
\alpha^{-1}(|x|^{-\alpha}-1) & \quad \text{if}\,\,\,\displaystyle\frac{1}{16}\leq|x|\leq1, \cr
\alpha^{-1}(16^{\alpha}-1) & \quad \text{if}\,\,\,\displaystyle |x|\leq\frac{1}{16}\,,
\end{cases}
\]
with $\alpha>0$, universal, to be chosen.
Then, for $x\in B_r(x_0)$, we set 
$$
\psi(x)=P(x;y_1)+ar^2\varphi\Big(\frac{x-x_0}{r}\Big)\,.
$$
Note that $|D\psi|\leq a(16^{\alpha+1}+2)\leq1$   if $a$ is small enough and that in the annulus $B_r(x_0)\setminus \overline B_{r/16}(x_0)$ we have
\begin{equation}\label{lemma2.21}
F(D^2\psi,D\psi)= aF\Big(D^2\varphi\Big(\frac{x-x_0}{r}\Big),D\psi\Big)-aF(I,D\psi) 
 \geq a\lambda(\alpha+1)-a\tilde\lambda\sqrt{n}>a,
\end{equation}
provided $\alpha>0$ is chosen large enough. 

Let us now denote by $z$ a minimum point of $v-\psi$ in $\overline B_r(x_0)$. Since $v(x_1)-\psi(x_1)=-ar^2\varphi\big(\frac{x_1-x_0}{r})<0$, the minimum of $v-\psi$ is strictly negative. Therefore $z\not\in\partial B_r(x_0)$ since $v-\psi\geq0$ on $\partial B_r(x_0)$. On the other hand, if $r/16<|z-x_0|<r$, observing that $|D\psi|\leq a(16^{\alpha+1}+2)\leq C_1a$, if we choose $C_1$ large enough, we would have that $\psi$ is an admissible test function for $v$ in a neighborhood of $z$ satisfying \eqref{lemma2.21}, which is impossible. Therefore $z\in\overline B_{r/16}(x_0)$ and we have
$$
v(z)\leq\psi(z)= P(z;y_1)+ar^2\alpha^{-1}(16^\alpha-1)\,,
$$
thus proving \eqref{lemma2.20}.

To conclude the proof, consider for every $y\in \overline B_{r/64}(z)$ the paraboloid
$$
P(x;y_1)-C'\frac{a}{2}|x-y|^2+c_y\,,
$$
where $C'$ is a universal constant to be chosen and $c_y$ is such that the above paraboloid touches $v$ from below. Note that the above paraboloid has opening $(C'+1)a$ and vertex at
$$
\frac{y_1}{C'+1}+\frac{C'y}{C'+1}\,.
$$
Observe that
$$
P(z;y_1)-C'\frac{a}{2}|z-y|^2+c_y\leq v(z)\leq P(z;y_1)+Car^2\,,
$$
hence $c_y\leq Car^2+\frac{C'a}{2}\big(\frac{r}{64}\big)^2$.
On the other hand, if $|x-z|\geq\frac{r}{16}$, we have
\begin{align*}
P(x;y_1)-C'\frac{a}{2}|x-y|^2+c_y
& \leq P(x;y_1)-C'\frac{a}{2}\Big(\frac{3r}{64}\Big)^2+Car^2+\frac{C'a}{2}\Big(\frac{r}{64}\Big)^2 \\
& < P(x;y_1)\leq v(x)\,,
\end{align*}
provided that we choose $C'$ large enough independently of $a$ and $r$. Thus the contact point $x_y$ belongs to the ball $B_{r/16}(z)\subset B_{r/8}(x_0)$. Note that for $y\in \overline B_{r/64}(z)$ the vertex $\frac{y_1}{C'+1}+\frac{C'y}{C'+1}$ spans the ball with center $\frac{y_1}{C'+1}+\frac{C'z}{C'+1}$ and radius $C'r/64(C'+1)$, which is contained in $B_1$, provided that $C'$ is large enough. Moreover, the gradient of the function $x\mapsto P(x;y_1)-C'\frac{a}{2}|x-y|^2+c_y$ is smaller than $2(C'+1)a$ and 
$$
v(x_y)=P(x_y;y_1)-C'\frac{a}{2}|x_y-y|^2+c_y\leq a+2a+c_y\leq C''a 
$$
for a sufficiently large, universal, constant $C''>2(C'+1)a$. Therefore, by applying Lemma~\ref{lemma2.1} with $B_{r}(x_0)$ and $B$ replaced respectively by $B_{r/8}(x_0)$ and  the ball of radius $C'r/64(C'+1)$ and center $\frac{y_1}{C'+1}+\frac{C'z}{C'+1}$, with $a$ replaced by $C''a$,  we have that, if $3C''a\leq1$
$$
\big|A_{C''a}\cap B_{r/8}(x_0)\big|\geq c_0\Big(\frac{C'}{C'+1}\Big)^n\big|B_{r/64}|\,,
$$
from which \eqref{lemma2.2} follows with $C_1=3\max\{C'',16^{\alpha+1}+2\}$.
\end{proof}
We can now give the
\begin{proof}[Proof of Proposition~\ref{prop:savin}]
Let $x_0\in B_{1/2}$ be a point where $v(x_0)\leq\nu$.  Then the function $u(x)=v(x_0+x)$ is a positive viscosity $(C_0^k\nu)$-supersolution of \eqref{prop:savin1} in $\overline B_1$ with $u(0)\leq\nu$. Consider the paraboloid with vertex at the origin and opening $20\nu$ touching $u$ from below and observe that since $u$ is positive this paraboloid touches $u$ in $B_{1/3}$. Therefore $A_{20\nu}\cap\overline B_{1/3}\not=\emptyset$. Moreover, setting $D_0=A_{20\nu}\cap\overline B_{1/3}$, from Lemma~\ref{lemma2.2} we know that if $B_r(x)\subset B_1$, $B_{r/8}(x)\subset B_{1/3}$  and $D_0\cap B_r(x)\not=\emptyset$, then
$$
|(A_{20C_1\nu}\cap\overline B_{1/3})\cap B_{r/8}(x)|=|A_{20C_1\nu}\cap B_{r/8}(x)|\geq c_1|B_r|\,,
$$
provided $20C_1\nu\leq1$ and $v$, hence $u$, is a viscosity $(20C_1\nu)$-supersolution. By applying the same lemma to $D_1=A_{20C_1\nu}\cap\overline B_{1/3}$ and proceeding by induction, we have that the sets $D_0\subset D_1\subset\dots\subset D_k\subset B_{1/3}$, where $D_i=A_{20C_1^i\nu}\cap\overline B_{1/3}$ for $i=1,\dots,k$, have all the property that if $B_r(x)\subset B_1$, $B_{r/8}(x)\subset B_{1/3}$ and $D_i\cap B_r(x)\not=\emptyset$, for $i\leq k$, then
$$
|D_{i+1}\cap B_{r/8}(x)|\geq c_1|B_r|\,, 
$$
provided $20C_1^k\nu\leq1$ and $v$, hence $u$, is a viscosity $(20C_1^k\nu)$-supersolution. Then, using Lemma~2.3 in \cite{Savin} and setting $C_2=20C_1$, it follow that for a suitable $\mu$ depending only on $c_1$, hence universal,
$$
\big|\{x\in B_{1/3}(x_0):\,v(x)\leq C_2^k\nu\}\big|\leq(1-\mu^k)|B_{1/3}|\,.
$$
From this inequality the conclusion then follows by a standard rescaling and covering argument.
\end{proof}

\bibliographystyle{siam}
\bibliography{capillarity_obstacle19}

\end{document}